\documentclass{amsart}
\usepackage{amsmath,amssymb} 
\usepackage{amsbsy}
\usepackage{latexsym}
\usepackage{mathrsfs}
\usepackage{enumitem}
\usepackage{dsfont}
\usepackage{color}
\definecolor{red}{rgb}{.7,0,0}
\definecolor{purple}{rgb}{0.7,0,0.7}
\usepackage{comment}

\usepackage{hyperref}
\hypersetup{
    colorlinks=true,
    linkcolor=blue,
    citecolor=blue,
    filecolor=magenta,      
    urlcolor=cyan
    }
    \usepackage[capitalise]{cleveref}

\usepackage{url} % simple URL typesetting

\usepackage{tikz}
\usetikzlibrary{calc}

\newcommand{\N}{\mathbb N}
\newcommand{\R}{\mathbb{R}}
\newcommand{\C}{\mathbb{C}}
\newcommand{\Z}{\mathbb{Z}}
\newcommand{\Q}{\mathbb{Q}}
\newcommand{\F}{\mathbb{F}}
\newcommand{\proj}{\mathbb{P}}

\newcommand{\Gl}{\mathrm{GL}}
\newcommand{\GL}{\mathrm{GL}}
\newcommand{\Hom}{\mathrm{Hom}}

\newcommand{\E}{\mathbb{E}\,}

\newcommand{\spann}{\mathrm{span}}
\newcommand{\rk}{\mathrm{rk}}
\newcommand{\vol}{\mathrm{vol}}
\newcommand{\cM}{\mathcal{M}}

\newcommand{\codim}{\mathrm{codim}}

\renewcommand{\a}{\alpha}
\renewcommand{\b}{\beta}

\newcommand{\s}{\sigma}
\newcommand{\e}{\varepsilon}
\newcommand{\g}{\gamma}

\newcommand{\frm}{\mathfrak{m}}
\newcommand{\val}{\mathrm{val}}
\DeclareMathOperator{\chara}{char}
\newcommand{\Prob}{\mathrm{Prob}}

\newcommand{\diag}{\mathrm{diag}} 
\newcommand{\ot}{\otimes}
\newcommand{\Reg}{\mathrm{Reg}}
\DeclareMathOperator{\im}{im}
\newcommand{\Gr}{G}
\newcommand{\Gra}{\mathrm{Gr}}

\newcommand{\Hy}{\mathcal{H}} 
\newcommand{\RnO}{R_{\ne 0}}
\newcommand{\mnO}{\frm_{\ne 0}}

\newcommand{\bN}{\bar{\N}}

\newcommand{\bZ}{\bar{\Z}}
\newcommand{\Nko}{\bN^k_o}%{R^k_{\le}}
\newcommand{\cG}{\mathcal{G}}

\definecolor{ao(english)}{rgb}{0.0, 0.5, 0.0}

\DeclareMathOperator{\Spec}{Spec}
\renewcommand{\:}{\colon}

\newtheorem{thm}{Theorem}
\newtheorem{lemma}[thm]{Lemma}
\newtheorem{cor}[thm]{Corollary}
\newtheorem{prop}[thm]{Proposition}
\newtheorem*{prop*}{Proposition}
\newtheorem*{claim}{Claim}

\theoremstyle{definition}
\newtheorem{defi}[thm]{Definition}

\theoremstyle{remark} 
\newtheorem{remark}[thm]{Remark}
\newtheorem{example}[thm]{Example}

\numberwithin{equation}{subsection}
\numberwithin{thm}{subsection}
\title{Nonarchimedean integral geometry} 
\author{Peter B\"urgisser}\thanks{Supported by the ERC under the European Union's Horizon 2020 research and innovation programme (grant agreement no. 787840).}
\address{Institute of Mathematics, Technische Universit\"at Berlin}
\email{pbuerg@math.tu-berlin.de}
\author{Avinash Kulkarni}\thanks{Supported by the Simons Collaboration on Arithmetic Geometry, Number Theory, and Computation (Simons Foundation grant 550029).}
\address{Department of Mathematics, 
Dartmouth College, USA} 
\email{avinash.a.kulkarni@dartmouth.edu}
\author{Antonio Lerario}
\address{SISSA, Trieste, Italy}
\email{lerario@sissa.it}

\begin{document}

\keywords{integral geometry, nonarchimedean analysis, coarea formula, Schubert varieties, Schubert calculus, 
fewnomial systems}

\subjclass[2010]{11S05, 14G20, 14N15, 32P05, 53C65, 60D05}

% 11S05: Algebraic number theory: local fields -> polynomials
% 14G20: Local ground fields in algebraic geometry 
% 14N15: Schubert calculus, classical problems
% 32P05: nonarchimedean analysis 
% 53C65: integral geometry
% 60D05: geometric probability and stochastic geometry

\maketitle

% \begin{abstract}
%   Let $K$ be a nonarchimedean local field of characteristic zero with
%   valuation ring~$R$, for instance, $K=\mathbb{Q}_p$ and $R=\mathbb{Z}_p$. 
%   We introduce the concept of an $R$--structure on a 
%   $K$--analytic manifold, which is akin to a Riemannian metric
%   on a smooth manifold. 
%   Using this notion, on the tangent spaces of a $K$--analytic manifold we define in a natural way a 
%   norm and volume form, which allows to integrate 
%   real valued functions on the manifold, and we prove a nonarchimedean version of Sard's lemma 
%   and the coarea formula. 
%   We also consider $R$--structures on 
%   $K$--analytic groups and homogeneous $K$--analytic spaces, which  
%   are compatible with the group operations.
%   In this framework, we prove a general integral geometric formula, which is 
%   analogous to the corresponding result over the reals. 
%   This generalizes the $p$--adic integral geometric formula for 
%   projective spaces recently discovered by Kulkarni and Lerario, 
%   e.g., to the setting of Grassmannians. 
%   As a first application we compute the volume of special Schubert varieties, and 
%   we outline the construction of a
%   probabilistic nonarchimedean Schubert Calculus. As a second application we bound 
%   (and exactly determine in some cases) 
%   the expected number of zeros in $K$ of random fewnomial systems.
% \end{abstract}

\begin{abstract}
  Let $K$ be a nonarchimedean local field of characteristic zero with
  valuation ring~$R$, for instance, $K=\mathbb{Q}_p$ and $R=\mathbb{Z}_p$. 
  We prove a general integral geometric formula for  
  $K$--analytic groups and homogeneous $K$--analytic spaces,
  analogous to the corresponding result over the reals.  
  This generalizes the $p$--adic integral geometric formula for 
  projective spaces recently discovered by Kulkarni and Lerario, 
  e.g., to the setting of Grassmannians. 
  {Based on this, we outline the construction of a  nonarchimedean probabilistic Schubert Calculus.
  For this purpose, we characterize the relative position of two subspaces of $K^n$ by a position vector, 
  a nonarchimedean analogue of the notion of principal angles,
  and we study the probability distribution of the position vector for random uniform subspaces. 
  We then use this to compute the volume of special Schubert varieties over $K$.} 
  As a second application of the general integral geometry formula, 
  we initiate the study of random fewnomial systems over nonarchimedean fields, bounding,
  and in some cases exactly determining,
  the expected number of zeros of such random systems. 
\end{abstract}

\setcounter{tocdepth}{2}
\tableofcontents

\section{Introduction}

\subsection{Nonarchimedean integral geometry}

%The classical subject of integral geometry has a vast range of applications and connections to different areas of mathematics,
%including differential geometry, representation theory,
%convex geometry, numerical analysis, and random geometry.
Classical integral geometry deals with the averaging of metric properties of submanifolds of a homogeneous space under
the action of a (real or complex) Lie group. A cornerstone result of this theory is the integral geometry formula in Riemannian homogeneous spaces, 
which goes back to the work of Poincar\'e and Blaschke, proved in modern language in the monograph by Howard  \cite{howard:93} (see also \cite{santalo}). 
In the classical literature (i.e., in the real or complex setting), formulas that express metric properties of a set
as an average over a  family of \emph{linear} slices are also called Crofton formula. 
In the nonarchimedean setting a local Crofton formula in the spirit of \cite{comte} was proved in \cite[Theorem 6.2.1]{CCL}, 
and a motivic version of this formula appears in \cite{forey}. In~\cite{KL:19}, an analogue of the Riemannian integral geometry from~\cite{howard:93} 
was established for submanifolds of projective spaces over~$\Q_p$. 
    
One of the main result of this article, \cref{thm:main}, is a generalization of \cite{KL:19} over a general nonarchimedean local field~$K$ of characteristic zero
and for general homogeneous spaces.

Our motivation for this study is related to the subject of \emph{random algebraic geometry}, 
which has attracted a lot of attention over the last decades. 
When no \emph{generic} number of solutions to a system of algebraic equations is defined, 
which happens to be the case when working over a non algebraically closed field, 
one seeks for a typical behaviour by insisting that the coefficients of the equations are sampled at random from some probability distribution, 
then asking for the \emph{expectation} of the number of solutions. Over the Reals, this approach goes back to Kac and Rice and 
has been investigated, for instance, in 
%\cite{EdelmanKostlan95, shsm, EKS, Ko2000, ShSm3, ShSm1, GaWe1, GaWe3, GaWe2, NazarovSodin1, Sarnak, Letwo, LeLu:gap, ET19,E21,EK22}. 
\cite{ShSm1, shsm, ShSm3, EdelmanKostlan95, NazarovSodin1, Ko2000, GaWe1, GaWe2, GaWe3, Letwo, LeLu:gap,Sarnak,ET19,E21,EK22}. 
Over nonarchimedean fields the study of random system of equations was initiated by Evans \cite{Evans} 
and recently received new interest, see \cite{evans:02,KL:19,bhar-et-al:21, AEML, Caruso22, shmueli}. 
(These lists of references are by no mean complete.)

The role of the integral geometry formula in this subject has been clear already from \cite{EdelmanKostlan95}, 
where the authors relate the expectation of the number of solutions of a random polynomial system 
with the volume of the image of an appropriate Veronese embedding -- by slicing this embedding with a random uniformly distributed linear space. 
In \cite{BL:19} B\"urgisser and Lerario explained how to use the integral geometry approach to construct a probabilistic version of Schubert Calculus, 
leading to a theory of random intersections over the Reals (even with an underlying ``probabilistic intersection ring'', see \cite{BBLM, Leothesis}).

In the nonarchimedean framework, the ideas from \cite{EdelmanKostlan95} 
can be extended to the study of systems of random $p$--adic equations using the $p$--adic  integral geometry formula
% AK: Split this sentence in two.
% proved by Kulkarni and Lerario in \cite{KL:19}, but to go beyond the basic projective framework one needs
{ proved by Kulkarni and Lerario in \cite{KL:19}. To go beyond the basic projective framework, one needs} 
the more general formula that we prove in this paper. In fact, as a consequence of \cref{thm:main}, 
in this work we will also explain how to set up a nonarchimedean version of the probabilistic Schubert Calculus from \cite{BL:19} (see \cref{sec:PSCintro}). 
As we will see, this will require a detour into the study of the relative position of a pair of random subspaces in the nonarchimedan Grassmannian, 
as well as the study of metric properties of nonarchimedean Schubert varieties, see 
\cref{sec:randomspacesintro} and \cref{sec:volumesintro}, respectively. 
Moreover, we will also give here a further application of the integral geometry approach, 
initiating the study of random system of fewnomial equations over a nonarchimedan field, see \cref{sec:fewnomials}.

\subsection{$R$--structure on $K$--analytic manifolds}

In order to formulate the main results, %in this section 
we have to introduce first some notation and terminology.
 Throughout the paper, $K$ denotes a nonarchimedean local field of characteristic zero with
valuation ring~$R$, maximal ideal $\frm$, and residue field 
$R/\frm \simeq\F_q$, where $q$ is a prime power. For instance, if
$K=\Q_p$, $R=\Z_p$ and $q=p$. 

% In our exposition we will use the concept of \emph{$R$--structure},
% a new notion which encodes 
% the properties of well known objects in the literature in a more analytic framework, which seems quite flexible for our purposes.
% This notion is inspired by \cite[Chapter~AG~11]{Borel} and \comm{can be seen as} a nonarchimedean analogue of a Finsler structure. 

% Loosely speaking, an $R$--structure on a $K$--analytic manifold $X$ is a locally constant choice 
% (with respect to $x\in X$) of a lattice $\Lambda_x\subseteq T_xX$, spanning $T_xX$ over $R$ (see \cref{def:R-struct-mfs}). 
% With this idea we tried to single out the relevant properties from the classical literature on the subject (see for instance \cite{igusa:2000, CLT, AA}) 
% in a way that makes it closer to \cite{howard:93, BL:19} and ready to use also for nonspecialists.

% XXX

In our exposition we will use the concept of an \emph{$R$--structure}.
This notion is inspired by \cite[Chapter~AG~11]{Borel} and can be seen as a nonarchimedean analogue of a Finsler structure.
Loosely speaking, an $R$--structure on a $K$--analytic manifold $X$ is a locally constant choice 
(with respect to $x\in X$) of a lattice $\Lambda_x\subseteq T_xX$, spanning $T_xX$ over $R$ (see \cref{def:R-struct-mfs}). 
The notion of an $R$--structure is very closely related to several concepts already
well studied in the literature, see for instance \cite{igusa:2000, CLT, AA}.
Whilst the introduction of $R$--structures may be redundant from an expert or specialist point of
view, our approach is meant to serve two aims: first, we wish to consider classically algebraic
or geometric concepts within a more analytic framework. Secondly, and more importantly,
we wished to streamline and single out the germane concepts in the existing theory
in a way that makes the analogies with \cite{howard:93, BL:19} transparent, particularly for
nonspecialists.

The notion of an $R$--structure endows a $K$--analytic manifold $X$ with a natural volume form~$\Omega_X$.
Using the volume form, we can integrate measurable functions $f\: X \rightarrow \C$
and, if the constant function $f\equiv 1$ is integrable, we can define the \emph{volume} of $X$ as 
\[
  |X| := \int_X \Omega_X.
\]
%{\color{blue}
For $U\subseteq X$ a Borel set, setting $\mathrm{vol}(U):=\int_U \Omega_X$ defines a Borel measure on $X$ that,
in the case $X\subseteq \Z_p^n$ is an analytic subset, agrees with the standard notion of Hausdorff measure on the
set of smooth points, as in \cite{oesterle:82, CLT, AA}.  See also Sections~\ref{se:volumeX} and~\ref{sec:volumeequi}.
%}
Moreover, if $X$ is a smooth scheme over $\Spec R$, then the set of $R$--valued points $X(R)$ is a $K$--analytic manifold
that carries a natural $R$--structure.
It turns out that 
the volume form on $X(R)$ determined by this $R$--structure
is a gauge form in the sense of \cite{weilAdeles}; in particular, the volume of this manifold
in our sense agrees with the Weil canonical volume (Proposition \ref{prop:weil}), see \cite{serre:81, weilAdeles}.
The notion of $R$--structure allows also to formulate geometric--measure statements (involving normal Jacobians) in a simple way, 
see for instance \cref{th:coarea}.

%The concept of an $R$--structure allows us to define a nonarchimedean version of
%the {\em absolute Jacobian} $J(\varphi)$ of a submersive map
%$\varphi\colon X \rightarrow Y$ of $K$--analytic manifolds and 
%to prove a nonarchimedean version of the coarea formula
%(see Theorem~\ref{th:coarea}).
%
%\AK{This seems to be in the paper [AA, Proposition 3.3.1]. However, the result is clearly
%not claimed there as new -- it appears to be folklore.}
%
%\begin{thm}[Nonarchimedean coarea formula]\label{thm:intro:coarea}
%Let $X$ and $Y$ be $K$--analytic manifolds  of dimensions~$n\ge m$, 
%endowed with $R$--structures, 
%and let $\varphi\colon X\to Y$ be a $K$--analytic map.
%Further, assume 
%$h\colon X\to [0,\infty]$ is measurable. Then
%\begin{equation*}\label{eq:coco}
%  \int_X h\, J(\varphi)\, \Omega_X 
%  = \int_{y\in Y} \Big(\int_{\varphi^{-1}(y)} h\, \Omega_{\varphi^{-1}(y)} \Big) \Omega_Y,
%\end{equation*}
%where each $\varphi^{-1}(y)$ carries the $R$--structure induced from $X$. 
%\end{thm}

 \subsection{The nonarchimedean integral geometry formula}
In a similar style as in Howard's~\cite{howard:93} generalization of Poincar\'e's classical integral geometry formula, 
we consider the action of a general compact $K$--analytic group on a homogeneous space.

The general setting is as follows: 
let $G$ be a compact (nonarchimedean) $K$--analytic Lie group and $H$ a closed $K$--analytic subgroup. 
Then the homogeneous space $G/H$ becomes a $K$--analytic manifold~\cite[Part~II, Chap.~IV, \S 5]{serre:64}. 
If $G$ comes with a compatible $R$--structure, then we can endow $G/H$ with an induced $R$--structure 
(see \cref{se:anKgroups} and \cref{se:hom_K_analytic_spaces}). 
Assume that $Y_1, \ldots, Y_m$ are $K$--analytic submanifolds of $G/H$.  Similar to the classical integral geometry formula in~\cite{howard:93},
the nonarchimedean statement requires the notion of an \emph{average scaling factor}.
The average scaling factor is a function $\sigma_H\colon Y_1 \times \cdots \times Y_m \rightarrow \R$ that measures the relative positions
of the tangent spaces of the $Y_i$ (see Definition~\ref{def:sigma-many}). 
For example, when $G = \operatorname{GL}_{n+1}(\Z_p)$ and $H$ is the stabilizer of a non-zero vector,
then $\sigma_H$ is a constant function (which we compute explicitly in~\eqref{eq:sigmaH-pro}). 
With this notation, the nonarchimedean version of the integral geometry formula takes the following form
(see Theorem~\ref{thm:GIGF}).

\begin{thm}[Nonarchimedean integral geometry formula]\label{thm:main} 
Let $G$ denote a compact $K$--analytic group with $R$--structure  
and $H\subseteq G$ be a closed $K$--analytic subgroup. 
Let $Y_1, \ldots, Y_m$ be $K$--analytic submanifolds of $G/H$ 
such that $\sum_{i=1}^m\codim_{G/H} Y_i \leq \dim G/H$. 
Then the submanifolds $g_1Y_1, \ldots, g_mY_m$ intersect transversely
for almost all $(g_1, \ldots, g_m)\in G^{m}$, and
$$
 \E_{(g_1, \ldots, g_m)\in G^m} |g_1Y_1\cap \cdots\cap g_m Y_m|
   =\frac{1}{|G/H|^{m-1}} \int_{Y_1\times \cdots\times Y_m}\sigma_H\, \Omega_{Y_1\times\cdots \times Y_m} , 
$$
where the expectation is taken over a uniformly random 
$(g_1, \ldots, g_m)\in G\times \cdots\times G$
defined with respect to the Haar measure. 
\end{thm}

% {\color{blue}
%   \begin{remark}
%     In the classical literature (i.e., in the real or complex setting) results like Theorem~\ref{thm:main},
%     where one averages metric properties of a set over a family of slices, generalize to the so-called Crofton formula,
%     where the set to be sliced sits in the affine space and the slices are linear sections (so integration is over the
%     affine grassmannian). In the nonarchimedean setting a local Crofton formula in the spirit of \cite{comte} was proved in \cite[Theorem 6.2.1]{CCL}.
%   \end{remark}
% }

A special situation of interest for us, in view of the application to Schubert Calculus, is when  in Theorem~\ref{thm:main}
the group $G$ acts transitively on the tangent spaces of each submanifold $Y_i$
(see Definition~\ref{def:TA}). In this case we get the following result
(see Corollary~\ref{cor:GIGF}). 

\begin{cor}\label{cor:main}
Under the assumptions of Theorem~\ref{thm:main}, 
if moreover $G$ acts transitively on the tangent spaces to $Y_i$ of $i=1, \ldots, m$, 
then we have 
\begin{equation}\label{eq:cohomogeneous}
  \E_{(g_1, \ldots, g_m)\in G^m}|g_1Y_1\cap \cdots\cap g_m Y_m|
     =\sigma_H(y_1, \ldots, y_m)\cdot |G/H| \cdot \prod_{i=1}^m\frac{|Y_i|}{|G/H|} ,
\end{equation}
where $(y_1, \ldots, y_m)$ is any point of $Y_1\times \cdots\times Y_m$.
\end{cor}

The main result of \cite{KL:19} is obtained from Corollary~\ref{cor:main} by specializing
$K = \Q_p$, $G = \operatorname{GL}_{n+1}(\Z_p)$, and $H \subseteq G$ being
the subgroup stabilizing a one-dimensional subspace of $K^{n+1}$.

\subsection{Toward a nonarchimedean probabilistic Schubert Calculus}\label{sec:PSCintro}

Schubert Calculus is, loosely speaking, the study of the ways that Schubert cycles intersect within Grassmannians.
  Over an algebraically closed field, the enumerative properties of intersections of Schubert cycles,
  such as the number of points in a zero--dimensional intersection, remain constant outside a locus of positive codimension.
  This principle is no longer true if one is looking at Grassmannians defined
  over non--algebraically closed fields. 
  In a previous work~\cite{BL:19}, B\"urgisser and Lerario introduced a probabilistic version of Schubert Calculus
  to study the typical behaviour of intersections of Schubert cycles. 
  More specifically, their work focused on the expectation of the number of points (or volume)  
  of an intersection of randomly chosen Schubert varieties in real Grassmannians.
  Motivated by the real case,
  we consider the study of Schubert cycles over a nonarchimedean local field.
\begin{example}
  Consider the elementary enumerative problem of counting the number of lines intersecting 
   four lines $L_1, \ldots, L_4$
  in $p$--adic projective space.
  For the \emph{generic} choice of the lines $L_1, \ldots, L_4$, there are precisely two lines $S_1, S_2$ 
   intersecting all of them, 
  but $S_1, S_2$ are in general defined over a degree--two extension of $\mathbb{Q}_p$. 
  The number of $p$--adic solutions of this enumerative problem depends on
   the configuration of the given four lines. 
  One can replace the word \emph{generic} with \emph{random} and ask for the expectation of 
  the number of solutions 
  defined over $\mathbb{Q}_p$. More precisely, let 
  $G=\mathrm{GL}_4(\mathbb{Z}_p)$, $H=\mathrm{GL}_2(\mathbb{Z}_p)\times \mathrm{GL}_2(\mathbb{Z}_p)$, 
 so that $G/H$ is the Grassmannian of planes in $\mathbb{Q}_p^4$, 
 i.e., the Grassmannian  of  lines in projective three--space. 
 Denoting by $\Omega\subseteq G/H$ the set of lines meeting a fixed line, asking for 
 ``the expected number of lines intersecting four random lines in $p$--adic projective space space'' 
 means computing the expectation
\begin{equation}\label{eq:integralgrass}
   \E_{(g_1, \ldots, g_4)\in G^4}\#(g_1\Omega\cap\cdots \cap g_4\Omega).
\end{equation}
\end{example}

Integrals like \eqref{eq:integralgrass}, where each $\Omega$ is replaced by a general Schubert variety, 
can be thought of as average intersection numbers in nonarchimedean Grassmannians. 
In principle, these numbers can be computed using Corollary~\ref{cor:main} (once the hypothesis of the statement 
are verified). 
A crucial observation is that $\mathrm{GL}_n(R)$ acts transitively on the set of tangent spaces 
to Schubert varieties---we provide a proof of this statement for
codimension one
Schubert varieties, 
but remark that a similar argument works in general case. In order to compute integrals as \eqref{eq:integralgrass},
one requires therefore to know the volumes of the Schubert varieties as well as the value of the average scaling factor $\alpha_H$ on them 
(which happens now to be a constant function because of the transitive action of $\mathrm{GL}_n(R)$ on their tangent spaces). 
We discuss now these two points.

\subsubsection{The volume of special Schubert varieties}\label{sec:volumesintro}
Computing the volume of Schubert varieties seems to be a quite
difficult task, already over the Reals (over the complex numbers it is
easier, because the volume can be read from the degree). In the real
case this was done in \cite{BL:19} for codimension one Schubert
varieties. In this paper we compute volumes of \emph{all} special
Schubert varieties in the nonarchimedean framework.\footnote{By essentially the same method, we also obtain the volume of all special Schubert varieties over the Reals, which is a new result.}
In order to state the result, it turns out to be helpful to use a more symmetric notation for Grassmannians: 
let us define for $a,b\in\N$
$$
 \Gra(a,b) := G(a,a+b) := \{ E\subseteq K^{a+b} \mbox{ subspace }\mid \dim E = a\}.
$$
We will consider the special Schubert variety
$$
 \Omega_{a,a_1;b,b_1} := \{ E \in \Gra(a+a_1,b+b_1) \mid \dim  (E\cap F) \ge a\},
$$
where $F\in\Gra(a+b_1, a_1 +b)$ is fixed. Note that $\Omega_{a,a_1;b,b_1}$ has codimension $ab$ in $\Gra(a+a_1,b+b_1)$. 
(In fact, this Schubert variety 
corresponds to the rectangular Young diagram $a\times b$, 
see~\cite{eisenbud-harris:16, manivel:01}.)

With this notation, we will prove the following theorem (this is \cref{th:main-vol-schubert} below).

\begin{thm}\label{th:main-vol-schubertintro} 
For any nonnegative integers $a,a_1,b,b_1$, we have 
\begin{equation*}
  \frac{|\Omega_{a,a_1;b,b_1}|}{|G(a+a_1,b+b_1)|} = \frac{
   |\Gra(a,b)| \cdot   |\Gra(a,b_1)| \cdot  |\Gra(a_1,b)|}{
   |\Gra(a+a_1,b)| \cdot  |\Gra(b+b_1,a)|}.
\end{equation*}
\end{thm}

For the codimension one case (i.e., for Schubert varieties corresponding to a $1\times 1$ Young diagram)  this specializes as follows. 

\begin{cor}\label{cor:dimvol-codim-1intro}
Consider the codimension one special Schubert variety 
$\Omega := \Omega_{1, k-1;1, n-1}$ of $G(k,n)$ 
consisting of the $k$--planes nontrivially meeting an $(n-k)$--plane. 
The volume of $\Omega$ is given by:
$$
\frac{|\Omega|}{|G(k,n)|} 
 = |\proj^1| \cdot \frac{|\proj^{k-1}|}{|\proj^{k}|} \cdot \frac{|\proj^{n-k-1}|}{|\proj^{n-k}|} . 
$$
\end{cor}

(The volume of the projective space is computed in \cref{example:volproj}.) Remarkably, the formulas in Corollary~\ref{cor:dimvol-codim-1intro} 
are {\em identical} to the ones over $\R$ and over $\C$, which were 
obtained in~\cite{BL:19}. 

\subsubsection{Probabilistic Schubert Calculus}

We do not make an attempt at the most general construction of 
a theory of nonarchimedean probabilistic Schubert Calculus -- instead, we leave such an investigation as the subject of future work.
Here we only focus on illustrating our methods for computing the numbers \eqref{eq:integralgrass}, 
where $\Omega\subseteq G(k,n)$ is a codimension one Schubert variety. More precisely, we consider the following specific problem, which we call ``the simplest Schubert problem'':
\begin{quote}
\emph{What is the expected number of $k$--planes in $K^n$ nontrivially 
intersecting $k(n-k)$ random independent subspaces $L_1, \ldots, L_{k(n-k)}\subseteq K^n$ of dimension $n-k$?}
\end{quote}

We denote the value of the above expectation by $\eta_{k,n}(K)$. For instance, the integral in \eqref{eq:integralgrass} equals $\eta_{2,4}(\Q_p)$. An analogue quantity can be defined over $\R$ or $\C$ and in \cite{BL:19} the authors prove that
\begin{equation}\eta_{2, n+1}(\R)=\frac{8}{3 \pi^{5/2}\sqrt{n}}\left(\frac{\pi^2}{4}\right)^n(1+O(n^{-1})),\quad \textrm{as $n\to \infty$}.
\end{equation}
The asymptotics for these numbers for general $k\in \N$ fixed and $n\to \infty$, both over  $\R$ and over $\C$, are computed in \cite{LMasymptotic}. Here we prove the following result, which considers the asymptotic value of $\eta_{k,n}(K)$ as $q$ (the cardinality of the residue field) goes to infinity (see Corollary~\ref{cor:k-plane-expectation-asyptotic}). This result is a corollary of the asymptotic for the average scaling factor for special Schubert varieties (\cref{prop:alpha}) and  should be compared with similar results obtained in random $p$--adic enumerative geometry, 
for instance in \cite{KL:19} and \cite[Theorem 1]{el-manssour-lerario:20}.

\begin{cor}
As $q\to \infty$, the expected number of $k$-planes in $K^n$ intersecting 
$k(n-k)$ random independent subspaces 
$L_1, \ldots, L_{k(n-k)}\subseteq K^n$ of dimension $n-k$ converges to $1$, i.e.
$$\lim_{q\to \infty}\eta_{k,n}(K)=1.$$ 
\end{cor}

\subsection{The relative position of two random subspaces}\label{sec:randomspacesintro} 

The proof of Theorem \ref{th:main-vol-schubertintro}
proceeds along an avenue
of independent interest, namely, the study of the relative position of two subspaces in $K^n$, 
see Section~\ref{se:rel-pos}. %{se:grassmannian}. 
This is done in terms of a notion which we call the \emph{(relative) position vector}, 
encoding the nonarchimedean analogue of the principal angles between two subspaces. 
Let $\overline{\N}:=\N\cup\{\infty\}$ and denote  
$\Nko := \{ x\in\bN^k \mid x_1 \le \ldots \le x_ k\}$.  
The relative position of two subspaces $E, F\subseteq V$ (here $V\simeq K^n$), of dimensions $k$ and $\ell$, 
respectively, and with $k\le \ell$,  will be an increasing list of nonnegative integers 
$(x_1,\ldots,x_k)\in \Nko$, see Definition \ref{def:PP1}. This list of
numbers completely characterizes  
the pair $(E, F)$ up to to the action of $\mathrm{GL}_n(R)$, 
as we show in \cref{cor:RPP-Char} (see also \cref{th:NF-pairs-subspaces}). 

Moving to the random side, we will be interested in describing the position vector of independent 
uniformly random subspaces $E\in G(k,n)$ and $F\in G(\ell,n)$. To this end, with loss of generality, 
we may assume $k\le\ell$ and fix $F := K^\ell\times 0$. 
%and put $F := K^\ell\times 0$. 
Consider the map
$$
 \chi\colon G(k,n) \to \Nko, \, E \mapsto x , 
$$
which sends $E$ to the position vector~$x$ of the pair $(E,F)$. 
We denote by $\rho_{k,\ell, n}$ 
the pushforward of the uniform probability distribution on $G(k,n)$. 
This is a discrete distribution on $\Nko$ that 
describes the joint distribution of the position vector of $E$ and $F$. %$\in G(k,n)$ and $F\in G(\ell,n)$. 
In \cref{th:joint-densityK} we provide an explicit formula for this distribution, 
which is a nonarchimedean version of~\cite[Theorem 3.2]{BL:19}.
Our formula involves an explicit parametrization of the various open
sets in the Grassmannian where the position vector is constant, and
requires the computation of the normal Jacobian of this
parametrization, see Lemma~\ref{le:Jac-psi}. 
Our results are closely related to $p$-adic random matrix theory, see \cite{vanpeski:21}. 
The corresponding classical results are of relevance in multivariate statistics~\cite{Muirhead}. 
Interesingly, unlike the classical situation, a repeated entry in $x$ does not force the 
vanishing of~$\rho_{k,\ell, n}(x)$.  
The reader may consult~\cref{re:psi-RC} for comparing our results with the ones  over~$\R$ and $\C$.

\subsection{Random systems of equations}\label{sec:fewnomials}

As we already mentioned above, 
Evans~\cite{evans:06} was the first to investigate the expected number of zeros 
in the field of $p$--adics. In a manner analogous to the classical case, it appears that nonarchimedean integral geometry 
lends insight into the study of these questions, see~\cite{KL:19,el-manssour-lerario:20}.
For instance, Evans' results~\cite{evans:06} were extended in~\cite{KL:19} using the $p$--adic projective integral geometry formula 
combined with the computation of the length of the image an appropriate Veronese map.

The recent studies of random systems of $p$-adic
equations~\cite{evans:02,caruso:20,bhar-et-al:21, el-manssour-lerario:20} only deal with random systems of polynomial equations 
whose support is dense, i.e., polynomials of a certain degree in which all monomial terms appear.
In Section~\ref{se:fewnom}, for the first time,
we study the average number of random fewnomial systems 
in a nonarchimedean local field.
This is inspired by the recent papers~\cite{BETC:19,PB:23, ETTC:23},
which bound the expected number of real zeros of fewnomial systems.
%We achieve this by applying our nonarchimedean coarea and integral geometry formulas.

The following is a consequence of Theorem~\ref{thm:n=1R}. 
Recall that $\e = q^{-1}$ denotes the reciprocal of the cardinality $q$ of the residue field of $R$. 

\begin{cor}\label{cor:uno}
Fix integers $a_1 < a_2 < \ldots < a_t$ and consider the random fewnomial 
$f := c_1 x^{a_1} + c_2 x^{a_2} + \ldots +  c_{t} x^{a_t}$ 
with i.i.d.\ uniformly distributed coefficients~$c_{j}$ in $R$.
Then the expected number  of zeros of $f$ in 
$R\setminus\{0\}$
is at most $(1+\e)^{-1}$; 
equality holds iff $a_2-a_1= 1$. 
Moreover, the expected number  of zeros of $f$ in $K^\times$ is at most one; 
equality holds iff $a_2 - a_1= a_t - a_{t-1}=1$. 
\end{cor}

\begin{example}
The random fewnomial 
$$
 f(x) = c_5 x^5 + c_6 x^6 + c_{100}x^{100} + c_{999}x^{999} + c_{1001}x^{1001} \in \Z_3[x]
$$
on average has exactly one zero in $\Z_3\setminus \{0\}$ 
and at most one zero in $\Q_3^\times$. 
In fact, on average, $f$ has exactly one zero in $\Q_3^\times$: this follows from 
a refined version of Corollary~\ref{cor:uno}, stated as Theorem~\ref{thm:n=1R}.
\end{example}

We can generalize these results from sparse univariate polynomials to systems as follows. 
Let us fix a finite subset $A\subseteq\Z^n$.  
The {\em Newton polytope} $P$ of $A$ is defined as its convex hull.
We call $P$ {\em rectangular} if it is of the form 
$$
 P =\{\ell_1,\ell_1+1,\ldots,r_1\} \times\ldots\times \{\ell_n,\ell_n+1,\ldots,r_n\} ,
$$ 
for some integers $\ell_i < r_i$. 
We call $A$ \emph{gap-free rectangular at the vertex $\ell=(\ell_1,\ldots,\ell_n)$} of $P$ iff 
$A$ contains the $n$  neighbours $\ell + e_1,\ldots,\ell+e_n$ at distance~$1$ of $\ell$ 
within $P \cap \Z^n$. 
(Here $e_i$ stands for the $i$th standard basis vector.) 
A similar definition applies to all the vertices of $P$. 

We now consider a random system 
$f_1(x)=0,\ldots,f_n(x)=0$, where 
$$
 f_i(x) = \sum_{a\in A} c^{(i)}_{a} x_1^{a_1}\cdots x_n^{a_n},
$$
with i.i.d.\ uniformly distributed coefficients~$c^{(i)}_{a}$ in $R$ 
and denote by $N_U(A)$ its number of nondegenerate zeros in a subset $U\subseteq(K^\times)^n$ .
The following result generalizes Corollary~\ref{cor:uno}. 
(A more detailed statement appears in Proposition~\ref{pro:rectN}.) 
We abbreviate $\RnO^n := (R\setminus\{0\})^n$ for convenience of notation. 

\begin{cor}
We have $\E(N_{\RnO^n}(A)) \le 1$; equality holds iff 
$A$ is gap-free rectangular at $\ell$. Moreover,  
$$
 \E(N_{(K^\times)^n}(A)) \ \le\ \frac{(1-\e) (1+\e)^n}{1-\e^{n+1}} ;
$$
equality holds iff $A$ is gap-free rectangular at all its vertices. 
\end{cor}

\begin{example}
For example, the support 
$A=\{0,1,99,100\}^2\subseteq\N^2$
is gap-free rectangular. 
Note that $0$ is almost surely not a zero of 
a random system with support $A$, since $0\in A$.  We thus obtain 
$\E(N_{\Z_p^2}(A)) = 1$ and 
$\E(N_{\Q_p^2}(A)) = \frac{p^2+2p+1}{p^2+p+1}$.
\end{example}

The proofs of these result use the integral geometry formula for relating the expectation of the number of solutions to the fewnomial system 
to the volume of the image of an appropriate Veronese type map.

Finally, let us mention that the study of the zeros of polynomial systems with few terms %(fewnomials) 
is a thoroughly investigated topic, closely related to toric varieties. A major result in this area 
is the Bernstein--Khovanskii--Kouchnnirenko Theorem, which expresses the generic number of zeros of a fewnomial 
system in terms of the mixed volumes of the Newton polytopes of the given supports, see~\cite{bernstein:75,BKK:76}. 
These works study the number of zeros in an algebraically closed field.
There is also a well developed theory concerned with counting the real zeros of such systems, 
for which we just refer to~\cite{khovanskii:91,sottile:11}. 
Much less is known in the nonarchimedean setting; we refer to~\cite{lenstra:97,poonen:98,rojas:04}.

\subsection*{Acknowledgments}
Peter B\"urgisser was supported by the ERC under the European Union's Horizon 2020 research and innovation programme (grant agreement no. 787840). 
Avinash Kulkarni was supported by the Simons Collaboration on Arithmetic Geometry, Number Theory, and Computation (Simons Foundation grant 550029). Antonio Lerario was supported by the Alice and Knuth Wallenberg foundation.

\section{Preliminaries}\label{se:basics-I}
We recall in this section now some known facts and establish some notation. \subsection{Nonarchimedean local fields}\label{se:na-fields}

We work in the general setting of a nonarchimedean local field $K$, 
see~\cite{lorenzII} for an introduction. 
The nonarchimedean absolute value of $K$ will be denoted by $|\ |_K$ and the standard
absolute value on $\mathbb{R}$ will be denoted by $| \ |_{\mathbb{R}}$; in this article we
are almost exclusively concerned with $| \ |_K$, and so we also denote the nonarchimedean
absolute value by $| \ |$ if there is no danger of confusion.
 Such a field~$K$ is complete with respect to the metric $d(x,y):=|x-y|$.
 We denote by $R :=\{x\in K \mid |x| \le 1\}$ the ring of integers of~$K$,
 which is known to be a discrete valuation ring
 with maximal ideal $\frm :=\{x\in K \mid |x| <1 \}$.
 Note that the group of units of $R$ is characterized by 
 $R^\times := \{x\in K \mid |x| = 1\}$.  
 We shall fix a generator $\varpi$ for the maximal ideal $\frm$, i.e., we
 fix a {\em uniformizer} for $K$. 
 The {\em residue field} of $K$ is defined as $R/\frm$, which is a finite field,
 say of of cardinality $q$. Then the residue ring 
 $R/\frm^k$ has cardinality $q^k$, where $k\in\N$.
 We shall denote $\e := q^{-1}$, and we normalize $|\ |_K$ so that $|\varpi|_K = \epsilon$.
Moreover, we normalize the \emph{valuation} $\val$ so that $|x|_K = q^{-\val(x)} =  \e^{\val(x)}$.

The typical example of a nonarchimedean local field is the field of $p$--adic numbers $\Q_p$ 
with the $p$--adic norm $|\ |_p$.
The ring of integers of $\Q_p$ is denoted by $\Z_p$, and $p$ is a uniformizer for $\Q_p$.
If $K/\Q_p$ is a finite extension, then the absolute value $|\;|_p$ extends uniquely to an absolute value of $K$.
If $n = [K:\Q_p]$, then with our chosen normalizations, we have
that $|x|_K = |x|_p^n$ for all $x \in \Q_p$. 

The classification of nonarchimedean local fields $K$ is well known, 
e.g., see~\cite{lorenzII}.  
If $\chara K=0$, then $K$ is a finite extension of 
a field~$\Q_p$ of $p$--adic numbers. 
In the case $\chara K>0$, $K$ is isomorphic to the field~$\F_q(\!(t)\!)$ of formal Laurent series in~$t$ 
over a finite field $\F_q$. The corresponding residue field again is $\F_q$. 
We note that if $K$ is a finite field extension of $\F_q(\!(t)\!)$, then 
$K \cong \F_{q'}(\!(s)\!)$, where $q | q'$.

In our work we will assume throughout that $\chara K=0$. We shall also point
out where the difficulties in positive characteristic arise.

Every nonzero element~$x$ of $K$ 
can be uniquely written as $x=\varpi^k u$ with $k\in\Z$ and $u\in R^\times$ a unit. 
This defines the discrete valuation $\val\colon K^\times \to \Z,\, x\mapsto k$.
We obtain the isomorphism 
\begin{equation}\label{eq:isoK}
 K^\times \to \Z \times R^\times,\ 
  x \mapsto (k,u) 
\end{equation}
of topological groups. We may view this as an 
analogue of the polar decomposition of nonzero complex numbers. 

\subsubsection{Norm} 

We define the {\em standard norm} on the vector space $K^n$ by 
\begin{equation}\label{eq:normdef}
  \|(x_1,\ldots,x_n)\| := \max_i |x_i| .
\end{equation}
It satisfies the ultrametric inequality:
$\|x+y\| \le \max\{\|x\|, \|y\|\}$ for $x,y \in K^n$.
Note that $\|\lambda x\| = |\lambda| \|x\|$ for 
$\lambda\in K$ and $x\in K^n$. 
Moreover, 
$R^n =\{x\in K^n \mid \|x\| \le 1 \}$. 

The {\em operator norm} of a matrix $A=[a_{ij}]\in K^{m\times n}$ is defined by 
$\|A\| := \max_{\|x\|=1} \|Ax\|$. Note that 
$\|AB\| \le \|A\| \|B\|$. 
The operator norm can be easily computed, since 
\begin{equation}\label{eq:opnorm}
  \|A\| = \max_{i,j} |a_{ij}| .
\end{equation}
Indeed, let $y=Ax$, say $y_i = \sum_j a_{ij} x_j$. Then
$|y_i| \le \max_j |a_{ij}| |x_j| \le \max_j |a_{ij}|$ 
if $\|x\|\le 1$. Hence 
$\|y\| =\max_i |y_i| \le \max_{i,j} |a_{ij}|$ 
and we see that $\|A\| \le \max_{i,j} |a_{ij}|$. 
Equality holds since 
$\|A\|\geq \|Ae_j\| = \max_i |a_{ij}|$.

Consider the group of $R$--linear automorphisms of the free module~$R^n$: 
$$
 \Gl_n(R) := \{ A\in R^{n\times n} \mid \det(A) \in R^\times \} .
$$ 
We next show that $\Gl_n(R)$
is the group of norm preserving linear 
endomorphisms of~$K^n$; thus it 
takes the role that the orthogonal group serves in the setting of 
Euclidean spaces. 

\begin{prop}\label{pro:isometry-grp}
We have 
$$
 \big\{ A\in \Gl_n(K) \mid \forall x \in K^n\ \|A x\| = \|x\| \big\} =\Gl_n(R).
$$ 
\end{prop}

\begin{proof}
Let $A\in \Gl_n(R)$. 
Since $R^n$ is the unit ball in our norm 
and $A$ maps $R^n$ into itself,
we have $\|A\| \le 1$. 
Suppose $x \in K^n$. Then,
$\|Ax\| \le \|A\| \, \|x\| \le \|x\|$. 
Applying this to $A^{-1}\in \Gl_n(R)$, 
we see that 
$\|x\| = \|A^{-1} Ax\| \le \|Ax\|$, hence  
equality holds.

Conversely, suppose that $A\in \Gl_n(K)$ preserves the norm. 
Let $A=SDT$ be the factorization coming from the 
Smith normal form (see Section \ref{se:smith}). 
Thus $S,T\in \Gl_n(R)$, and $D$ is diagonal. 
Since $S$ and $T$ are norm preserving, 
$D$ is norm preserving as well. This implies that the diagonal 
entries of $D$ are units in~$R$. Hence $A\in \Gl_n(R)$. 
\end{proof}  

We call $B(p,\e) := \{x\in K^n \mid \|x-p\| < \e\}$ 
the {\em open ball} around $p\in K^n$ with radius $\e>0$. 
A strange consequence of the ultrametric inequality is that 
for any two non-disjoint open balls $B$ and $B'$ in $K^n$, 
we have  $B\subseteq B'$ or $B'\subseteq B$, see \cite[Lemma~1.3]{schneiderp:11}. 
As a consequence, any covering by open balls of $R^n$ can be refined to a 
{\em disjoint} covering, compare~\cite[Lemma~1.4]{schneiderp:11}.

\subsubsection{Haar measure} 

Since the additive group of $K$ is a locally compact abelian group, there is a uniquely determined regular 
measure~$\mu$ on $K$, which is invariant under translations 
$K\to K, x\mapsto x + a$ for $a\in K$ and satisfies the normalizing condition $\mu(R)=1$; 
see~\cite{bourbaki-II-7-9}. This is called the {\em Haar measure} on $K$. 
The uniqueness implies that the multiplication $K\to K, x\mapsto ux$ 
with a unit $u\in R^\times$ preserves $\mu$. Therefore, 
the restriction of $\mu$ to subsets of $R^\times$ gives a Haar measure of $R^\times$; 
after dividing by $\mu(R^\times)$, this is the normalized Haar measure of the compact 
group $R^\times$. 

Restricting $\mu$ to subsets of $R$, we obtain the Haar measure on $R$, which is a probability measure. 
Its pushforward measure with respect to the canonical map 
$\pi_k\colon R\to R/\frm^k$ is the Haar measure on $R/\frm^k$, hence it equals the uniform 
measure on the finite group $R/\frm^k$ of order~$q^k$. 
This means that 
$\mu(\pi_k^{-1}(Y)) = \#Y/q^k$ for all subsets $Y$ of $R/\frm^k$.  
This leads to the following concrete description of the Haar measure. 
Let $X\subseteq R$ be a closed subset 
and define $X_k := \pi_k^{-1}(\pi_k(X))$.
One easily checks that $X_k\supseteq X_{k+1}$ and 
$X=\cap_{k\in\N}X_k$. Hence 
\begin{equation}\label{eq:mu-count}
 \mu(X) = \lim_{k\to\infty} \mu (X_k) = \lim_{k\to\infty} \#\pi_k(X)q^{-k} .
\end{equation}
This characterization implies that 
$\mu(\varpi X) = \e \mu(X)$, 
where we recall that $\varpi$ denotes a generator of $\frm$ 
and $\e =q^{-1}$. 
In particular, $\mu(\frm^k) = \e^{k}$ for $k\in\N$. 

Clearly, the product of the Haar measures $\mu$ on $K$ defines a Haar measure $\mu_n$ on $K^n$: 
we have $\mu_n(X_1\times\cdots\times X_n) = \mu(X_1)\cdot\ldots\cdot\mu(X_n)$ for Borel measurable $X_i\subseteq K$. 
This defines a probability measure on $R^n$, that we call the 
{\em uniform measure on $R^n$.}  
We remark that this measure is considered an appropriate analogue of the standard multivariate 
Gaussian distribution~\cite{evans:06,zelenov:19}.
We claim that any $A\in\Gl_n(R)$ preserves the Haar measure, that is, 
$\mu_n(A(X)) = \mu_n(X)$ for any Borel measurable $X\subseteq K^n$. 
Indeed, $\tilde{\mu}(X) := \mu_n(A(X))$ defines a translation invariant measure on $K^n$, hence 
$\tilde{\mu} = c\mu$ for some fixed $c$, by uniqueness. Taking $X=R^n$ reveals that $c=1$.

We can characterize the uniform probability measure on $R^n$ as in~\eqref{eq:mu-count}:
for $k\in\N$ we consider the canonical map 
$\pi_k\colon R^n\to (R/\frm^k)^n$. The pushforward of $\mu_n$ equals the Haar measure on the 
finite group $(R/\frm^k)^n$ of order $q^{kn}$, which is nothing but the normalized counting measure.
(Up to a scaling factor,
$\mu_n$ is the $n$--dimensional Hausdorff measure.) 
If we denote by $N_k(X)$ the cardinality of $\pi_k(X)$ for $X\subseteq R^n$, 
we conclude with the same reasoning as above that
\begin{equation}\label{eq:mu-n-count}
  \mu(X) = \lim_{k\to\infty} \pi_k(X) q^{-kn} 
\end{equation}
for any Borel measurable subset $X\subseteq R^n$ (see~\cite[Thm.~2]{oesterle:82}).

\begin{defi}\label{def:sphere}
We define the \emph{unit sphere} in $K^n$ as the subset $S(K^n) := \{x \in K^n : \|x\| = 1\}$. 
\end{defi}

Note that $S(K^n)$ has positive $n$--dimensional measure! 
Indeed, 
$S(K^n) =\{ x \in K^n \mid \|x\| = 1\}$, 
hence $S(K^n) = R^n \setminus \frm^n$ and 
$\mu_n(S(K^n)) = 1 - \e^{n}$. 

Similarly as for~\eqref{eq:isoK} we can uniquely factor any nonzero 
$x\in K^n$ in the form $x=\varpi^k u$, where $k\in\Z$ and $u\in S(K^n)$. 
For this we take $\e^k = \|x\|$ and $u =x/\|x\|$. This leads to a 
homeomorphism
\begin{equation}\label{eq:isoKn}
 K^n\setminus 0 \to \Z \times S(K^n),\, x\mapsto (k,u) . 
\end{equation}

\begin{prop}\label{prop:pushfHaarII}
  The following properties hold:
  \begin{enumerate}
  \item
    The pushforward measure of the Haar measure $\mu$ on $K^n$ 
    with respect to the map~\eqref{eq:isoKn} 
    is the product of the normalized Haar measure on $S(K^n)$ 
    with the discrete measure $\rho_n$ on~$\Z$ given by 
    $$
    \rho_n(k) = (1-\e^n) \e^{kn},\quad k\in\Z .
    $$

  \item
    For $m\in\N$, the expectation of $\|x\|^m$ over $x\in R^n$ is given by 
    $$
    \E(\|x\|^m) = (1- \e^n) (1-\e^{n+m})^{-1} .
    $$
  \end{enumerate}
\end{prop}

\begin{proof}
  \begin{enumerate}[itemindent=\dimexpr \labelwidth+\labelsep+\parindent \relax, leftmargin=0pt, itemsep=\parskip]
    \item
      Let $X\subseteq S(K)$ be closed and $k\in \Z$. As above, 
      using~\eqref{eq:mu-n-count}, one shows that 
      $\mu(\varpi^k X) = \e^{kn} \mu(X)$. 
      Hence, $\mu(\varpi^k X ) =   (1-\e^n) \e^{kn} \nu(X)$, 
      where $\nu(X) =\mu(X)/(1-\e^n)$ is the normalized Haar measure of $S(K^n)$. 
      This implies the first assertion. 

    \item
      With the first assertion, we get for $m\in\N$, 
      $$ 
      \E(\|x\|^m) = \sum_{k\in\N} \e^{km} \rho_n(k) = (1-\e^n) \sum_{k\in\N} \e^{k(m+n)}  
      = (1-\e^n) (1-\e^{m+n})^{-1}  
      $$
      and the second assertion follows. \qedhere
\end{enumerate}
\end{proof}

\subsection{Smith normal form}\label{se:smith}
In this section we recall some properties of the Smith normal form, 
which is the nonarchimedean analogue of the Singular Value Decomposition.
Let $A\in K^{m\times n}$ be of rank $r$.
The fundamental theorem on the  {\em Smith normal form} assigns to~$A$ 
uniquely determined integers $k_1\le\ldots\le k_r$ such that 
$SAT =D$ for some $S\in\Gl_m(R)$, $T\in\Gl_n(R)$, and where 
$D=[d_{ij}]$ with $d_{ij} =0$ for $i\ne j$, 
$d_{ii} = \varpi^{k_i}$ if $i\le r$ and $d_{ii} = 0$ otherwise. 
We say that such matrix $D$ is in {\em Smith normal form}.
The $\varpi^{k_1},\ldots,\varpi^{k_r}$ are the {\em invariant factors} of~$A$. 
Equivalently, these are the invariant factors of the $R$--module spanned by the columns of~$A$. 
The factorization $A=SDT$ can be seen as an 
analogue of the singular value decomposition of linear 
maps of Euclidean spaces; see~\cite{evans:02}.   
We shall call the absolute values of the diagonal elements $\s_i:=|d_{ii}|$
the \emph{singular values} of $A$, for $1\le i \le \min\{n,m\}$,
and assume that
$$\sigma_1\ge\ldots\ge \sigma_{\min\{n,m\}}\ge 0.$$

We can characterize the singular values as follows.
Let $A_{I,J}$ denote the submatrix of $A$ obtained by selecting the 
rows indexed by $I\subseteq [m]$ and $J\subseteq [n]$, respectively.
Then, given $s\le \min\{n, m\}$, 
\begin{equation}\label{eq:invar-fact-char}
  \sigma_1\cdots\sigma_s = \e^{k_1 + \ldots + k_s} =\max \{ |\det(A_{I,J})| : |I| = |J| = s\}.
\end{equation}
We shall give another characterization of singular values below, in Section~\ref{sec:minmax}, 
after having introduced the notion of $R$--structure.

In order to emphasize that the singular values are discrete parameters, 
let us encode them differently. 
Note that 
$$
 \Z \to \{|r| \mid r\in K^\times\},\, k \mapsto \e^k
$$
is a group isomorphism: 
its image in $\mathbb{R}$ is the discrete subgroup of the multiplicative group 
of positive rational numbers generated by $\e$.
It is convenient to extend this map to 
$\bZ:=\Z\cup\{\infty\}$ 
by the convention $\e^\infty := 0$. 
Setting $\s_i = \e^{x_i}$, 
the list of singular values of $A\in K^{m\times n}$ can be 
therefore uniquely encoded by an element of 
$$
 \bZ^r_o := \{ x\in\bZ^r \mid x_1 \ge \cdots \ge x_{\min\{n,m\}}\} .
$$

We define the {\em absolute determinant} of $A\in K^{m\times n}$ as 
the product of the its singular values, thus, 
\begin{equation}\label{eq:abs-det}
  N(A) : = \sigma_1 \cdot \ldots \cdot \sigma_{\min\{m, n\}} = \max\{|\det A_{I,J} | : |I| = |J| = \min\{n,m\}\}.
\end{equation}
This quantity turns out to be the appropriate substitute 
for the normal Jacobian in the coarea formula. The following observation will be extremely useful.
If the matrix $A \in K^{m\times n}$ is a row ($m=1$) 
or a column ($n=1$), then 
its absolute determinant equals
\begin{equation}\label{eq:adet-row}
 N(A) = \|A\| .
\end{equation} 

\subsubsection{A block Smith normal form decomposition}\label{se:block-SVD}

We already noted that the Smith normal form decomposition can be seen
as an analogue of the singular value decomposition. We now provide
a block version of this decomposition.
It can be seen as a nonarchimedean analogue of the CS-decomposition over $\R$ and $\C$, 
see~\cite{paige-wei:94} and \cite[Appendix]{stewart:77}. 
(While elementary, we could not locate this result within the literature.)

Let $1\le \ell \le n$ and decompose 
 \[\label{eq:A-decomp}
    A=\begin{bmatrix} A_{11} & A_{12}\\ A_{21} & A_{22}\end{bmatrix} \in \Gl_n(R) ,
  \]
where the upper left block has format $\ell \times \ell$. 
We denote by $U_{\ell,n}$
the subgroup of block-upper triangular matrices satisfying $A_{21}=0$.

\begin{thm}\label{th:BSVD}
  Let $A \in \Gl_n(R)$ and $1\le \ell \le n$. Then 
  there are $S, T \in U_{\ell,n}$ such 
  \[
    SA T= 
    \begin{bmatrix} I_\ell & 0\\ D & I_{n-\ell} \end{bmatrix}  ,
  \]
  where $D\in R^{(n-\ell)\times \ell}$ is in Smith normal form and has full rank 
  $\min\{n-\ell, \ell\}$.
\end{thm}

\begin{proof}
Suppose $A$ is decomposed as in \eqref{eq:A-decomp}. 
We apply Gaussian elimination to the matrix~$A_{11}$, making sure that  
within each nonzero column, the pivot positions contain elements of maximal norm.
(This amount to carrying out Gaussian elimination over the residue field $R/\frm$.) 
This implies that there is $Q_1 \in \Gl_\ell(R)$ such that 
$Q_1A_{11}\bmod\frm$ is in reduced row echelon form.  
Moreover, by a column permutation, we may achieve that this 
is an upper triangular matrix. More specifically, there is a permutation matrix $P_1\in\GL_\ell(R)$ such that  
$H_1:= Q_1 A_{11} P_1 \bmod\frm$ is upper triangular.  
Similarly, working with the possibly non-square matrix~$A_{21}$, 
we show that  there are $Q_2,P_2 \in \Gl_{n-\ell}(R)$ such that 
$H_2:= Q_2 A_{21}P_2 \bmod\frm$ has zero entries below the main diagonal. 

For any matrix $E$ in $R^{\ell\times \ell}$, we have 
 \[
   \begin{bmatrix} Q_1& E\\ 0 & Q_2\end{bmatrix} 
   \begin{bmatrix} A_{11} & A_{12}\\ A_{21} & A_{22}\end{bmatrix} 
   \begin{bmatrix} P_1& 0\\ 0 & P_2\end{bmatrix}
  = \begin{bmatrix} (Q_1 A_{11} + E Q_2 A_{21})P_1& *\\ Q_2 A_{21}P_2 & *\end{bmatrix}  
  = \begin{bmatrix} H_1 + EH_2& *\\ H_2 & *\end{bmatrix} .
\]
We want to choose $E$ such that $H_1 + EH_2$ is invertible over $R$. 
Note that for $E=0$, the matrix above is invertible over $R$. 
From the upper triangular form of the $H_i\bmod\frm$, 
we see that for all $i\le \ell$, the $i$th diagonal entry $d_i$ of $H_1$ 
or the $i$th diagonal entry $d'_i$ of $H_2$ must be a unit in~$R$. 
By the ultrametric inequality, either $|d_i|=1$, or $|d_i|<1$ and $|d_i +d'_i|=1$. 
We now define $E$ to be the diagonal matrix whose $i$th diagonal entry equals $1$ if
$d_i$ is not a unit and equals zero otherwise. Then the diagonal entries of $Q:=H_1 + EH_2$
are equal to $d_i$ or $d_i +d'_i$, respectively.
In particular, $Q$ is invertible over $R$.
Thus, we arrive at a decomposition 
\[
      \begin{bmatrix}
         I_\ell & B_{12} \\ B_{21} & B_{22}
     \end{bmatrix} := 
    \begin{bmatrix} Q^{-1}& 0\\ 0 & I_{n-\ell}\end{bmatrix} 
    \begin{bmatrix} Q_1& E\\ 0 & Q_2\end{bmatrix} 
   \begin{bmatrix} A_{11} & A_{12}\\ A_{21} & A_{22}\end{bmatrix} 
   \begin{bmatrix} P_1& 0\\ 0 & P_2\end{bmatrix} .
\]
Now we use the Smith normal form of $B_{21} = T_1D S_1$,
where $T_1 \in \Gl_{n-\ell}(R)$, $S_1 \in \Gl_{\ell}(R)$
and the matrix $D$ is as in the statement of the theorem. 
Then we obtain 
\[
    \begin{bmatrix}
      S_1 & 0 \\ 0 & T_1^{-1}
    \end{bmatrix}
    \begin{bmatrix}
      I_\ell & B_{12} \\ B_{21} & B_{22}
    \end{bmatrix}
    \begin{bmatrix}
      S_1^{-1} & -B_{12} \\ 0 & I_{n-\ell}
    \end{bmatrix}
    =
    \begin{bmatrix}
      I_\ell & 0 \\ D & C_{22}
    \end{bmatrix}
\]
for some matrix $C_{22} \in \Gl_{n-\ell}(R)$, which is over $R$ by construction. 
By multiplication with a further block diagonal matrix, 
we can turn $C_{22}$ to $I_{n-\ell}$ and the assertion follows. 
\end{proof}

Suppose we are in the setting of Theorem~\ref{th:BSVD}. 
Let $A=[A_{ij}]_{1\le i,j\le 2}$ and $T=[T_{ij}]_{1\le i,j\le 2}$ be 
block decompositions of $A$ and $T$, respectively, 
where $A_{11},T_{11}$ have the format $\ell\times\ell$. 
Then 
\begin{equation}\label{eq:finding}
  S \begin{bmatrix} A_{11} \\ A_{21} \end{bmatrix}  T_{11} =
  \begin{bmatrix}
      I_{\ell} \\ D 
   \end{bmatrix} .
\end{equation}
We next generalize this finding, which 
will be our main application of Theorem~\ref{th:BSVD}.
(The above observation covers the special case $k=\ell$.) 

\begin{cor} \label{cor:rectangle-BSVD}
Let $M \in R^{n \times k}$ and denote by $U\subseteq K^n$ the $K$-span of the columns of $M$. 
We assume that the $R$-span of these columns equals $U\cap R^n$. Then, 
for any $k \leq \ell \leq n$,  
there exist $P \in U_{\ell, n}$ and $Q \in \GL_k(R)$ such that
\[
    P M Q =
    \begin{bmatrix}
      I_k \\ 0 \\ D
    \end{bmatrix} ,
\]
where $D\in R^{(n-\ell)\times k}$ is in Smith normal form and has full rank $\min\{n-\ell,k\}$ 
(the middle matrix~$0$ has format $(\ell-k)\times k$). 
\end{cor} 

\begin{proof}
We extend the list of columns of $M$ to an $R$--basis of $K^n$ (see 
the forthcoming Lemma~\ref{le:extend-basis}) 
and call the resulting matrix $A\in\GL_n(R)$. 
Theorem~\ref{th:BSVD} yields a factorization 
\begin{equation}\label{eq:BSDec}
  SAT= 
  \begin{bmatrix} I_\ell & 0\\ E & I_{n-\ell} \end{bmatrix} 
\end{equation}
for some $S, T \in U_{\ell,n}$ and where $E$ is in Smith normal form. 
We block decompose 
$A=[A_{ij}]_{1\le i,j\le 2}$, $S=[S_{ij}]_{1\le i,j\le 2}$, $T=[T_{ij}]_{1\le i,j\le 2}$,  
where the $(1,1)$--blocks have the format $\ell\times\ell$.
Note that $S_{21}=0$ and $T_{21}=0$.  
From \eqref{eq:BSDec} we get as for \eqref{eq:finding} that 
\begin{equation}\label{eq:BSDec1}
  S\begin{bmatrix} A_{11} \\ A_{21} \end{bmatrix} T_{11} = 
  \begin{bmatrix} I_\ell \\ E \end{bmatrix} .
\end{equation}
This is a matrix of format $n\times \ell$, that 
we further block decompose:  
\begin{equation}\label{eq:BSDec2}
  S\begin{bmatrix} A_{11} \\ A_{21} \end{bmatrix} =:
  \begin{bmatrix} X_1 & X_2\\ Y_1 & Y_2  \end{bmatrix} ,
\end{equation}
where $X_1$ has the format $\ell\times k$ and 
$Y_1$ has the format $(n-\ell)\times k$. 
In particular, restricting to the first $k$ columns, we get 
\begin{equation}\label{eq:BSDec3}
  S M =\begin{bmatrix} X_1 \\ Y_1 \end{bmatrix} .
\end{equation}
From \eqref{eq:BSDec1} and \eqref{eq:BSDec2} we obtain 
$[X_1 \ X_2 ]\; T_{11} = I_\ell$.  
This implies 
$T_{11} \; [ X_1  \ X_2]  = I_\ell $, 
and hence 
\begin{equation}\label{eq:BSDec4}
  T_{11} X_1  = 
  \begin{bmatrix} I_k \\ 0\end{bmatrix} .
\end{equation}

Now we take the Smith normal form decomposition of $Y_1$: 
there are $U \in \GL_{n-\ell}(R)$, $T \in \GL_{k}(R)$, 
and $D$ in Smith normal form such that $UY_1 T = D$.
Together with Equations~\eqref{eq:BSDec3} and \eqref{eq:BSDec4},
we conclude that 
\[
\begin{bmatrix}
      T_{11} & 0 \\
      0 & U
\end{bmatrix}
 SMT = 
\begin{bmatrix}
      T_{11}X_1 T \\
       U Y_1 T
\end{bmatrix} = 
\begin{bmatrix}
      T \\ 0 \\ D
\end{bmatrix} .
\]
Finally, we arrive at 
\[
\begin{bmatrix}
      T^{-1} & 0 & 0 \\
      0 & I_{\ell-k} & 0 \\
      0 & 0 & I_{n-\ell}
    \end{bmatrix}
\begin{bmatrix}
      T_{11} & 0 \\
      0 & U
\end{bmatrix}
 SMT = 
\begin{bmatrix}
      I_k \\ 0 \\ D
\end{bmatrix} ,
\]
which is a decomposition as desired, 
where $P$ is the product of the three invertible matrices 
over $R$ on the left, and $Q=T$.  
\end{proof}

%%%

{
\subsection{Analysis on $K$--analytic manifolds}
In this section, we recall some basic results on nonarchimedean 
analytic functions, nonarchimedean manifolds, and integration parallel
to the standard material over the reals. 
We refer to \cite{robert:00,schikhof,schneiderp:11} 
for general background on calculus on ultrametric spaces.
Our interest is in analytic functions and analytic manifolds, but we remark that 
one could consider more generally strictly differentiable maps 
and smooth ultrametric manifolds, see for instance \cite{strictC1}.

\subsubsection{$K$--analytic functions}\label{se:analytic-maps} 

Let $U\subseteq K^m$ be an open subset. 
A {\em (locally) $K$--analytic} map $\varphi\colon U \to K^n$
is a map that can be locally defined by convergent power series with coefficients in~$K$. 
When $n=m$, we define the {\em absolute Jacobian determinant} 
$J(\varphi)(x) := |\det(D\varphi(x))|$  of~$\varphi$ at $x\in U$
as the absolute value of the determinant of the Jacobian matrix
$D\varphi(x)\in K^{n\times n}$. We will say that the map $\varphi$ is {\em $K$--bianalytic} 
if $\varphi\colon U \to V$ is a bijection onto an open subset $V\subseteq K^n$ 
and $\varphi$ and its inverse are $K$--analytic.
The inverse function theorem holds for 
analytic maps, see~\cite{schneiderp:11}. 
This implies the implicit function theorem in the usual way. 

\subsubsection{Integration}

Suppose $U\subseteq K^n$ is open and $f\colon U\to [0,\infty]$ 
is (Borel) measurable. By the general theory of integration, we can integrate $f$ 
with respect to the Haar measure $\mu_n$ and obtain a well defined value 
$\int_U f d\mu_n \in [0,\infty]$, see~\cite{bourbaki-II-7-9} or \cite{popa}.
As usual, a function $f\colon U\to \R$ is called {\em integrable} iff 
the integrals of its positive and negative parts $f_+,f_-$ are finite, and one defines 
$\int_V f \, d\mu_n := \int_V f_+ \, d\mu_n - \int_V f_- d\mu_n$. 

\subsubsection{$K$--analytic manifolds}

For precise definitions and background on this section we refer to the summary~\cite{bourbaki:var-diff-analyt} 
(without proofs) and the detailed treatments in Serre~\cite{serre:64} and Schneider~\cite[Chap.~2]{schneiderp:11}.

\begin{defi}\label{def:K-mf}
A {\em $K$--analytic manifold} $X$ of dimension~$n$ is a Hausdorff topological space 
satisfying the second countability axiom, together with a {\em $K$--analytic atlas}. 
Such atlas consists of a family of {\em charts} 
$\psi_i\colon U_i \to V_i$, where the collection of $U_i$ forms an open cover of~$X$, 
the $V_i$ are open subsets of $K^n$ and the $\psi_i$ are homeomorphisms. 
Moreover, it is required that every {\em change of coordinates map} 
$\psi_i \circ \psi_j^{-1}$ is $K$--analytic on its domain of 
definition $\psi_j(U_i\cap U_j)$ (if nonempty).
\end{defi}
 
In particular, the change of coordinate maps are $K$--bianalytic.
The second axiom of countability ensures the existence of partitions of unity 
(paracompactness). We note that in the ultrametric setting, it is always possible to refine 
an atlas to arrive at pairwise disjoint domains~$U_i$ (strict paracompactness), 
see~\cite[Lemma~1.4 and Prop.~8.7]{schneiderp:11}.  

The space $K^n$ is a $K$--analytic manifold of dimension~$n$ via the atlas consisting of the 
single chart given by the identity on $K^n$. 
Clearly, an open subset $U$ of a $K$--analytic manifold $X$ of dimension~$n$ 
inherits from $X$ the structure of a $K$--analytic manifold $X$ of dimension~$n$.
Hence the open subset $R^n$ is a $K$--analytic manifold $X$ of dimension~$n$, 
which is also closed and compact. 

In fact, Serre~\cite{serre-class:65} observed that any compact 
$n$--dimensional $K$--analytic manifold~$X$ is analytically isomorphic 
to a disjoint union of finitely many copies of $R^n$, 
see also~\cite[III, Appendix~2]{serre:64}. 

The definition of $K$--analytic maps between $K$--analytic manifolds, their differentials, etc.\
carries over as in the classical framework.
Also, the concept of a submanifold is defined in the usual way as follows. 

\begin{defi}\label{def:submf}
Let $X$ be a $K$--analytic manifold. 
A subset $Y\subseteq X$ is called an {\em (embedded) $K$--analytic submanifold} of $X$ 
if for every $p\in Y$, there exists an open neighbourhood $U$ of $p$ in $X$, 
an open subset $V\subseteq K^n$, and a 
$K$--bianalytic map $\varphi\colon U \to V$ 
such that 
$\varphi(U\cap Y) = V \cap (K^m\times 0^{n-m})$.
Then $Y$ inherits from $X$ a $K$--analytic atlas, which turns $Y$
into a $K$--analytic manifold.
\end{defi}

A $K$--analytic map $\varphi\colon X \to Y$ of $K$--analytic manifolds is
called a {\em ($K$--analytic) submersion} if its derivative 
$D_x\varphi\colon T_xX \to T_{\varphi(x)}Y$ is surjective for all $x\in X$.  
The implicit function theorem implies that locally, in suitable charts, a submersion 
is a coordinate projection $K^n\to K^m$ onto the first $m$ coordinates. 
This implies that the fibres $\varphi^{-1}(y)$ are $K$--analytic submanifolds of 
dimension $\dim X -\dim Y$, or empty. 

We note that an $n$--dimensional analytic $K$--manifold has a {\em tangent bundle} $TX$. 
Moreover, for $k\le n$, we can define the bundle $\Lambda^k X$ of alternating tensors of order~$k$, 
whose fibre over $x\in X$ consists of the alternating multilinear forms $(T_xX)^k \to K$. 
By a {\em $K$--analytic $k$--form} on~$X$, we 
mean an analytic section of $\Lambda^k X$. 

\subsubsection{Sard's lemma}

Let $\varphi\colon X \to Y$ be an analytic map of $K$--analytic manifolds. 
We say that $x\in X$ is a {\em critical point} of $\varphi$ if 
the derivative $D_x\varphi:T_xX\to T_{\varphi(x)}Y$ is not surjective.
We denote by $C(\varphi)$ the set of critical points of $\varphi$ 
and we call its image $\varphi(C(\varphi))$ the set of its {\em critical values}. 

Let $X$ be a $K$--analytic manifold of dimension $n$. We say that a measurable subset $S\subseteq X$ 
has \emph{measure zero} if $\psi_i(S\cap U_i)\subseteq K^n$ has (Haar) measure zero 
for every chart $\psi_i\colon U_i \to V_i$ of the $K$--analytic atlas of~$X$.
As in the classical case, this notion is well defined: 
$K$--analytic functions are locally Lipschitz, and hence
the image of a set of measure zero under an analytic map still has measure zero.
(The well-definedness also follows from Proposition~\ref{th:TF} below.)

A version of Sard's lemma, Theorem~\ref{th:sard} below, holds for $K$--analytic manifolds. Surprisingly, we could not locate a proof of 
this result in the literature (the reference~\cite{tempered} only treats the case of a polynomial map $\varphi$ between vector spaces).
Therefore, we provide a proof in Appendix~\ref{sec:proofsard}, which is essentially the same as the classical one, see~\cite{GP:74,milnor:65}, with minor adaptations.

\begin{thm}[Sard's lemma]\label{th:sard}
Let $\varphi\colon X \to Y$ be a $K$--analytic map of $K$--analytic manifolds (and assume $\chara K=0$).  
Then the set of critical values of~$\varphi$ has measure zero. 
\end{thm}

\begin{remark}
Sard's lemma generally fails over nonarchimedean local fields $K$ of positive characteristic. 
The following example taken from~\cite{tempered}  is due to Deligne.  
Let $K = \F_p(\!(t)\!)$ denote the field of formal Laurent series in $t$ over $\F_p$ 
and consider the polynomial map $f\: K^{p+1} \rightarrow K^2$ defined by
\[
  f(y, x_1, \ldots, x_p) = \Big( y, \ \textstyle \sum_{j=1}^p t^{j-1} x_j^p \Big).
\]
Notice that $D_xf$ is a $2 \times (p+1)$ matrix whose only non-zero entry is $\partial_{y} y = 1$. 
Thus the set of critical values of $f$ is all of $K^{p+1}$. 
Additionally, we see that $f$ is surjective: indeed,  
the elements $1, t, \ldots, t^{p-1}$ are a basis for $K$ over the subfield $\F_p(\!(t^p)\!)$, 
and so the second component of $f$ has image~$K$.
However, we note that in~\cite{tempered} it was shown that Sard's lemma holds in characteristic~$p$ if 
$p>\dim(X)-\dim(Y)+1$.
\end{remark}

\subsubsection{Analytic change of variables}

The following transformation formula is analogous to the known result over 
$\R$ and $\C$. Evans proved it for $K=\Q_p$ in \cite[Proposition~2.3]{evans:06}; 
it is straightforward to check that his proof extends to any nonarchimedean local 
field~$K$ of characteristic zero.

\begin{prop}\label{th:TF}
Let $U\subseteq K^n$ be open, let $\varphi\: U \rightarrow K^n$ be a $K$--analytic map, 
and let $f\:K^n \rightarrow [0, \infty)$ be measurable (and assume $\chara K=0$).  
Then
$$
  \int_U (f\circ \varphi) \cdot |\det D\varphi| \, d\mu_n  = \int_{y\in K^n} f(y)\,\#\varphi^{-1}(y)  \, d\mu_n(y) .
$$	
\end{prop}

\subsubsection{Volume measure and integration absolute value of forms}\label{se:vol-int} 

Suppose that $X$ is a $K$--analytic manifold of dimension~$n$.  
Let $\alpha$ be a $K$--analytic $n$--form on~$X$ and $h\colon X\to\R$ 
be a continuous function with compact support. 
We define the {\em integration with respect to the absolute value of the form} $\alpha$
as follows, cf.~\cite{bourbaki-II-7-9}, \cite[\S3.3]{serre:81}, \cite{popa}  and \cite{AA}.
Assume first that the support of $h$ is contained in an open chart $U\subseteq X$ 
with coordinates $t^1,\ldots,t^n$. Then we can write 
$\a = f dt^1 \wedge\ldots \wedge dt^n$
with a $K$--analytic function $f\colon U \to K$.
In this case we define
\begin{equation}\label{eq:def-int}
 \int_X h \, |\alpha| := \int_{U} h\,|f| \, d\mu_n ,
\end{equation}
where $d\mu_{K^n}$ denotes the Haar measure. This quantity is 
well defined due to the change of variable formula (Proposition~\ref{th:TF}). 
In general, we define 
$\int_X h\, |\alpha|$ using a partition of unity in the usual way.   

We will need the following elementary result on fibre integration.

\begin{lemma}\label{le:FInt}
Let $X$ and $Y$ be $K$--analytic manifolds 
and $\varphi\colon X\to Y$ be a $K$--analytic submersion.
Moreover, let $\a$ be a $K$--analytic $(n-m)$--form on $X$ and 
$\b$ be a $K$--analytic $m$--form on $Y$, where $n=\dim X$ and $m=\dim Y$. 
Moreover, suppose $h\colon X\to [0,\infty]$ is measurable. Then we have 
\begin{equation*}\label{eq:FIT}
 \int_X h\, |\a \wedge \varphi^*\b | = \int _{y\in Y} \Big(\int_{\varphi^{-1}(y)} h\, |\a| \Big)\, |\b| (y) .
\end{equation*}
\end{lemma}

\begin{proof}
We proceed as in~\cite[Appendix]{howard:93}. 
Using a partition of unity, we can reduce to the local statement, where 
$X\subseteq K^n$ and $Y\subseteq K^m$ are open subsets.  
Moreover, we may assume that 
$\varphi$ is given by the projection onto the first $m$ coordinates
(note that $\varphi\colon X\to Y$ is a submersion and apply the implicit function theorem). 
Then we can write
$\varphi^*\beta = b\, dx^1\wedge\ldots\wedge dx^m$ 
and
$\alpha = \sum_{|I|=n-m} a_I\, \wedge_{i\in I} dx^i$
with analytic functions $a_I,b$ defined on $X$. 
Setting $J=\{m+1,\ldots,n\}$, we obtain 
$$
 \a \wedge \varphi^*\b = \pm a_{J}\, b\, dx^1\wedge\ldots\wedge dx^n .
$$ 
By~\eqref{eq:def-int}, the assertion now translates to 
$$
 \int_X h\, |a_J b|\, d\mu_n = \int_{y\in Y} \Big(\int_{\varphi^{-1}(y)} h\, |a_J|\, d\mu_{n-m} \Big) |b| d\mu_m . 
$$
But this is a consequence of Fubini's theorem. 
\end{proof}

}

\section{$R$--structures} \label{se:R-structures}
In this section we introduce the notion of \emph{$R$--structure}, which encodes the properties of well known objects in the literature 
in a more analytic framework, which seems more flexible for our purposes. This notion is inspired by \cite[Chapter~AG~11]{Borel} and 
is akin to a Finsler structure on a smooth manifold. With this idea we tried to single out the relevant properties from the classical literature 
on the subject (see for instance \cite{CLT, AA}) in a way that makes it  closer to \cite{howard:93, BL:19} and ready to use also for nonspecialists. 
Throughout the exposition we will explain the connection with already existing objects, more familiar to the experts.

We assume that $K$ is a nonarchimedian local field 
of characteristic zero with discrete valuation ring $R$ 
and we adopt the notation from Section~\ref{se:basics-I}. 

\subsection{$K$--vector spaces with linear $R$--structure}

Here is the central definition.

\begin{defi}
Let $V$ be an $n$--dimensional $K$--vector space. 
A {\em linear $R$--structure} on $V$ is a finitely generated $R$--submodule $V_R$ of $V$ spanning $V$ 
as a $K$--vector space. We call $V_R$ the {\em set of points of $V$ defined over $R$}. 
\end{defi}

In other words, $V_R$ is an $R$-lattice spanning $V$. 
The paradigmatic example is $V=K^n$, which has the linear $R$--structure $R^n$; we call 
it the {\em standard $R$--structure} on $K^n$.   
The crucial observation is that a linear $R$--structure defines a {\em norm} on $V$ as follows 
(cp.~\eqref{eq:normdef}): 
\begin{equation}\label{eq:def-norm}
 \| x_1 e_1 + \ldots + x_n e_n \| : = \|(x_1,\ldots,x_n)\| = \max_i |x_i| ,
\end{equation}
where $(e_1,\ldots,e_n)$ is any $R$--basis of $V_R$. The norm does not depend on the choice 
of the basis, since any $S\in\Gl_n(R)$ preserves the standard norm of~$K^n$ 
by Proposition~\ref{pro:isometry-grp}.
By definition, we have $\|v\| \le 1$ for any $v\in V_R$. 
We shall view $R$--bases of $V_R$ as a substitute of the notion of an 
orthonormal basis in a Euclidean vector space. 

Similarly, we define the {\em measure $\mu_V$} on $V$ as the pushforward of $\mu_n$ via the 
$K$-isomorphism $K^n\to V$, which maps $x$ to $x_1 e_1 + \ldots + x_n e_n$. 
This does not depend on the choice of the basis~$(e_i)$, since the maps defined by 
$S\in\Gl_n(R)$ preserve $\mu_n$. Alternatively, we may characterize $\mu_V$ as the 
Haar measure of the locally compact group $V$ with the normalizing condition 
$\mu_V(V_R)=1$. 

\begin{remark}\leavevmode
\begin{enumerate}
	\item
	A linear $R$--structure $V_R$ on $V$ is a free $R$--module, since it is torsion-free and finitely generated.
	For any $v\in V$ there is a nonzero $r\in R$ such that $rv\in V_R$. 
	More precisely, $K\ot_R V_R \to V, \lambda\ot v \mapsto \lambda v$ 
	provides an isomorphism of $K$--vector spaces.

	\item
	We can assign to $g\in \Gl_n(K)$ the linear $R$--structure on $K^n$ given by the $R$--span of the columns of $g$. 
	This equals the standard $R$--structure on $K^n$ iff $g\in\Gl_n(R)$. Hence, 
	the linear $\R$--structures on $K^n$ are in bijective correspondence with $\Gl_n(K)/\Gl_n(R)$. 
\end{enumerate}
\end{remark}

In the following, we assume $V$ and $W$ are finite dimensional $K$--vector spaces endowed with linear $R$--structures. 
Let $\varphi\colon V\to W$ be a $K$-linear map.  
 Let $A\in K^{m\times n}$ be the matrix representing~$\varphi$ with respect to 
some $R$--bases of $V$ and $W$, respectively. 
Since a base change amounts to a transformation $SAT$ with $S\in\Gl_m(R)$ and $T\in\Gl_n(R)$, 
we see that the singular values $\s_1,\ldots,\s_{\min\{m,n\}}$ of $A$ do not depend 
on the choice of the $R$--bases, see Section~\ref{se:smith}. 
(Of course, $\varphi$ also has well defined {\em invariant factors}, $\varpi^{k_1},\ldots,\varpi^{k_r}$, 
where $r$ is the rank of $A$.) 
We define the {\em absolute determinant} of $\varphi$ as follows (see \eqref{eq:abs-det}):
\begin{equation}\label{eq:abs-det-inv}
  N(\varphi) := N(A) =  \sigma_1 \cdot \ldots \cdot \sigma_{{\min\{m, n\}}}.
\end{equation}

\begin{defi}\label{def:defoverR}
We say that a $K$-linear map $\varphi\colon V\to W$ is {\em defined over $R$} iff 
$\varphi(V_R) \subseteq W_R$. If $\varphi$ induces an $R$-isomorphism $V_R \to W_R$, 
then we say $\varphi$ is an {\em isomorphism defined over $R$}.
\end{defi}

Suppose that $\varpi^{k_1},\ldots,\varpi^{k_r}$ are the invariant factors of $\varphi$. 
It is easy to check that $\varphi$ is defined over $R$ iff  $k_i\ge 0$ for all $i$. 
Moreover, $\varphi$ is an isomorphism defined over $R$ iff 
$r=\dim V=\dim W$ and $k_1=\ldots=k_r=0$.

We show now that linear $R$--structures are naturally preserved 
when forming standard algebraic constructions. 
In particular, this shows that the finite dimensional $K$--spaces with linear $R$--structure 
form an abelian category. 
We leave the straightforward proof of the next lemma to the reader. 

\begin{lemma}\label{le:Rstruct}
  Let $V$ and $W$ be $K$--vector spaces with linear $R$--structures $V_R$ and $W_R$, respectively.
  \begin{enumerate}
  \item
    If $U \subseteq V$ is a a $K$--subspace, then 
    $U_R := U \cap V_R$ is a linear $R$--structure on $U$.
  \item Suppose $U\subseteq V$ is a $K$--subspace and consider 
    $W=V/U$ with the canonical projection $\pi\colon V\to W$. 
    Then, $W_R:=\pi(V_R)$ defines a  linear $R$--structure on $W$.
    This induces a surjective $R$-morphism $V_R\to W_R$ with kernel~$U_R$. 

  \item $(V^*)_R := \{f \in V^* \mid f(V_R) \subseteq R \}$ defines a linear $R$--structure on the dual space~$V^*$.

  \item $(V\ot W)_R := \spann_R\{v\ot w \mid v\in V_R, w \in W_R \}$ 
    defines a linear $R$--structure on $V\ot W$.  
    We have $(V\ot W)_R = V_R \ot W_R$.

  \item $(\Lambda^mV)_R := \spann_R \{ v_1\wedge\ldots\wedge v_m \mid v_i \in V_R\}$
    defines a linear $R$--structure on the  exterior power $\Lambda^mV$. 
    Similarly for the symmetric power $S^m V$. \qed
  \end{enumerate} 
\end{lemma}

We denote the 
linear $R$--structure defined on a $K$--subspace $U\subseteq V$ as in Lemma~\ref{le:Rstruct}(1)
the {\em induced linear $R$--structure} on~$U$.
Moreover, if $\pi\: V \rightarrow W$ is a surjective morphism,
we call the linear $R$--structure on~$W$ given by Lemma~\ref{le:Rstruct}(2) the 
{\em quotient linear $R$--structure}.
Additionally, 
we call a $K$--subspace $W\subseteq V$ an {\em $R$--complement} of $U$ 
if $U_R \oplus W_R = V_R$.

\begin{lemma}\label{le:extend-basis}
Any $K$--subspace $U\subseteq V$ has an $R$-complement. 
In particular, any $R$--basis of $U_R$ can be extended to an $R$--basis of $V_R$.
\end{lemma}

\begin{proof}
Consider $W:=V/U$ with the canonical projection $p\colon V\to W$. 
Put $m:=\dim W$. 
By the Smith normal form, there are $R$--bases $e_1,\ldots,e_n$ of $V_R$ and 
$f_1,\ldots,f_m$ of $W_R$, respectively, such that 
$p(e_i) = \varpi^{k_i} f_i$ for $i\le m$ and $\varphi(e_i)=0$ otherwise. 
The $K$--span of $e_1,\ldots,e_m$ is an $R$-complement of $U=\ker p$.  
\end{proof}

If $n=\dim V$, then 
$(\Lambda^n V^*)_R$ is generated by a single element $\omega_V$, 
which is uniquely determined up to multiplication by a unit of $R$.
More specifically, if $e_1,\ldots,e_n$ is an $R$--basis of $V_R$, we may 
take $\omega_V := e_1\wedge\cdots\wedge e_n$.
Then the Haar measure $\mu_V$ can be characterized using 
integration over the absolute value of $\omega_V$
(see Section~\ref{se:vol-int}):
$$
  \mu_V(U) = \int_U |\omega_V| .
$$
We thus view $ |\omega_V|$ as the {\em volume element} of the 
$K$--space $V$, endowed with the linear $R$--structure~$V_R$.

Since $\Lambda^n V$ has an induced linear $R$--structure, there is 
a well defined norm on this space. This leads to the following useful 
characterization.  

\begin{lemma}\label{le:RbasChar}
Let $v_1,\ldots,v_n \in V_R$ be a $K$--basis of $V$. Then 
$v_1,\ldots,v_n$ is an $R$--basis of $V_R$ iff 
$\|v_1\wedge\cdots\wedge v_n\|=1$.
\end{lemma}

\begin{proof}
Recall from the definition of the norm on $V$ that $\|v\| \leq 1$ for
all $v \in V_R$. By the Smith normal form (see Section~\ref{se:smith}), 
there is an $R$--basis $e_1,\ldots,e_n$ of $V_R$ and 
there are $d_i\in R$ 
such that $v_i=d_i e_i$ for all $i$. Therefore, 
$\|v_1\wedge\cdots\wedge v_n\|=|d_1\cdots d_n|$. 
We now see that $v_1,\ldots,v_n$ is a basis of $V_R$ iff 
all $d_i \in R^\times$. 
\end{proof}

We already defined in~\eqref{eq:abs-det-inv} 
the absolute determinant $N(\varphi)$ of a $K$-linear map 
$\varphi\colon V\to W$. 
We can now characterize it in a different, coordinate-free way. 

\begin{prop}
The absolute determinant $N(\varphi)$ equals the norm of the 
induced $K$-linear map 
$\Lambda^r \varphi\colon \Lambda^r V \to \Lambda^r W$,
where $r=\rk(\varphi)$. 
\end{prop}

\begin{proof}
Let $A$ be the matrix representing $\varphi$ with respect to chosen 
bases of $V_R$ and $W_R$. Then the matrix 
entries of $\Lambda^r\varphi$ in the induced bases of $\Lambda^r V$ 
and $\Lambda^r W$ are given by the minors $\det(A_{I,J})$, 
where $I,J$ have cardinality~$r$. We conclude that 
\[
 N(\varphi) = N(A) = \max_{I,J} |\det(A_{I,J}) | = \|\Lambda^r \varphi\|. \qedhere
\]
\end{proof}

%%%
\subsection{$R$--structures on $K$--analytic manifolds}\label{se:mf-R-struct}

%{\AK{See comment on the previous section. One other thing that might be interesting --
%  it may well be the case that an $R$--structure as we describe it is equivalent to defining
%  a metric on the (analytic) tangent bundle alla [CLT]. In fact, this is exactly the case I
%  believe, since we designate an $R$--structure by locally assigning a lattice within each
%  tangent bundle. In this case, it could be possible to point out this connection.
%}}

The idea of an $R$--structure on a $K$--analytic manifold~$X$ is to assign 
a linear $R$--structure to each tangent space of~$X$, which ``varies analytically''. 
In order to discuss this concept precisely, we require some preparations.

\subsubsection{Some integrality properties of $K$--analytic maps}

\begin{lemma}\label{le:RK-ana}
  Suppose $f$ is a $K$--analytic function defined by a power series in a neighborhood of~$0$ 
  such that $f(0)=0$. 
  Then there is a nonzero $r\in R$ such that the function
  $g(x):=f(r x)$ is given by a power series with coefficients in $R$,
  in a neighborhood of $0$.
\end{lemma}

\begin{proof}
Let $f(x)=\sum_{\a\in\N^n} c_\a x^\a$ with coefficients $c_\a\in K$ and $c_0=0$.
Write $|\a| := \sum_i \a_i$. The power series has a positive radius of convergence, 
so there is $\rho>0$ such that $|c_\a| \rho^{|\a|}$ goes to $0$ for $|\a|\to \infty$.
Hence 
$\sup_{\a} |c_\a| \rho^{|\a|}$ 
is bounded, say by $B\ge 1$.
Choose $r\in R$ such $0 <|r| \le \rho/B$. 
Then $g(x) := f(rx)=\sum_{\a\in\N^n} d_\a x^\a$ with coefficients
$d_\a = c_\a r^{\a}$. 
We have $d_\a \in R$ since $d_0=0$ and 
for all $\a\ne0$, 
$$
 |d_\a| = |c_\a| |r|^{|\a|} \le  \frac{|c_\a| |\rho|^{|\a|}}{B^{|\a|}} 
  \le \frac{B}{B^{|\a|}} \le 1 .
$$ 
The assertion follows.
\end{proof}

\begin{lemma}\label{le:LC}
Suppose that $\psi\colon U\to V$ is a $K$--bianalytic map of open subsets of $K^n$. 
Then $x\mapsto (D_x\psi)(R^n)$ is a locally constant assignment of points to lattices. 
\end{lemma}

\begin{proof}
By a shift we may assume that $0\in U$ and $\psi(0)=0$. 
We shall apply a sequence of transformations to~$\psi$ to show that 
$x\mapsto (D_x\psi)(R^n)$ is locally constant around $0$. 
First, by Lemma~\ref{le:RK-ana}, we may assume that 
each component of $\psi$, locally around~$0$, is given by a power series 
with coefficients in $R$.  
Then $D_x\psi \in R^{n\times n}$ for all $x$ sufficiently close to $0$. 

We now proceed similarly as in \cite[Prop.~20]{KL:19}.
By the Smith normal form, there are $S,T\in\Gl_n(R)$ such that 
$S D_{0} \psi  T = \mathrm{diag}(s_1,\ldots,s_n)$ with  $s_i \in R$ nonzero.  
Thus we may assume that $D_{0}  \psi= \mathrm{diag}(s_1,\ldots,s_n)$ without loss of generality.  
We put $s:= s_1\cdots s_n$ and transform $\psi_j$ as follows: 
$$
 \widetilde{\psi}_j (x) := s_j^{-1}s^{-1} \psi_j(s x) .
$$
Then we have $\widetilde{\psi}(0)=0$ and $D_{0} \psi  = I_n$. 
Moreover, $\widetilde{\psi}_j$ is still given by a power series with coefficients in $R$ (convergent 
in a sufficiently small neighborhood of~$0$). For instance, the coefficients of the quadratic terms 
of $\widetilde{\psi_j}$ are obtained from those of $\psi_j$ by multiplication with 
$s_j^{-1}s^{-1}s^2 = s_j^{-1}s = \prod_{k\ne j} s_k$, which lies in $R$.

So we have $D_x\widetilde{\psi} \in R^{n\times n}$ for all $x$ sufficiently close to $0$ 
and $\det(D_0\widetilde{\psi}) = 1$. 
By continuity, we have $\det(D_x\widetilde{\psi}) \in R^{\times}$ for all $x$ in a neighbourhood  of $0$. 
This implies that 
$(D_x\psi)(R^n)=R^n$ for all~$x$ in a neighborhood  of $0$, and finishes the proof. 
\end{proof}

\subsubsection{$R$--structures on manifolds}

Throughout, $X$ stands for an $n$--dimensional $K$--analytic manifold. 

\begin{defi}\label{def:R-struct-mfs}
By an \emph{$R$--structure on $X$} we understand an assignment 
of a linear $R$--structure $\Lambda_x$ of $T_xX$ to each point $x\in X$ 
such that $x\mapsto \Lambda_x$ is locally constant, in the following sense:
for each $K$--analytic chart 
$\psi\colon U\to V$, with $U\subseteq X$ and $V\subseteq K^n$ being open subsets, 
and each $p\in U$, there is a linear $R$--structure $F_p$ of~$K^n$ such that 
$D_x\psi(\Lambda_x) =F_p$ for all $x$ in a neighborhood of~$p$.
By Lemma~\ref{le:LC}, it suffices to test the above property for the 
charts of a $K$--analytic atlas. 
\end{defi}

A similar notion appears in arithmetic geometry. If $k$ is a
  subfield of an algebraically closed field $K$, one can view a $k$-scheme
  as a $K$-scheme with a $k$-structure. See~\cite[Chapter~AG~11]{Borel}.
\begin{remark}One can define more generally the notion of $R$--structure on an analytic vector bundle $\pi:E\to X$. This will be an  an assignment 
of a linear $R$--structure $\Lambda_x$ of $E_x:=\pi^{-1}(x)$ to each point $x\in X$ 
such that $x\mapsto \Lambda_x$ is locally constant, in the following sense:
for each vector bundle $K$--analytic chart
$\varphi:E|_{U}\to U\times K^r$ and for each $p\in U$, there is a linear $R$--structure $L_p$ of~$K^r$ such that 
$\varphi|_{E_x}(\Lambda_x) =\{x\}\times F_p$ for all $x$ in a neighborhood of~$p$. (In the case of $E=TX$, this reduces to the above definition.) With this notation, an $R$--structure on a line bundle $\pi:E\to X$ gives a metric in the sense of \cite[Section 2.1.3]{CLT}. In \cite{CLT} the authors use the notion of metric on $\Lambda^n(T^*X)$ (the line bundle of top forms) to choose a volume form on $X$, which is essentially what we will do in \cref{se:volumeX} below.
\end{remark}
% A similar notion appears in the context of algebraic geometry where, for $k$ a subfield
% of an algebraically closed field $K$, one defines the notion of $k$-structures on $K$-schemes,
% see~\cite[Chapter~AG~11]{Borel}.

The  {\em standard $R$--structure} on $K^n$ 
(as a manifold, as opposed to as a vector space) is defined by taking $R^n$ 
as the linear $R$--structure in each tangent space $T_xK^n =K^n$.
The standard atlas consisting of the single chart with the identity map  
shows that the standard structure indeed defines an $R$--structure on $K^n$. 
More generally, any linear $R$--structure $V_R$ on a $K$-vector space~$V$ defines
an $R$--structure on~$V$, considered as a manifold, by the constant assignment $x \mapsto V_R$.

There is an equivalent way of defining $R$--structures on manifolds. 
For this we introduce the following. 

\begin{defi}\label{def:R-atlas}
  Let $\{(\psi_i,U_i)\}_{i\in I}$ be a $K$--analytic atlas of $X$  (see Definition~\ref{def:K-mf}).
  We say that the \emph{atlas $\{(\psi_i,U_i)\}_{i\in I}$ is $R$--compatible} if for all $i,j\in I$ and all 
  $x\in U_i\cap U_j$ the derivative of the transition map satisfies 
  $D_x(\psi_i \circ \psi_j^{-1}) \in \Gl_n(R)$.
\end{defi}

An $R$--compatible atlas distinguishes on each tangent space $T_xX$
an $R$--structure $\Lambda_x$ of $T_xX$ with the property that the linear maps
$D_x\varphi_i\colon T_xX \to K^n$ 
are defined over $R$, when $K^n$ is endowed with the
standard $R$--structure. Moreover, using  Lemma~\ref{le:LC}, 
we see that the assignment
$x\mapsto \Lambda_x$ is locally constant. 

\begin{lemma}\label{le:R-struct-atlas}
  Every $R$--structure on $X$ arises from an $R$--compatible atlas on $X$.
\end{lemma}

\begin{proof}
Let $x\mapsto \Lambda_x$ be an $R$--structure on $X$. 
By definition,
for each $K$--analytic chart 
$\psi\colon U\to V$, with $U\subseteq X$ and $V\subseteq K^n$ being open subsets, 
and each $p\in U$, there is an open neighborhood $U'\subseteq U$ of~$p$ such that 
$x\mapsto D_x\psi(\Lambda_x) $ is constant.
According to \cite[Lemma~1.4, Prop.~8.7]{schneiderp:11},  
it is possible to select a subcollection of the $(\psi,U')$ to arrive at 
a $K$--analytic atlas with pairwise disjoint domains $U'$. 
The resulting atlas is $R$-compatible since there is only one transition map (the identity). 
\end{proof}

\begin{remark}
Along the same lines, we could define the notion of an $R$--structure 
on a $K$--analytic vector bundle.
\end{remark}

We note that if $X$ is a  $K$--analytic manifold endowed with an $R$--structure, 
then every open subset $U$ of $X$ inherits from $X$ an $R$--structure 
in the obvious way. 

\begin{remark}\label{re:R-an-mf} 
There are many different $R$--structures on $K^n$. 
For instance, fix an integer $m\in \mathbb{Z}$ and consider the partition of $K^n$ into the open subsets
$U_i:= \{x\in K^n \mid \e^{i+1} <\|x\| \le \e^{i}\}$, for $i< m$, and we set 
$U_m:= \{x\in K^n \mid \|x\| \le \e^{m}\}$, 
with the charts 
$\psi_i\colon U_i \to (R^\times)^n\, x\mapsto \varpi^{-i} x$
(recall $|\varpi| = \e$). 
This is a $K$--analytic atlas which is $R$--compatible
(there are no change of coordinate maps).
It distinguishes at the points $x\in U_i$ the $R$--structure $\varpi^i R^n$.
\end{remark}

Suppose now $Y\subseteq X$ is a submanifold of $X$, see Definition~\ref{def:submf}.
The $R$--structure on $X$ induces an $R$--structure on every tangent space $T_xX$
and hence an $R$--structure $(T_xY)_R$ on every tangent space $T_xY$, for $x\in Y$.
It is obvious that this assignment is locally constant.

\begin{example}[Projective spaces]\label{ex:proj-space}
We endow the projective space $\proj^{n}(K)$ with an $R$--structure.
Hereby, the distinguished linear $R$--structure induced on the tangent space
$T_{[x]}\proj^n(K) = K^{n+1}/Kx$ coincides with the linear quotient $R$--structure (see Lemma~\ref{le:Rstruct}).
An equivalent description of this $R$--structure is as follows. 
We cover $\proj^{n}(K)$ with the open subsets
$U_i := \{[x_0:\ldots:x_n] \in \proj^{n}(K) \mid |x_i| = \max _j |x_j|\}$, 
for $0\le i \le n$, and use the charts $\psi_i\colon U_i \to R^n$, 
where, e.g., 
$$
 \psi_0\colon U_0 \to R^n, [x_0:\ldots:x_n] \mapsto (x_1/x_0,\ldots,x_n/x_0) .
$$
Clearly, $\varphi_0$ is a homeomorphism with inverse 
$\psi_0^{-1}\colon R^n \to U_0,\, (t_1,\ldots,t_n) \mapsto [1:t_1:\ldots:t_n]$.
To verify that the above atlas is $R$--compatible,
consider the change of coordinate map $\psi_1\circ \psi_0^{-1}$
on $\psi_0(U_0 \cap U_1) = R^\times \times R^{n-1}$ given by 
$$
 R^\times \times R^{n-1} \to R^\times \times R^{n-1},\, 
 (t_1,t_2,\ldots,t_n) \mapsto (1/t_1, t_2/t_1,\ldots,t_n/t_1) .
$$
Its derivatives indeed are linear transformations defined over $R$.
Let us finally show that the distinguished $R$--structure induced on the tangent space
$T_{[x]}\proj^n(K) = K^{n+1}/Kx$ coincides with the quotient $R$--structure. 
For this, consider the derivative 
$$
 D_t\psi_0^{-1}\colon K^n \to K^{n+1}/K(1,t_1,\ldots,t_n),\, 
 (\dot{t}_1,\ldots,\dot{t}_n) \mapsto [0,\dot{t}_1\ldots,\dot{t}_n] .
$$
Its image equals the image of $R^{n+1}$ under the canonical projection 
$K^{n+1} \to K^{n+1}/K(1,t_1,\ldots,t_n)$, since 
$0\times K^n$ is an $R$--complement of $K(1,t_1,\ldots,t_n)$. 
\end{example}

\begin{remark}\label{cor:alg-sets-R-struct} 
Let $X$ be an algebraic subset of $K^n$ or $\proj^n(K)$. 
We denote by $\Reg(X)$ its set of regular points.
The Implicit Function Theorem implies that $\Reg(X)$ is a 
$K$--analytic submanifold of $K^n$ or $\proj^n(K)$, respectively. 
We thus see that $\Reg(X)$ has a well defined induced $R$--structure.
\end{remark}
\begin{example}[Grassmannians]\label{ex:Grass}

Consider the Grassmann manifold $\Gr(r,n)$ of $r$--dimensional $K$--subspaces of $K^n$. 
Similarly to Example~\ref{ex:proj-space}, we shall endow $\Gr(r,n)$ with a
$K$--analytic atlas, which is $R$--compatible. 
For this, we first note that 
the matrix inversion map $\Gl_n(R)\to \Gl_n(R),\, M \to M^{-1}$
locally around $M_0\in \Gl_n(R)$ is given by power series with coefficients in~$R$. 
Indeed, 
\begin{equation}\label{eq:matrix-inversion}
 (M_0 - Z)^{-1} = M_0^{-1}  (1 -  ZM_0^{-1})^{-1}  
  =  M_0^{-1} \sum_{i\ge 0} (ZM_0^{-1})^i .
\end{equation}

We endow $\Gr(r,n)$ with an 
$R$--compatible atlas as follows. 
For $A\in R^{n\times r}$ and $I\subseteq [n]$ of cardinality~$r$, 
we denote by $A_I$ the square submatrix obtained from $A$ by selecting the rows indexed by $I$. 
We consider the closed and open subsets
$$
 \hat{U}_I := \big\{ A\in R^{n\times r} \mid  |\det(A_I)| = 1 \} 
$$
of $R^{n\times r}$ and the map 
$\hat{\psi}_I \colon \hat{U}_I \to R^{(n-r)\times r}$, 
which sends $A$ to $(A A_I^{-1})_{I^c}$, where $I^c$ denotes the complement of $I$ in $[n]$. 
For instance, for $I=[r]$, we have  
$$
 \hat{\psi}_I \colon \hat{U}_I \to R^{(n-r)\times r},\, 
  \begin{bmatrix} A_I \\ A_{I^c} \end{bmatrix}\mapsto A_{I^c} A_I^{-1} .
$$
Note that $A_{I^c} A_I^{-1}$ indeed has entries in $R$, since $A_I\in\Gl_r(R)$ 
by the definition of $\hat{U}_I$. 
Let $U_I$ denote the image of $\hat{U}_I$ under the canonical map~$\pi$, 
which sends $A$ to its column span. 
Then $\hat{\psi}_I$ factors via a map 
$\psi_I\colon U_I \to R^{(n-r)\times r}$. 
Moreover, $\psi_I$ is bijective: e.g., if $I=[r]$, the inverse of $\hat{U}_I$ is given by 
$$
 R^{(n-r)\times r} \to U_I,\, B \mapsto \mathrm{colspan} \begin{bmatrix} I_r \\ B \end{bmatrix} .
$$ 
It is easily checked that the usual topology on $\Gr(r,n)$ equals the quotient topology 
with respect to the map $\pi$, that the $U_I$ are open in $\Gr(r,n)$, and that 
$\psi_I$ is a homeomorphism. 

We claim that the family of $\psi_I$ forms an $R$--compatible 
$K$--analytic atlas for $\Gr(r,n)$. 
Indeed, matrix inversion $\Gl_r(R)\to \Gl_r(R),\, M \to M^{-1}$ is 
locally given by power series with coefficients in $R$; 
see~\eqref{eq:matrix-inversion}.
This implies that 
the change of coordinate map on $\psi_I(U_I \cap U_J)$ is given 
as well by power series with coefficients in $R$. 
It remains to show that the sets $U_I$ cover $\Gr(r,n)$. 
This follows from the claim below.
\end{example}

\begin{claim}
For $A\in R^{(r+n)\times r}$ of rank~$r$ there is a permutation matrix~$P$ 
and $Q \in\Gl_r(K)$ such that 
$$
 P A Q = \begin{bmatrix} I_r \\ B\end{bmatrix} ,
$$ 
where $B\in R^{n\times r}$ and $I_r$ denotes the unit matrix of format~$r$. 
\end{claim}

\begin{proof}
Let $S AT  = [D \ 0]^T$ be in Smith normal form, 
where $S \in \GL_{r+n}(R)$, $T\in \GL_r(R)$, and 
$D$~is invertible over $K$ since $A$ has full rank. 
The matrix $AT D^{-1} = S^{-1} (S AT) D^{-1} = S^{-1} [I_r\ 0]^T$ has its entries
in $R$ and at least one of its maximal minors is a unit. Thus there is a permutation matrix~$P$ 
such that 
$PAT D^{-1} = [X \ Y]^T$ with $X \in \GL_{r}(R)$ and $Y$ having entries in $R$. 
Then the factorization 
$P A Q  = [I_r\  B]^T$ 
with $Q= D^{-1}X^{-1}$ and $B= YX^{-1}$ is as desired.
\end{proof}

Finally, one can check that the $R$--structure on the tangent spaces of~$\Gr(r,n)$
resulting from the $R$--analytic structure is the expected one. 
Namely, for $A\in\Gr(r,n)$, 
$T_A \Gr(r,n) =\mathrm{Hom}_K(A,K^{n}/A)$
obtains the same $R$--structure as induced 
via tensor product from the induced $R$--structures on the subspace $A$ and 
the quotient space $K^{n}/A$, cf.~Lemma~\ref{le:Rstruct}.

\subsubsection{Connection to algebraic theory}

We now comment on the connection of our analytic theory of $R$--structures 
to the algebraic theory of schemes over $\Spec R$. 
To do so, we consider a refinement of
the definition of a $K$--analytic function by requiring the 
defining power series to have coefficients in $R$;
compare~\S\ref{se:analytic-maps}. 
The material in this subsection is not needed for the rest of the paper.

\begin{defi}\label{def:R-analytic}
  Let $f\colon U\to R^m$ be a function defined on an open subset 
  $U\subseteq R^n$. We call the function~$f$ {\em $R$--analytic} if, for every $x_0 \in U$, 
  there is a ball $B_{x_0}$ around $x_0$ of positive radius such that 
  each component of the function $f$ can be represented on $B_{x_0}$ 
  by a convergent power series with coefficients in $R$.
\end{defi}

\begin{example}\label{ex:matrix-inversion}
Let us verify that the matrix inversion map $\Gl_n(R)\to \Gl_n(R),\, M \to M^{-1}$ is $R$--analytic.
Indeed, for fixed $M_0\in \Gl_n(R)$, the power series expansion 
$$
\mbox{$ (M_0 - Z)^{-1} = M_0^{-1}  (1 -  ZM_0^{-1})^{-1}  
  =  M_0^{-1} \sum_{i\ge 0} (ZM_0^{-1})^i$} 
$$ 
has coefficients in $R$.
\end{example}

We now refine the definition of a $K$--analytic manifold.
This is achieved by the following modification of Definition~\ref{def:K-mf}: 
we require the images of the charts $\varphi_i$ to be open subsets of $R^n$ 
(instead of $K^n$) and we require the change of coordinates maps 
$\varphi_i \circ \varphi_j^{-1}$ to be $R$--analytic (instead of $K$--analytic) 
on their domain of definition $\varphi_j(U_i\cap U_j)$.

\begin{defi}\label{def:K-mf-R}
  An {\em $R$--analytic manifold $X$} of dimension~$n$ is a 
  Hausdorff topological space satisfying the second countability axiom, together with 
  an {\em $R$--analytic atlas}. 
  Such atlas consists of a family of {\em charts} 
  $\varphi_i\colon U_i \to V_i$, where the collection of $U_i$ forms an open cover of~$X$, 
  the $V_i$ are open subsets of $R^n$, and the $\varphi_i$ are homeomorphisms. 
  Moreover, it is required that every {\em change of coordinates map} 
  $\varphi_i \circ \varphi_j^{-1}$ is $R$--analytic on its domain of 
  definition $\varphi_j(U_i\cap U_j)$.
\end{defi}

Clearly, an $R$--analytic manifold is automatically a $K$--analytic manifold with an $R$--structure; 
the change of coordinate functions define an $R$-compatible atlas and each chart $V_i \subseteq R^n$ 
carries the standard $R$--structure. Note however that the converse is not necessarily true.

We now explain the relation of smooth schemes over $\Spec R$ to $R$--analytic manifolds.

\begin{prop} \label{prop: algebraic R-structure}
Let $X$ be a smooth equidimensional scheme over $\Spec R$ of finite type. 
Then the set of points $X(R)$ is canonically endowed with the structure of an $R$--analytic manifold. 
In particular, $X(R)$ is a $K$--analytic manifold with $R$--structure. 
\end{prop}

\begin{proof}
Let $x \in X(R)$; the case $X(R) =\emptyset$ being trivial.
Because $X$ is a smooth scheme over $R$, the tangent bundle at $x$ 
has rank equal to the local dimension $\dim_x X$ at $x$ and 
there is an affine neighbourhood $U \subseteq X$ containing $x$, 
where $X$ is locally defined by a complete intersection. 	
Since $U$ is an affine scheme over $R$ of finite type, there is a polynomial ring $R[x_1, \ldots, x_n]$ 
such that the coordinate ring $\mathcal{O}_U(U)$ is the localization of a quotient of $R$. 
Furthermore, $U$ is equipped with the structure morphism $U \rightarrow \Spec R$, 
so there is a canonical inclusion $R \hookrightarrow \mathcal{O}_U(U)$.
Let $F$ be a choice of $n - \dim_x X$ local defining equations for $X \cap U$. 
By smoothness over $R$, the rank of $D_x F$ modulo $\varpi$ equals $\dim_x X$; in other words, 
the singular values of $D_xF$ are all equal to $1$.

Each connected component of $X$ has the same dimension,
so we see that $X(R)$ carries the structure of a $K$--analytic manifold (see~\cite[Prop.~20]{KL:19}). 
We show that the transition functions defining the scheme $X$ define an $R$--analytic structure for $X(R)$. 
	
Let $U' \subseteq U$ be an affine subscheme of $U$ over $R$ and let $f\: U' \rightarrow V$ be a biregular gluing map 
between two affine charts over $R$. For any $p \in U'(R)$, which we think of as a section $p\: \Spec R \rightarrow U'$ 
of the structure map, we have that $f$ is regular along the image of $p$, so we may represent the components of $f$ by
\[
  f_j = \frac{h_j(x)}{g_j(x)} \in \mathcal{O}_{U', p} ,
\]
where $g_j, h_j \in R[x_1, \ldots, x_n]$ such that $g_j(x)$ is regular and nonvanishing along $p$. 
In particular, $g_j(x)$ does not vanish at the closed point $p(\mathfrak{m})$, 
where $\mathfrak{m}$ is the maximal ideal of $R$, so the image of $g_j(x)$ 
in the residue field of $\mathcal{O}_{U', p(\mathfrak{m})}$ is a unit. 
Thus, $p^*g_j$ is a unit in $R$, and so $g_j(x)^{-1}$ can be represented by a power series with coefficients in $R$ 
with positive radius of convergence near $p$. In other words, the transition function $f$ is locally given 
by $R$--analytic functions. Furthermore, because $f$ is a biregular morphism over $R$, 
we see that $f$ is (locally) $R$-bianalytic. 
\end{proof}

\subsubsection{The Weil canonical volume}

In \cite{weilAdeles}, Weil showed how to construct a canonical volume for a scheme $X$ that is smooth over $R$. 
First, assume that $X$ is affine and that $X$ carries a form $\omega \in \operatorname{H}^0(X, \Omega_{X/R}^n)$ 
such that $\omega$ is nonvanishing at every point of $X$. Such a differential form is called a \emph{gauge form}. 

Let $p \in X(R)$.
We may express $\omega = f(x) dx_1 \wedge \ldots \wedge dx_d$ locally near $p$. 
Since $\omega$ is nonvanishing at each point of $X$, including the closed points, we have that $f(p) \in R^\times$.  
Next, we consider a $K$--analytic isomorphism $\psi\: D \rightarrow U$ from a neighbourhood of the origin 
to an analytic open subset $U$ of $X$ containing $p$.
Then $\psi^*\omega =  (f\circ\psi) \cdot |J\psi| \cdot y_1\wedge \ldots \wedge dy_d$. 
We may further assume that $|J\psi| = 1$ on $D$. 
Note that $f(U) \subseteq R^\times$. 
The \emph{Weil canonical measure} is defined to be the measure defined by these local differentials on the local disks $D$. 
(One must check that the measure is independent of the choice of local coordinates near $p$: e.g.,  
this is verified in \cite{weilAdeles}.) 
Thus, 
\[
    \int_D dy_1 \wedge \ldots \wedge dy_d = \int_D \psi^*\omega  = \int_U \omega.
\]
The change of variables theorem shows that a different choice of analytic chart $\psi'\: D' \rightarrow U$ 
does not change the value of the integral. 

If $X$ is a smooth scheme over $R$, 
then Proposition~\ref{prop: algebraic R-structure} shows that $X(R)$ is a $K$--analytic manifold with an $R$--structure. 
In particular, this $R$--structure locally defines a gauge form on~$X(R)$.
We have shown the following.

\begin{prop}\label{prop:weil}
  Let $X$ be a smooth equidimensional scheme of finite type over $R$. Then the Weil canonical volume of $X$ is equal to the volume of $X(R)$ 
  as a $K$--manifold with $R$--structure.
\end{prop}

{
\subsection{Volume form, absolute Jacobian and the coarea formula}\label{se:volumeX}

Suppose that $X$ is a $K$--analytic manifold of dimension~$n$ 
endowed with an $R$--structure.
Then the space of alternating $n$ forms on $T_xX$ is one dimensional and 
has an induced $R$--structure (see Lemma~\ref{le:Rstruct}); let $\omega_X(x)$ denote a generator.  

\begin{defi}
For a $K$--analytic manifold $X$ with $R$--structure, we define the 
associated {\em volume form} $\Omega_X(x)$ as the absolute value of $\omega_X(x)$. 
\end{defi}

Note that  $\Omega_X$ is well defined since $\omega_X(x)$ is determined up to a multiplication by 
a unit of~$R$. Moreover, $\Omega_X$ depends continuously on $x$. 
We can also express it as 
\begin{equation}\label{eq:defOmega}
 \Omega_X = |\sigma_1\wedge \ldots \wedge \sigma_n| ,
\end{equation}
where $\sigma_1,\ldots,\sigma_n$ is an $R$--basis of the dual space $(T_x^*X)_R$. 
Thus we have a well defined integral $\int_X f \Omega_X$ for 
a measurable function $f\colon X\to [0,\infty]$, see Section~\ref{se:vol-int}.
We define the volume of $X$  by the integral
\begin{equation}\label{eq:defvol}
 |X| : = \int_X \Omega_X.
\end{equation}

If $X$ is a projective algebraic set, 
then $\Reg(X)$ is a $K$--analytic manifold endowed with an $R$--structure,
see Remark~\ref{cor:alg-sets-R-struct}. 
We then define $\int_X f\, \Omega_X$ as the integral 
of the restriction of $f$ over $\Reg(X)$. 
In particular, if $X$ is compact, then it has a well defined volume $|X|$ 
(note that $\Reg(X)$ is compact as well). 

In \cite{KL:19}, a different way of defining the volumes of 
algebraic subsets $X$ of $R^n$ and $\proj^n(K)$ was used 
(for $R=\Z_p$ and $K=\Q_p$).  
However, it turns out that this leads to the same values as here, 
see Proposition~\ref{pro:vol-consistent} below. 

Let $X$ and $Y$ be $K$--analytic manifolds endowed with $R$--structures
and let $\varphi\colon X\to Y$ be a $K$--analytic map.
We define the {\em absolute Jacobian} of $\varphi$ at $x\in X$, 
\begin{equation}\label{eq:AJ-nl}
  J (\varphi)(x) := N(D_x\varphi) ,
\end{equation}
as the absolute determinant of the $K$-linear map 
$D_x\varphi\colon T_x X \to T_{\varphi(x)}$. 

\subsubsection{Coarea formula}

Suppose now $n=\dim X\ge m=\dim Y$ and let 
$C(\varphi) := \{ x \in X \mid \rk (D_x\varphi) < n\}$  
denote the set of critical points of $\varphi$.
Theorem~\ref{th:sard} (Sard's lemma) states that 
the set $\varphi(C(\varphi))$ of critical values has measure zero. 
Suppose now that $y$ is a regular value of $\varphi$, i.e., 
$y\not\in \varphi(C(\varphi))$. 
If the fiber $\varphi^{-1}(y)$ is nonempty, then it is an 
$K$--analytic submanifold of $X$ having dimension $n-m$. 
It is endowed with the $R$--structure inherited from~$X$.
Then the integral 
$\int_{\varphi^{-1}(y)} h \Omega_{\varphi^{-1}(y)}$ 
is well defined for a measurable function $h\colon X\to [0,\infty]$. 
If $\varphi^{-1}(y) = \varnothing$ we interpret the integral as $0$.

The coarea formula is a crucial tool of geometric measure theory going back 
to Federer~\cite{federer:59}, 
see \cite[Thm.~III.5.2, p.~138]{chavel:06} for a comprehensive account and 
\cite[p.~159]{chavel:06} for its smooth version. 
A version of the coarea formula also holds 
in the nonarchimedean setting. The following is a variation of \cite[Theorem 3.4]{AA}. (We give a proof of this result using our framework  in the appendix, see Section \ref{sec:proofcoarea}.)

\begin{thm}\label{th:coarea}
Let $X$ and $Y$ be $K$--analytic manifolds  of dimensions~$n\ge m$, 
endowed with $R$--structures, 
and let $\varphi\colon X\to Y$ be a $K$--analytic map.
Further, assume 
$h\colon X\to [0,\infty]$ is measurable. Then
$$
  \int_X h\, J(\varphi)\, \Omega_X 
  = \int_{y\in Y} \Big(\int_{\varphi^{-1}(y)} h\, \Omega_{\varphi^{-1}(y)} \Big) \Omega_Y,
$$
where each $\varphi^{-1}(y)$ carries the $R$--structure induced from $X$. 
By Sard's lemma, the right hand integral can be thought to be 
over the regular values of $\varphi$ only.
\end{thm}

\begin{example}[The volume of projective spaces]\label{example:volproj}As an application, let us determine the volume of projective spaces with the quotient $R$--structure.
Consider 
$\varphi\colon R^{n+1}\setminus 0 \to \proj^n(K),\, x\mapsto [x]$. 
We claim that $J(\varphi)(x) = \|x\|^{-n}$. To see this, assume
without loss of generality 
$x_0=(r,0,\ldots,0)$ and note that the derivative of $\varphi$ at $x_0$ 
is given by $R^{n+1}\to R^n,\, (v_0,v_1,\ldots,v_n) \mapsto |r|^{-1}(v_1,\ldots,v_n)$,
when identifying $T_{x_0}R^{n+1}=R^{n+1}$ and $T_{\varphi(x_0)}\proj^n(K)=R^n$. 
Also note that the fiber over $\varphi(x_0)$ is isomorphic to $R \backslash \{0\}$. 
Applying Theorem~\ref{th:coarea} with $h(x) := \|x\|^n$ thus yields 
$$
 1 = |R^{n+1} \setminus \{0\}| =  |\proj^n(K)| \int_{R \backslash \{0\}}|r|^n d\mu(r)
  = |\proj^n(K)|\cdot \frac{1-\e}{1-\e^{n+1}} ,
$$
where we used Proposition~\ref{prop:pushfHaarII} for the last equality. 
We obtain
\begin{equation}\label{eq:vol-proj}
 |\proj^n(K)| = \frac{1-\e^{n+1}} {1-\e} .
\end{equation}
Alternatively, we may look at the restriction $\psi\colon S(K^{n+1}) \to \proj^n(K)$ 
to the unit sphere $S(K^{n+1})$, see Definition~\ref{def:sphere}. 
Using that $J(\psi)=1$, we obtain from Theorem~\ref{th:coarea} that 
$\mu(\varphi^{-1}(A)) = \mu(S^1)\, \mu(A)$
for a measurable subset $A\subseteq \proj^n$.
This is exactly the definition of the measure on $\proj^n$ provided in~\cite{KL:19}.
\end{example}

We complement the coarea formula with a useful result on computing 
the volume of images.

\begin{thm}\label{th:vol-image}
  Let $X$ and $Z$ be $K$--analytic manifolds endowed with $R$--structures, such that 
  $\dim(X)\leq \dim(Z)$ and let 
  $\varphi\colon X\to Z$ be a $K$--analytic map.
  Further, 
  we assume $Y\subseteq Z$ is a
  $K$--analytic submanifold
  of dimension $\dim(X)$
  such that $\varphi(X)\subseteq Y$. 
  Then for any measurable subset $U\subseteq X$, 
  we have
  $$
  \int_{y\in Y} \# \big( U \cap\varphi^{-1}(y)\big)\, \Omega_Y(y) = \int_U J(\varphi)\, \Omega_X .
  $$
  In particular, if $\varphi$ is injective, then the volume of the image $\varphi(X)$ in $Y$ satisfies 
  $$
  |\varphi(X)|  = \int_X J(\varphi)\, \Omega_X .
  $$
\end{thm}

\begin{proof}
Clearly, $\widetilde{\varphi}\colon X\to Y,\, x\mapsto\varphi(x)$ 
is a $K$--analytic map. Moreover, by the definition of the absolute Jacobian 
(see \eqref{eq:abs-det-inv} and \eqref{eq:AJ-nl}),  
we have $J(\varphi)= J(\widetilde{\varphi})$. 
Applying Theorem~\ref{th:coarea} to $\widetilde{\varphi}$ 
and taking for $h$ the indicator function of $U$ yields the assertion.
\end{proof}

The following corollary will be useful for the computation of the volume of Schubert varieties, but it has an independent interest.
\begin{cor}\label{cor:vol-graph} .
Let $X$ and $Y$ be $K$--analytic manifolds endowed with $R$--structures, 
and let $\varphi\colon X\to Y$ be a $K$--analytic map with the property that its 
derivatives are defined over $R$. Endow $X\times Y$ with the product $R$--structure.
Then the graph of $\varphi$ has  
the same volume as $X$. 
\end{cor}

\begin{proof}
Denoting by $\mathcal{G}$ the graph of $\varphi$, we see that $\mathcal{G}$ is a $K$--analytic submanifold of $X\times Y$. 
Consider 
the $K$--bianalytic map 
$\psi\colon X\to \mathcal{G}, \; x\mapsto (x,\varphi(x))$. Its 
derivatives are given by 
$$
 D_x\psi \colon T _xX \to T_{(x,\varphi(x))}\mathcal{G}, \; \dot{x}\mapsto (\dot{x},D_x\varphi(\dot{x})) .
$$ 
Since the linear map $D_x\varphi\colon T_xX \to T_{\varphi(x)}$ is defined over $R$, 
it follows that $J(\psi)=1$. Therefore, 
the second part of Theorem~\ref{th:vol-image}
implies that $|\mathcal{G}| =|X|$. 
\end{proof}

}

%%% SUBSECTION

\subsubsection{Some volumes and $q$-binomial coefficients}
We show here that the volume of $\GL_n(R)$ and 
the Grassmannians $G(k,n)$ can be expressed in terms of 
familiar quantities from combinatorics. 
For a parameter $q$ one defines the
{\em $q$-bracket},
the
{\em $q$-factorial}, and the
{\em $q$-binomial coefficient (or Gaussian binomial) coefficient} 
as follows:  
\[
  [i]_q := \sum_{j=0}^{i-1} q^j,
  \quad
  [k]_q! := \prod_{i=1}^k [i]_q,
  \quad
  \begin{bmatrix} n \\k\end{bmatrix}_q :=
  \frac{[n]_q!}{[k]_q! \ [n-k]_q!} . 
\]
It is well known that 
$\left[\begin{smallmatrix} n \\k\end{smallmatrix} \right]_q$
counts the number of $k$-dimensional subspaces of $\F_q^n$, 
when $q$ is a prime power. We also note the asymptotics
$[i]_q \to i$, 
$[k]_q! \to k!$, and  
$\left[\begin{smallmatrix} n \\k\end{smallmatrix} \right]_q \to {n \choose k}$ 
for $q\to 1$.

Recall that in our setting, the prime power $q$ denotes 
the cardinality of the residue class field 
$R/\frm\simeq\F_q$ and we have set $\e=q^{-1}$.

\begin{prop}\label{pro:vols}
  We have
  \[
    \begin{gathered} 
      \gamma_n :=|\Gl_{n}(R)| = \prod_{i=1}^n (1-\e^i) = (q-1)^n q^{-\frac{n(n+1)}{2}}\; [n]_q!, \quad \text{and} \\
      |G(k,n)| = \frac{\g_n}{\g_k \g_{n-k}} = q^{-k(n-k)} \; \begin{bmatrix} n \\k\end{bmatrix}_q .
    \end{gathered}
  \]
\end{prop}

\begin{proof}
The group $\Gl_{n+1}(R)$ acts transitively on $\proj^n(K)$. 
Let $x_0 :=[1:0\ldots:0]$, and note that the fibre of the orbit map 
$\varphi\colon\Gl_{n+1}(R) \times \{x_0\} \to \proj^n(K)$ is given by the stabilizer group
\begin{equation*}\label{eq:fiberproj}
 \varphi^{-1}(x_0) = \Big\{
  \begin{pmatrix} r & b \\0 & A \end{pmatrix} \mid r\in R^\times, A\in\Gl_n(R), b\in R^n \Big\} ,
\end{equation*}
which has the volume $(1-\e)|\Gl_n(R)|$. 
It is easily checked that $J(\varphi)=1$. Theorem~\ref{th:coarea} thus implies the recursive formula
$$
 |\Gl_{n+1}(R)| = |\proj^n(K)| \cdot (1-\e) \cdot |\Gl_n(R)| .
$$
Resolving this with Equation~\eqref{eq:vol-proj} yields 
$|\Gl_{n}(R)| = \prod_{i=1}^n (1-\e^i)$.  
The stated expression in terms of $[n]_q!$ is straightforward to verify.

As for $\Gr(k,n)$, 
we have a transitive action of $\Gl_n(R)$ on $\Gr(k,n)$. Its stabilizer at 
$K^k \times 0^{n-k}$ is given by 
$$
 \Big\{ \begin{pmatrix} A & C\\0 & B \end{pmatrix} \mid A\in\Gl_k(R), B\in \Gl_{n-k}(R), C\in R^{k\times (n-k)}\Big\} ,
$$
which has the volume $|\Gl_k(R)| \cdot |\Gl_{n-k}(R)|$. 
The absolute Jacobians of the orbit maps are easily checked to be one. 
Theorem~\ref{th:coarea} and part one imply the claim.
\end{proof}

\begin{remark}
The reduction 
$R^{n\times n} \to \F_q^{n\times n}$ modulo $\frm$  
maps $\GL_n(R)$ to $\GL_{n}(\F_q)$.  
Since the pushforward of the uniform distribution on $R^{n\times n}$ 
under the reduction equals the uniform distribution on $\F_q^{n\times n}$, 
we get 
\begin{equation*}\label{eq:gn-komb}
 |\Gl_n(R)| =  q^{-n^2}\#G_{n}(\F_q) . 
\end{equation*}
The resulting formula for $\#G_{n}(\F_q)$ following from Proposition~\ref{pro:vols}
is well known. 

Let 
$G_{\F_q}(k,n)$ denote the Grassmannian 
defined over the finite field $\F_q$. 
Based on the reduction map 
$G(k,n) \to G_{\F_q}(k,n)$,  
one shows that 
\begin{equation*}\label{eq:Grass-komb}
 |G(k,n)| = 
  q^{-k(n-k)} \# G_{\F_q}(k,n) . 
\end{equation*}
Together with Proposition~\ref{pro:vols}, 
this implies the well known fact 
$\# G_{\F_q}(k,n) =\left[\begin{smallmatrix} n \\k\end{smallmatrix} \right]_q$.
\end{remark}
\subsubsection{Equivalent definitions of the volume}\label{sec:volumeequi}
We outline how to show that our definition~\eqref{eq:defvol}
of volume is consistent with the one 
from~\cite{KL:19, oesterle:82}, which was given in the case $K=\mathbb{Q}_p$ 
for an $m$--dimensional algebraic subset $X$ of $R^n$.
Let us first recall the definition of volume from~\cite{KL:19}.
For $k\in\N$, we consider the reduction homomorphism  
$\pi_k\colon R^n \to (R/\frm^k)^n$ and recall that $(R/\frm^k)^n$ 
has cardinality $q^{kn}$. 
The following definition of the $m$--dimensional volume of an open subset~$X'$ of~$X$
was used in~\cite{KL:19}: 
\begin{equation}\label{eq:def-KL-vol}
 \vol_m(X') := \lim_{k\to\infty} \frac{\# \pi_k(X')}{q^{km}} .
\end{equation}
%{\color{blue}
The next result shows that $U\mapsto \mathrm{vol}_m(U)$ indeed
defines a Borel measure on $X$, which agrees with $U\mapsto \int_U\Omega_X$.
%}

\begin{prop}\label{pro:vol-consistent}
We have $\vol_m(X') = |X'|$.
\end{prop}

\begin{proof} (Sketch)
We first note that the assertion is obvious in the linear case, where $X=R^n$. 
The strategy is to reduce to the linear case.

By Remark~\ref{cor:alg-sets-R-struct},
$\Reg(X)$ is a $R$--analytic submanifold 
of $R^n$. Removing lower dimensional strata neither affects the value of $\vol_m(X)$ 
nor of $|X|$, so that we may replace $X$ by $\Reg(X)$.
It suffices to show that for every $p\in X$, there is a sufficiently small neighborhood~$U$ of $p$ in $X$, 
such that \eqref{eq:def-KL-vol} holds for $U$. The reason is that it is known that 
every covering of $R^n$ by open subsets can be refined to a {\em disjoint} 
covering of $R^n$ by open subsets; see \cite[Lemma~1.4]{schneiderp:11}. 

Let $L$ be an $R$--complement of $T:=(T_pX)_R$ in $R^n$ 
and consider the projection  
$\psi\colon R^n \to T$ along $L$.
By \cite[Prop.~20]{KL:19}, there is an open neighborhood $U$ of $p$ in $X$ 
such that $\psi$ maps $U$ isometrically onto an open neighborhood $V$ of $0$ in $T$. 
We apply the coarea formula (Theorem~\ref{th:coarea}) to the 
restriction $\widetilde{\psi}\colon U \to V$ of $\psi$. 
Note that $N(\widetilde{\psi})=1$ since we project along an $R$--complement.
The coarea formula implies $|X| = |V|$. 

On the other hand, it is easy to check from 
the counting definition~\eqref{eq:def-KL-vol} that 
$\vol_m(U) = \vol_m(V)$. 
We therefore have reduced the claim to the case of an open subset 
of $(T_pX)_R\simeq R^m$ and we are done.
\end{proof}

\begin{remark}
Proposition~\ref{pro:vol-consistent} also follows, rather indirectly, from the
two integral geometric formulas proved in~\cite[Prop.~36]{KL:19}
and in this paper (Corollary~\ref{cor:KL}). More specifically, 
let $X$ be an open subset of $\proj^m(K)$. Then these integral 
geometry formulas express the $m$--dimensional volumes $\vol_m(X)$ and $|X|$, respectively, 
in terms of the expected number of intersection points of $X$ with a random projective linear 
subspace of complementary dimension.
\end{remark}

\section{Relative position between two subspaces}\label{se:rel-pos}

We first state the minimax characterization of singular values. 
Then we use this to define the position vector of a pair of subspaces of $K^n$, 
which is an analogue of the notion of principal angles of two real subspaces of $\R^n$
(or complex subspaces of $\C^n$), which characterizes their relative position.
%We define an analogue of the notion of principal angles of two real subspaces of $\R^n$ 
%(or complex subspaces of $\C^n$), which characterizes their relative position. 
A quantity derived from this 
%As it will become clear later, turning the relative position betweentwo subspaces into a quantitative object is 
is an indispensable ingredient for the general integral geometry
formula that we will discuss later.

In Section~\ref{sec:jointrandom}, we determine the probability distribution 
of the position vector of a uniformly random pair of subspaces.
This will be essential for the determination of the volume of special Schubert varieties 
in Section~\ref{se:Vol-Schubert}.

Throughout we denote by $V$ an $n$-dimensional $K$-vector space endowed with a linear $R$--structure.

%aaa
\subsection{Minimax characterization of singular values}\label{sec:minmax}

The well known characterization of the singular values 
by minimax properties~\cite{golub-van-loan:13} over $\R$ or $\C$ carries 
over to the nonarchimedean setting.
This appears to be a little-known piece of folklore: the only
  reference in this direction we could locate is \cite[\S 4.4]{kedlaya:22}.
Since we will need it later on, we give a proof of this fact of independent interest.

Let $V$ and $W$ be $K$-vector spaces endowed with a linear $R$--structure 
and $\varphi\colon V\to W$ be a $K$-linear map. 
Put $m :=\dim V$ and $n := \dim W$.   

\begin{prop}\label{pro:minimax}
  Let $\varphi\colon V\to W$ be a $K$-linear map 
  of $K$-vector spaces endowed with a linear $R$--structure.  
  The singular values $\s_i$ of $\varphi$ satisfy for all $k$:
  \begin{equation*}
    \arraycolsep=0pt
    \begin{array}{rccl}
      \sigma_k  \ &= & {\displaystyle \max_{U\subseteq V \atop \dim U = k}}  &\min_{v\in U, \|v\|=1} \|\varphi(v)\| , \\[3ex]
      \sigma_k  \ &= & {\displaystyle \min_{U\subseteq V \atop \codim U= k-1}}  &\max_{v\in U, \|v\|=1} \|\varphi(v)\| .
    \end{array}
  \end{equation*}
\end{prop}

\begin{proof}
The proof is completely analogous to the one over~$\R$. 
Without loss of generality we may assume that 
$V=K^m, W=K^n$ and $\varphi(v)=(s_1 v_1,\ldots)$ 
where $\s_i=|s_i|$ and the list is appended by zeros if $m < n$. 

For the first assertion, let $U\subseteq V$ with $\dim U =k$. Then 
$U\cap (0\times K^{n-k+1}) \ne 0$. Hence there exists 
$v'\in U\cap (0\times K^{n-k+1})$ with $\|v'\|=1$, and we have 
$$
  \min_{v\in U, \|v\|=1} \|\varphi(v)\| \le \|\varphi(v')\| = \max_{i \ge k} |s_i| |v'_i| \le  \s_k .
$$ 
Since $U$ was arbitrary, we get 
$$
 \max_{U}    \min_{v\in U, \|v\|=1} \|\varphi(v)\| \le \s_k .
$$
Choosing $U= K^k$ shows that equality holds.

For the second assertion, let $U\subseteq V$ with $\codim (U) =k-1$. Then 
$U\cap (K^{k} \times 0) \ne 0$. 
Take 
$v'\in U\cap (K^{k} \times 0)$, then
\[
  \|\varphi(v')\| = \max_{i \le k} |\s_i| |v'_i| \ge  \s_k \max_i |v'_i| = \s_k \, \|v'\|.
\]
Thus,
$\max\{ \|\varphi(v)\| \ |\ v\in U, \|v\|=1 \} \geq \s_k$. 

Since $U$ was arbitrary, we get 
$$
 \min_{U}    \max_{v\in U, \|v\|=1} \|\varphi(v)\| \ge \s_k .
$$
Choosing $U= 0\times K^{n-k+1}$ shows that equality holds.
\end{proof}

As a consequence, we obtain the interlacing properties 
for singular values, analogous to the situation over $\R$ and $\C$,
see \cite{thompson:72}.

\begin{prop}\label{prop:interlacing}
  Let $\varphi\colon V\to W$ be as in Proposition~\ref{pro:minimax}, 
  $V_1\subseteq V$ be a subspace of codimension~$d$ and  
  $W_1\subseteq W$ be a subspace of dimension~$d$. 
  Let $\s'_k$ denote the singular values of the restriction 
  of $\varphi$ to $V_1$, and let $\s''_k$ denote the 
  singular values of the composition of $\varphi$ 
  with the canonical projection $W\to W/W_1$.  
  Then $\s_k'$ and $\s_k''$ satisfy the interlacing property
  $$
  \s_k \ge \s'_k \ge \s_{k+d} ,\quad \s_k \ge \s''_k\ge \s_{k+d} .
  $$
\end{prop}

\begin{proof}
By duality, it suffices to prove the first claim. 
The claim $\s_k \ge \s'_k$ follows immediately from 
the first characterization of Proposition~\ref{prop:interlacing}. 
In order to prove that $\sigma_k'\geq \sigma_{k+d}$, let $U\subseteq V$ with $\dim U = k+d$ be such that 
$\s_{k+d} = \min_{v\in U, \|v\|=1} \|\varphi(v)\|$. 
Then $\dim(U\cap V_1) \ge k$, hence there is a $k$-dimensional subspace 
$U_1 \subseteq U\cap V_1$. Now note 
$\s_{k+d} \le \min_{v\in U_1, \|v\|=1} \|\varphi(v)\| \le \s'_k$.
\end{proof}

A more familiar way to express the interlacing property is the following immediate consequence.

\begin{cor}\label{cor:interlacing}
Let $A\in K^{m\times n}$ with $m\le n$ and 
$B$ be a submatrix of $A$ obtained by deleting a row and a column. Then 
singular values $\tau_i$ of $B$ satisfy the interlacing property
$$
 \s_1 \ge \tau_1 \ge \s_2 \ge \tau_2\ge \ldots\ge \s_{m-1} \ge \tau_{m-1} \ge \s_m .
$$
\end{cor}

%bbb

\subsection{The position vector of a pair of subspaces}

We begin by defining an analogue of the notion of principal angles of subspaces. 
The development is parallel to the situation over $\R$ and $\C$; surprisingly
we were unable to locate any reference in the literature.
We start with the following lemma.

\begin{lemma}\label{le:sv-addmap}
Let $E,F\subseteq V$ be $K$-subspaces of dimensions $k$ and $\ell$, 
respectively, such that $k+\ell \le n$, assume $k\le \ell$, and put $d:=\dim(E\cap F)$. 
We consider the addition map 
$\varphi\colon E\times F \to V,\, (e,f)\mapsto e+f$. 
Then the decreasing list of singular values of $\varphi$ has the shape 
$$(\sigma_1, \ldots, \sigma_{\ell+k})=(1, \ldots, 1, s_1, \ldots, s_k)=(1,\ldots,1,s_1,\ldots,s_{k-d},0,\ldots,0),$$
i.e., it starts with $\ell$ occurrences of $1$ and ends with $d$ occurrences of~$0$. 
\end{lemma}

\begin{proof}
  
We apply the minimax characterizations in Proposition~\ref{pro:minimax} to~$\varphi$.
That the $\ell$ largest singular values equal $1$ follows from the first characterization  
noting that the restriction of $\varphi$ defines an isometric map $\{0\}\times F \to V$.
That the list of singular values ends with $d$ zeros follows from the second characterization 
noting that $\varphi$ vanishes on $U=\{(x,-x) \mid x \in E\cap F \}$. 
\end{proof}

We shall see that the relative position of the subspaces $E, F\subseteq V$ 
is described by the values~$s_i\in [0,1]$ from
Lemma~\ref{le:sv-addmap}
(i.e., the singular values of the addition map, excluding the first $\ell$ of them).
In order to emphasize that these are discrete parameters, we encode 
$s_i = \e^{x_i}$ with $x_i\in\bN$: then $(x_1,\ldots,x_k)$ is an element of 
$$
% \Nko := \{ x\in\bN^k \mid x_1 \ge \ldots \ge x_ k\} .
 \Nko := \{ x\in\bN^k \mid x_1 \le \ldots \le x_ k\} .
$$
We recall that $\varpi$ denotes a generator of the maximal ideal $\frm$ of $R$ 
and $\e:=|\varpi|$.  

\begin{defi}\label{def:PP1}
Let $E, F\subseteq V$ be subspaces of dimensions $k$ and $\ell$, 
respectively, such that  $k+\ell \le n$ and assume $k\le \ell$. 
The {\em (relative) position vector} of the pair $(E, F)$ 
is defined as the increasing list of nonnegative integers 
$(x_1,\ldots,x_k)=(x_1,\ldots, x_{d-k}, \infty, \ldots, \infty) \in\Nko$, 
where the $s_i=\e^{x_i}$ are the last $k$ decreasing singular values from Lemma~\ref{le:sv-addmap}.
\end{defi}

Later on, in Definition~\ref{def:PP2}, 
we will generalize this, dropping the assumption $k+\ell \le n$. 
We call the components of the position vector the 
{\em position parameters} of $(E,F)$ and view them
as an analogue of the notion of principal angles of two real or complex subspaces. 
The justification for this will be given in Theorem~\ref{th:NF-pairs-subspaces}. 
Let us first compute the position vector in a particular situation. 

\begin{lemma}\label{le:NF-PP}
Assume $k\le\ell$ and $k+\ell \le n$. 
Further, let 
$$
 e_1,\ldots,e_\ell, f_1,\ldots,f_k, g_{1},\ldots,g_{n-k-\ell} 
$$   
be an $R$-basis of $K^n$ and let $(x_1,\ldots,x_k) \in \bN^k_o$. 
Then the pair of subspaces 
$$
 E= \spann\{ e_1+ \varpi^{x_1} f_1,\ldots, e_k + \varpi^{x_k} f_k\}, \quad
 F=\spann\{e_1,\ldots,e_{\ell}\}
$$ 
has the position vector $(x_1,\ldots,x_k)$.
Moreover, 
$\dim (E\cap F) =\#\{ i \mid x_i =\infty\}$.
\end{lemma}

\begin{proof}
The addition map 
$\varphi\colon E\times F \to V$ 
has the following matrix representation with respect to the given bases:
$$
 \begin{bmatrix} I_k & 0  & I_k \\
                           0 &  I_{\ell-k} & 0\\
                         D_x& 0  & 0 \\
                          0 & 0 & 0\end{bmatrix} ,
$$ 
where $D= \diag(\varpi^{x_1},\ldots,\varpi^{x_k})$.
Via elementary column and row transformation, this can be 
brought to the block diagonal form
$$
 \begin{bmatrix} I_k & 0   & I_k \\
                           0 & I_{\ell-k} & 0  \\
                         0 & 0  & D_x \\ 
                          0 & 0 & 0\end{bmatrix} ,
$$ 
hence its list of singular values indeed consists 
of $\ell$ times $1$ and the entries of~$D_x$. 
This proves the first assertion. 

For the second assertion note that 
$\dim(E+F) = \ell + \#\{ i \mid \varpi^{x_i} \ne 0\}$, 
which indeed implies 
$\dim(E\cap F) = k+\ell - \dim(E+F) = \#\{ i \mid \varpi^{x_i} =0\}$.
\end{proof}

\begin{example} 
If $E$ and $F$ are two lines in $V\simeq K^2$ spanned by unit norm vectors $v$ and $w$, 
then the position parameter $x\in\bN$ of $(E,F)$ is given by 
$\e^x= \s= \|v \wedge w\|$. 
Note that $E=F$ iff $\s=0$ iff $x=\infty$. 
The other extremal case is 
$\s=1$, which corresponds to $x=0$: this means that $E_R+F_R =V_R$. 
Over $\R$ (or $\C$), this case corresponds to $E$ and $F$ being orthogonal.
More generally, in the situation of two real lines in $\R^n$ meeting at angle $\theta$,
the quantity $\s$ corresponds to $\sin\theta$.
We may thus interpret $\s$ (and its encoding $x$) as a notion 
of angle between the lines. 
\end{example}

The following important result shows that the situation of Lemma~\ref{le:NF-PP} 
actually is general: any pair of subspaces of $V$ arises this way. 
The proof relies on the block Smith normal form decomposition of Theorem~\ref{th:BSVD},
more specifically, on its consequence stated as Corollary~\ref{cor:rectangle-BSVD}.

\begin{thm}\label{th:NF-pairs-subspaces}
  Let $E,F\subseteq V$ be $K$-subspaces of dimensions $k$ and $\ell$, 
  respectively, such that 
  $k\le\ell$ and $k+\ell \le n$. 
  Then there exists $(x_1,\ldots,x_k)\in \Nko$ 
  and there is an $R$-basis of $K^n$,
  $$
  e_1,\ldots,e_\ell,f_1,\ldots,f_k, g_{1},\ldots,g_{n-k-\ell} 
  $$ 
  such that 
  $F=\spann\{e_1,\ldots,e_{\ell}\}$ 
  and 
  $$
  E= \spann\{ e_1+\varpi^{x_1} f_1,\ldots, e_k + \varpi^{x_k} f_k\}.
  $$ 
  The pair $(E,F)$ has the position vector~$x$. 
\end{thm}

\begin{proof}
  Without loss of generality, we may change coordinates so that $F = K^\ell \times 0$.
  Let $M \in R^{n \times k}$ be a matrix whose columns form an $R$--basis of the subspace~$E$.
  By Corollary~\ref{cor:rectangle-BSVD} on the block Smith normal form decomposition,
  there are $P\in U_{\ell,n}$ and $Q \in \GL_k(R)$ such that 
  \begin{equation} \label{eq:new-basis}
    MQ =  P 
    \begin{bmatrix} I_k \\ 0 \\ D \end{bmatrix} ,
  \end{equation} 
  where $D \in R^{(n-\ell) \times k}$ is in Smith normal form. 
  Thus the diagonal entries of~$D$ are of the form 
  $\varpi^{x_1},\ldots,\varpi^{x_k}$, where $(x_1,\ldots,x_k)\in \Nko$ 
  (with the convention $\varpi^{\infty}=0$).
  The column span of $M$ and hence of $MQ$ equals $E$ by construction. 

  We decompose $P=[P_{ij}]_{i,j=1,2}$ with the matrix $P_{11}$ of format $\ell\times\ell$ 
  and note that $P_{21}=0$. 
  Now we define  
  $e_1,\ldots,e_\ell$ to be the columns of $\begin{bmatrix} P_{11}\\ 0 \end{bmatrix}$.
  The span of these vectors indeed equals $F=K^\ell\times 0$, since $P$ is block upper triangular.
  Moreover, we define 
  $f_1, \ldots, f_k, g_1, \ldots, g_{n-k-\ell}$
  to be the columns of 
  $\begin{bmatrix} P_{12} \\ P_{22} \end{bmatrix}$. 
  The total list of vectors  
  $e_1,\ldots,e_\ell, f_1, \ldots, f_k, g_1, \ldots, g_{n-k-\ell}$
  form an $R$-basis of~$K^n$, since $P \in \GL_n(R)$. 
  In addition, 
  Equation~\eqref{eq:new-basis} reveals that the columns of $MQ$ 
  are given by 
  $e_1 + \varpi^{x_1} f_1,\ldots,e_k +  \varpi^{x_k} f_k$. 
  Finally, Lemma~\ref{le:NF-PP} shows that the pair $E,F$ has the position vector~$x$. 
\end{proof}

So far we only defined the position vector for pairs of subspaces  
satisfying $k\le\ell$ and $k+\ell\le n$. 
We now get rid of these constraints. 

\begin{defi}\label{def:PP2}
Let $E,F\subseteq V$ be $K$-subspaces and 
$\bar{E}$ and $\bar{F}$ denote the images of $E$ and $F$ 
under the canonical projection $V\to V/(E \cap F)$, respectively. 
Put $d:=\dim(E\cap F)$. 
We define the {\em position vector} of the pair $E, F$ 
as that of the pair $\bar{E}, \bar{F}$, appended with $d$ many $\infty$.
\end{defi}

Lemma~\ref{le:sv-addmap} reveals that this definition is 
consistent with Definition~\ref{def:PP1}.
We now show that the position vector characterizes the 
relative position of two subspaces, justifying the naming.

\begin{cor}\label{cor:RPP-Char}
Let $E,E'\subseteq V$ be $K$-subspaces of dimension $k$ 
and $F,F'\subseteq V$ be $K$-subspaces of dimension $\ell$. 
Then 
$$
 \exists g\in\Gl_n(R)\quad g(E)= E', \ g(F)=F'
$$ 
iff the pairs of subspaces $(E,F)$ and $(E',F')$ have the same position vector.
\end{cor}

\begin{proof}
In the case $k\le\ell$ and $k+\ell\le n$, the assertion follows immediately from 
Theorem~\ref{th:NF-pairs-subspaces}.  We now reduce the general case to this. Assume that 
$E,F$ and $E',F'$ have the same position vector~$x$.
Then $\dim(E\cap F) = \#\{i\mid x_i= \infty\} = \dim(E'\cap F')$. 
After applying a transformation in $\Gl_n(R)$ we may assume that 
$U:=E\cap F = E'\cap F'$. 
Let $W$ be an $R$-complement of $U$ in $V$ and 
$\pi\colon V\to W$ the projection along $U$. 
Let $\bar{E}$ denote the image of $E$ under $\pi$.
Then $E=\bar{E} \oplus U$.
Analogous statements holds for the images 
$\bar{F},\bar{E'},\bar{F'}$ of 
$F,E',F'$ under $\pi$.
By Definition~\ref{def:PP2}, $\bar{E},\bar{F}$ and $\bar{E'},\bar{F'}$ 
have the same position vector.  
We are now in the situation already covered: hence 
there is $h\in\Gl(W)$ defined over~$R$ 
such that $h(\bar{E})= \bar{E'}$ and $h(\bar{F})=\bar{F'}$. 
Then $g=\mathrm{id}_U\oplus h$ does the job.
\end{proof}

Let $V^*$ denote the dual space of $V$ and 
$E^\perp\subseteq V^*$ denote the orthogonal complement of $E$.
We show now that the position parameters behave nicely when passing to 
orthogonal complements.

\begin{prop}\label{pro:PPdual}
Let $E,F\subseteq V$ be $K$-subspaces and $E^\perp$ and 
$F^\perp$ be their orthogonal complements in the dual space $V^*$. 
Then, after removing the components equal to $\infty$, 
the pairs $(E^\perp,F^\perp)$ and $(E,F)$ have the same position vector.
\end{prop}

\begin{proof}
We may assume that $\dim E + \dim F \le \dim V$, 
otherwise we consider $E^\perp,F^\perp$ instead of $E,F$. 
As before, we write $k:=\dim E$, $\ell := \dim F$, $n:= \dim V$ 
and we assume without loss of generality that $k\le \ell$. 

By Theorem~\ref{th:NF-pairs-subspaces} there is 
an $R$-basis of $V$,
$$
 e_1,\ldots,e_\ell,f_1,\ldots,f_k, g_{1},\ldots,g_{n-k-\ell},
$$  
such that 
$E= \spann\{ e_1+\varpi^{x_1} f_1,\ldots, e_k + \varpi^{x_k} f_k\}$
and 
$F=\spann\{e_1,\ldots,e_{\ell}\}$.
Consider the dual basis of $V^*$,
$$
 e_1^*,\ldots,e_\ell^*,f_1^*,\ldots,f_k^*, g_{1}^*,\ldots,g_{n-k-\ell}^* . 
$$ 
Then 
\begin{align*}
 E^\perp &= \spann\{ \varpi^{x_1}e_1^* - f_1^*,\ldots,  \varpi^{x_k} e_k^* - f_k^*, e_{k+1}^*,\ldots,e_{\ell}^*,
   g_{1}^*,\ldots,g_{n-k-\ell}^*\}
  , \\ 
 F^\perp &= \spann\{  f_1^*,\ldots,f_k^*, g_{1}^*,\ldots,g_{n-k-\ell}^*\} .
\end{align*}
Note that $E^\perp\cap F^\perp$ contains the space 
$U := \spann\{g_{1}^*,\ldots,g_{n-k-\ell}^*\}$.
Take the following $R$-complement $E'$ of $U$ in $E^\perp$ and $F'$ of $U$  in $F^\perp$: 
$$
 E' := \spann\{\varpi^{x_1}e_1^* - f_1^*,\ldots,  \varpi{x_k} e_k^* - f_k^*, e_{k+1}^*,\ldots,e_{\ell}^*\},
 F' := \spann\{f_1^*,\ldots,f_k^*\}.
$$
We see that $(x_1\ldots,x_k)$ is the position vector of $E',F'$. 
Hence the position vector of $E^\perp,F^\perp$ is given by 
$(x_1\ldots,x_k,\infty,\ldots,\infty)$, where $\infty$ occurs with multiplicity 
$n-k-\ell$.
\end{proof}

\subsection{A quantity derived from the position parameter between subspaces} \label{sec:relative-position-general}

We generalize here a volume like quantity appearing in \cite[Section~3.5]{BL:19}
from the reals to the nonarchimedean setting. 
The resulting quantity is an indispensable ingredient for the general integral geometry formula
to be developed.

Again, $V$ denotes an $n$--dimensional $K$--space with linear $R$--structure, 
but now we consider several $K$-linear subspaces $V_1,\ldots,V_s$, 
of dimensions $m_1,\ldots,m_s$,
respectively, such that $\sum_{j=1}^s m_j \le n$. 
By Lemma~\ref{le:Rstruct} we have an induced linear $R$--structure on the exterior powers 
$\Lambda^{m}V$ and hence a well defined norm. 
The following real number in $[0,1]$ measures the relative position of the subspaces
$V_1,\ldots,V_s$:
\begin{equation} \label{eq:def-sigma-many}
 \s(V_1,\ldots,V_s) :=\|v_{11}\wedge\ldots \wedge v_{1m_1}\wedge\ldots
     \wedge v_{s1}\wedge\ldots\wedge v_{sm_s} \| ,
\end{equation}
where $(v_{j1},\ldots,v_{jm_j})$ is an $R$--basis of $(V_j)_R$ 
(recall that the linear $R$--structure on the exterior algebra is defined 
in Lemma \ref{le:Rstruct}). 

\begin{lemma}
The quantity $\s(V_1,\ldots,V_s)$ is well defined (i.e., independent of the 
choice of the $R$--bases). It is 
positive iff $V_1+\ldots + V_s$ is a direct sum.
Moreover, we have $\s(V_1,\ldots,V_s) =1$ iff 
$(V_1)_R \oplus \ldots \oplus (V_s)_R = (V_1+\ldots + V_s)_R$.
\end{lemma}

\begin{proof}
The first two assertions are clear. The third one follows from 
Lemma~\ref{le:RbasChar}. 
\end{proof}

This lemma again indicates that we may view $R$--bases as a substitute 
for the notion of orthonormal bases in euclidean spaces. 

In the case of just two subspaces $E,F$ of $V$, the 
quantity $\sigma(E,F)$ is easily computed from the 
position vector of a pair $(E, F)$ as follows.

\begin{lemma}
For subspaces $E, F\subseteq V$ with the position vector 
$(x_1, \ldots, x_k)$ we have 
\begin{equation*}\label{eq:sigmaPP}
  \sigma(E,F) = \e^{x_1}\cdots \e^{x_k} = \e^{x_1+\cdots+ x_k} .
\end{equation*}
\end{lemma}

\begin{proof}
Let $(e_1,\ldots,e_\ell,f_1,\ldots,f_k, g_{1},\ldots,g_{n-k-\ell})$ be
an $R$-basis of $K^n$  
as in Theorem \ref{th:NF-pairs-subspaces} such that 
$F=\spann\{e_1,\ldots,e_{\ell}\}$ and 
$E= \spann\{ e_1+\varpi^{x_1} f_1,\ldots, e_k + \varpi^{x_k} f_k\}$.
Using the definition~\eqref{eq:def-sigma-many}, we get 
\begin{align*}
\sigma(E, F)&=\|(e_1+\varpi^{x_1} f_1)\wedge\cdots\wedge(e_k+\varpi^{x_1} f_k)\wedge e_1\wedge \cdots \wedge e_\ell\|\\
                 &=\|\varpi^{x_1} f_1\wedge\cdots\wedge \varpi^{x_1} f_k\wedge e_1\wedge \cdots \wedge e_\ell\|\\
                 &=|\varpi^{x_1}\cdots \varpi^{x_k}|\cdot \|f_1\wedge\cdots \wedge f_k\wedge e_1\wedge \cdots \wedge e_\ell\|\\
                 &=\e^{x_1}\cdots \e^{x_k}. \qedhere
\end{align*}
\end{proof}

\begin{remark}
In the analogous situation of two subspaces $E,F$ of a Euclidean space~$V$, 
we have $\s(E,F) = \prod_j |\sin\theta_j|$, where the $\theta_j$ denote the 
principal angles between $E$ and $F$; see \cite[Lemma~3.9]{BL:19}. 
\end{remark}

\begin{cor}\label{cor:PPdual}
If $E,F$ are subspaces of $V$ such that $V=E\oplus F$, then 
$\sigma(E^\perp,F^\perp)=\sigma(E,F)$. 
\end{cor}

\begin{proof}
Assume $k=\dim E \le \dim F =\ell$ and let $(x_1\ldots,x_k)$ denote the 
position vector of $E,F$. Then $x_i\ne \infty$ for all~$i$ since $E\cap F=0$.  
Proposition~\ref{pro:PPdual} tells us that the pair $(E^\perp,F^\perp)$ also has the 
position vector~$x$. Hence Lemma~\ref{eq:sigmaPP} implies 
$\s(E,F) = \e^{x_1+\ldots +x_k} =\s(E^\perp,F^\perp)$.
\end{proof}

% PB: commented out the following since out of context.
%From Lemma~\ref{le:Rstruct} we know that the linear $R$--structure on $V$ induces  
%a linear $R$--structure on the dual space $(V^*)_R$.
%We denote by $V_i^\perp\subseteq V^*$ the orthogonal complement of the subspace~$V_i$, 
%which consists of the linear forms on $V$ vanishing on $V_i$. Hence, 
%$\s(V_1^\perp,\ldots,V_s^\perp)$ is defined, provided $\sum_j (n-m_j) \le n$. 

%As a consequence, we obtain the following.
%we can reprove \cite[Prop.~4.11]{BL:19}.
%\PB{Wrong reference!}
%(The proof in the latter reference relied on an analogue of the Hodge star operation,
%which is not needed here.)

We proceed with a result that will be essential in the proof of the basic integration formula in Proposition~\ref{pro:BIF}. 

\begin{lemma}\label{le:auxy}
Let $V$ be a $K$--vector space with linear $R$--structure and $E,F$ be $K$--subspaces of~$V$ 
such that $E+F=V$. Then the subtraction map
$$
 \psi\colon E\times F \to V,\, (x,y) \mapsto x- y 
$$
satisfies $N(\psi) = \s(E^\perp,F^\perp)$.
\end{lemma} 

\begin{proof}
Assume first that $E \cap F=0$. Then $\psi$ is invertible and we have 
$N(\psi) = J(\psi)$. 
If $v_1,\ldots,v_k$ and $w_1,\ldots,w_{n-k}$ are $R$--bases of $(E)_R$ and $(F)_R$, respectively, then 
we can express
$$ 
  J(\psi) = \|v_1\wedge\ldots\wedge v_k\wedge w_1\wedge\ldots\wedge w_{n-k}\| 
  = \s(E,F) = \s(E^\perp,F^\perp) .
$$
where the last equality is due to Corollary \ref{cor:PPdual}. 

For the general case, we put $U :=E \cap F$ and 
choose $R$-complements $E',F'$ 
so that 
$U \oplus E' = E$ and $U \oplus F' = F$.
Then we have 
$E' \cap F' =0$, and $U\oplus V' =V$, 
where $V' := E'\oplus F'$. 
Note that $D :=\ker\psi = \{ (x,x) \mid x \in U\}$. 
Clearly, $U\times 0$ is an $R$-complement of $D$ in $U\times U$. Hence 
the decomposition of $K$--spaces,
$$
 E \times F = D \oplus (U\times 0) \oplus V' , 
$$
leads to a direct decomposition of the corresponding $R$-modules. 
We restrict the map $\psi$ to $(U\times 0) \oplus V' \to U \oplus V'$ 
and observe that it decomposes as $\mathrm{id}_U \oplus \psi'$, 
where $\psi'\colon V' \to V',\, (x,y) \mapsto x-y$. 
Therefore, 
$$ 
 N(\psi) = J(\mathrm{id}_U \oplus \psi') = J(\psi')  = \s(E',F') ,
$$
where the last equality is due to the above analysis of the special case. 
We finally note that the orthogonal complement of $E$ in $V$ can be identified with 
the orthogonal complement of $E'$ in~$V'$ 
and similarly for $F'$.  
The assertion follows now with Corollary \ref{cor:PPdual}. 
\end{proof}

\subsection{Joint density of position parameters of random subspaces}\label{sec:jointrandom}
We again fix $k\le\ell$ such that $k+\ell \le n$. 
For $x\in\Nko$ we denote by $\rho_{k,\ell, n}(x)$ the probability that the pair $(E,F)$ 
of independent uniformly random subspaces $E\in G(k,n)$ and $F\in G(\ell,n)$ 
has the position vector~$x$.
The goal of this section to derive an explicit formula for the discrete distribution $\rho_{k,\ell, n}$.
%The result is a nonarchimedean version of~\cite[Theorem~3.2]{BL:19}.
For this purpose, we adopt the method in~\cite[D.3]{amelunxen-diss} from the real to the nonarchimedean setting. 

%\comm{Our development is technically quite involved. 
%It generalizes~\cite[D.3]{amelunxen-diss} to the nonarchimedean setting.}

%\PB{Can we restructure this and next section to make it more digestible?} 

%We use these results in Section \ref{se:Vol-Schubert} to compute the volume of special Schubert varieties. 
The idea is to use Theorem~\ref{th:NF-pairs-subspaces} to
parameterize subspaces $E\in G(k,n)$ 
in terms of their relative position with the fixed subspace 
$F:=K^\ell \times 0\subseteq K^n$.
The parameterization will be done in terms of a position vector $x\in \Nko$ 
and in terms of the invertible matrix $g\in \Gl_n(R)$, 
with columns 
$e_1,\ldots,e_{\ell}, f_1,\ldots,f_k,g_1,\ldots,g_{n-k-\ell}$,  
such that $e_1,\ldots,e_\ell$ span $F$ and 
$e_1+\varpi^{x_1} f_1,\ldots,e_k+ \varpi^{x_k} f_k$ span~$E$. 

We consider the following subgroup of $\Gl_n(R)$, 
which already appeared in Section~\ref{se:block-SVD}: 
\begin{equation}\label{eq:def-U}
 U := U_{\ell,n} := \left\{
\begin{bmatrix} A & B \\ 0 & C \end{bmatrix} 
  \Big| \; A\in \Gl_\ell(R), C\in \Gl_{n-\ell}(R), B \in R^{\ell\times (n-\ell)} \right\} .
\end{equation}
Moreover, we define the collection of maps  $\psi_x$, 
parameterized by $x\in\Nko$:  
\begin{equation}\label{eq:def-psi}
 \psi_x := \psi_{x;k,\ell, n} \colon U_{\ell, n} \to G(k,n),\, g\mapsto E
\end{equation}
which sends $g \in U$ to the 
span $E$ of the first $k$ columns of the matrix product
\begin{equation}\label{eq:def-psi1}
  g \cdot
  \begin{bmatrix}
    I_k & 0 & 0 & 0\\ 
    0 & I_{\ell-k} & 0 & 0 \\
    D_x & 0 & I_k & 0 \\
    0   &  0 & 0    & I_{n-k-\ell}
  \end{bmatrix} ,
\end{equation}
where $D_x := \diag(\varpi^{x_1},\ldots,\varpi^{x_k})$. 

%Let us summarize some of the properties of $\psi_x$.

\begin{lemma}\label{le:psi}
The image of $\psi_x$ consists of the subspaces $E\in G(k,n)$ having with 
$F=K^\ell \times 0$ the position vector~$x$.
\end{lemma}

\begin{proof}
      We decompose the $n\times n$ matrix $g$ into 16 blocks 
      according to $n=k+(\ell-k)+k+(n-k-\ell)$, which yields
      \begin{equation}\label{eq:decomp-g}
        \begin{bmatrix}
          A^{11} & A^{12} & B^{11} & B^{12}\\ 
          A^{21} & A^{22} & B^{21} & B^{22} \\
          0  & 0 & C^{11} & C^{12} \\ 
          0  &  0 & C^{21} & C^{22}
        \end{bmatrix} 
        \cdot 
        \begin{bmatrix}
          I_k & 0 & 0 & 0\\ 
          0 & I_{\ell-k} & 0 & 0 \\
          D_x & 0 & I_k & 0 \\
          0   &  0 & 0    & I_{n-k-\ell}
        \end{bmatrix} 
        = 
        \begin{bmatrix}
          A^{11} + B^{11} D_x & * & * & *\\
          A^{21} + B^{21} D_x & * & * & * \\
          C^{11} D_x & * & *  & * \\
          C^{21} D_x   &  * & *  & *
        \end{bmatrix} .
      \end{equation}
      Let us denote by  
      $e_1,\ldots,e_k; e_{k+1},\ldots,e_\ell; f_1,\ldots,f_k; g_1,\ldots,g_{n-k-\ell}$ 
      the columns of the matrix $g$ on left-hand side. They form 
      an $R$-basis of $K^n$ iff $g\in\Gl_n(R)$. In this case,
      $e_1,\ldots,e_k; e_{k+1},\ldots,e_\ell$ span $F=K^{\ell}\times 0$.
      Note that $E=\psi_x(g)$ is the column span of the matrix 
      \begin{equation}\label{eq:psi-lift}
        \begin{bmatrix}
          A^{11} + B^{11} D_x \\
          A^{21} + B^{21} D_x \\
          C^{11} D_x         \\
          C^{21} D_x
        \end{bmatrix} 
        =
        \begin{bmatrix}
          A^{11} \\
          A^{21} \\
          0      \\
          0
        \end{bmatrix} 
        + 
        \begin{bmatrix}
          B^{11}  \\
          B^{21}  \\
          C^{11}  \\
          C^{21}
        \end{bmatrix} 
        \cdot D_x ,
      \end{equation}
      which has the columns 
      $e_1 + \varpi^{x_1} f_1,\ldots,e_k + \varpi^{x_k} f_k$. 
      This shows that the pair of subspaces $E$, $F$ indeed 
      has the position vector~$x$.
      By Theorem~\ref{th:NF-pairs-subspaces}, any such $E$ is of the form $\psi_x(g)$ 
      for some $g\in U$. 
\end{proof}

Clearly, the group $U$ acts transitively and isometrically by left multiplication on $U$.
Moreover, $\psi_x$ is $U$-equivariant. 
Hence,  
%This greatly simplifies the computations. 
in order to analyze~$\psi_x$,
is sufficient to focus at the point~$I$.
We postpone the technical proof of the next lemma to Section~\ref{se:tech}.

%Let us first determine the absolute Jacobian of $\psi_x$. 

\begin{lemma}\label{le:Jac-psi}
Let $x\in\Nko$ and put $r_i := \e^{x_i}$. 
The absolute Jacobian of $\psi_x$ at $I$ is given by 
$$
 J(\psi_x)(I) = \prod_{i=1}^k |r_i|^{n+1-k-\ell} \; \prod_{i< j} |r_j|^2 .
$$
\end{lemma}

\begin{remark}\label{re:non-repellent}
Unlike the situation over $\R$ and $\C$, see Remark~\ref{re:psi-RC}, 
Lemma~\ref{le:Jac-psi} reveals that the normal Jacobian $J(\psi_x)$
does not vanish if the position vector $x$ has a repeated entry. 
This distinct feature is a consequence of the ultrametric triangle inequality.
In the ultrametric setting, the singular values of random matrices 
do not repell each other!
\end{remark}

%Hence  
%the absolute Jacobian of $\psi_x$ is constant: 
%   $J(\psi_x)(g) =  J(\psi_x)(I)$.
%Moreover, the volume of the fibers of $\psi_x$ is constant: 
%    $|\psi_x^{-1} (\psi_x(g))| = |\psi_x^{-1}(\psi_x(I))|$.

%HHH

%We again fix $k\le\ell$ such that $k+\ell \le n$ 
%and put $F := K^\ell\times 0$. 
%Consider the map
%$$
% \chi\colon G(k,n) \to \Nko, \, E \mapsto x , 
%$$
%which sends $E$ to the position vector~$x$ of the pair $(E,F)$. 
%We denote by $\rho_{k,\ell, n}$ 
%the pushforward of the uniform probability distribution on $G(k,n)$. 
%This is a discrete distribution on $\Nko$ that 
%describes the joint distribution of the position parameters of 
%uniformly random subspaces $E\in G(k,n)$ and $F\in G(\ell,n)$. 
%We provide an explicit formula for this distribution.

%\ref{pro:vols}

%The proof relies on a parameterization of the Grassmannian $G(k,n)$ in terms of the relative position of a space
%with respect to a fixed space. This will be needed in Section~\ref{sec:jointrandom} to describe the
%distribution of the position vector for a uniformly random subspace. 
%Our development is technically quite involved. 
%It generalizes~\cite[D.3]{amelunxen-diss} to the nonarchimedean setting.}

%The proof is quite involved and requires several steps.

For describing the fiber of $\psi_x$ over $\psi_x(I)$, 
we restrict ourselves to the case $x\in\N^k_o$, that is,  
we assume that $x_i\ne\infty$ for all~$i$. This 
amounts to requiring that $r_i\ne 0$ for all~$i$. 
This assumption will guarantee that 
the image of $\psi_x$ is full dimensional, i.e., 
of dimension $\dim G(k,n)$. 
The fiber volume will be expressed in terms of the quantity
$\g_n = |\Gl_n(R)|= \prod_{i=1}^n (1-\e^i)$, 
see Proposition~\ref{pro:vols}. 
We note that 
\begin{equation}\label{eq:volume-UG}
 |U_{\ell, n}| = \g_\ell \, \g_{n-\ell}  . 
\end{equation}

We introduce some notation to partition $[k]$ according to the multiplicities of the entries of $x$. 
Let $\xi_1>\ldots > \xi_t\ge 0$ denote the different values 
of the map $[k]\to\N,\,i\mapsto x_i$.
We denote by $\nu$ the number of times $0$ occurs 
among the components of $x$. We partition 
$[k]=S_1\cup\ldots\cup S_t$ into the index sets  
$$
 S_\tau := \{ i\in [k] \mid x_i = \xi_\tau \}
$$
and define $\mu_\tau:= \# S_\tau$. 
Note that  $\mu_1+\ldots + \mu_t = k$ and  
$\nu = \mu_t$ if $\xi_t=0$, and $\nu=0$ otherwise.

\begin{lemma}\label{le:fibers}
Let $x\in\N^k_o$. Then:  
\begin{enumerate}
  \item
    The fiber $\psi_x^{-1}(\psi_x(I))$ 
    is a $K$--analytic submanifold of $U_{\ell, n}$ of 
    codimension $k(n-k)$. 
  \item
    The image $\im\psi_x$ is of full dimension $k(n-k)$.
  \item
    The volume of the fiber $\psi_x^{-1}(\psi_x(I))$
    is given by 
    \[
      |\psi_x^{-1}(\psi_x(I))|  = \g_{n-k-\ell} \cdot F_{k,\ell}(x) ,
      \quad \text{where} \ F_{k,\ell}(x) := \g_{\nu + \ell -k}\; \prod_{\tau=1}^t\g_{\mu_\tau}.
    \]
    The factor $F_{k, \ell}(x)$ 
    does not depend on the ambient dimension~$n$. 

  \end{enumerate}
\end{lemma}

The proof is postponed to Section~\ref{se:tech}. 
{By contrast with the situation over $\R$ and $\C$, the fibers of $\psi_x$ have large dimension, 
see Remark~\ref{re:psi-RC}.} 

%We investigate the probability distribution of the position parameters of 
%uniformly random subspaces. %$E\in G(k,n)$ and $F\in G(\ell,n)$. 
%To this end, we again fix $k\le\ell$ such that $k+\ell \le n$ and 
%we put $F := K^\ell\times 0$. 
%Consider the map
%$$
% \chi\colon G(k,n) \to \Nko, \, E \mapsto x , 
%$$
%which sends $E\in G(k,n)$ to the position vector~$x$ of the pair $(E,F)$. 
%We denote by $\rho_{k,\ell, n}$ the pushforward of the uniform probability distribution on $G(k,n)$. 
%%This discrete distribution describes the joint distribution of the relative position of $E$ and $F$. %$E\in G(k,n)$ and $F\in G(\ell,n)$. 

%The goal of this section to prove following explicit formula for the discrete distribution $\rho_{k,\ell, n}$.
%(This is a nonarchimedean version of~\cite[Theorem~3.2]{BL:19}.)
%formula for this distribution, which is a nonarchimedean version of~\cite[Theorem 3.2]{BL:19}.
%(This provides a nonarchimedean generalization of well known results
%of relevance in multivariate statistics~\cite{Muirhead}.)

%Our formula involves an explicit parametrization of the various open
%sets in the Grassmannian where the position vector is constant, and
%requires the computation of the normal Jacobian of this
%parametrization, see Proposition~\ref{le:Jac-psi}. 
%\comm{This provides a nonarchimedean generalization of well known results
%of relevance in multivariate statistics~\cite{Muirhead}.}

We can finally state the main result of this section, which 
is a nonarchimedean version of~\cite[Theorem~3.2]{BL:19}.

\begin{thm}\label{th:joint-densityK}
The joint distribution $\rho_{k,\ell, n}$ of the position vector of 
uniformly random subspaces $E\in G(k,n)$ and $F\in G(\ell,n)$ is given by 
%($x\in\Nko$)
$$
 \rho_{k,\ell, n}(x) 
 = c_{k,\ell,n} \cdot  \frac{J(\psi_{x;k,\ell ,n})(I,x)}{F_{k,\ell}(x)}  \quad \mbox{ for $x\in\Nko$,}
$$
where 
the constant $c_{k,\ell,n}$ is given by 
$$
 c_{k,\ell,n} = \frac{|U_{\ell, n}|}{|G(k,n)|} \frac{1}{\gamma_{n-k-\ell}} = 
\frac{\g_k \g_{n-k} \g_\ell \g_{n-\ell}}{\g_{n-k-\ell} \g_n} .
$$ 
and $F_{k,\ell}$ is defined in Lemma~\ref{le:fibers}. 
\end{thm}

\begin{proof} 
Put $F:=K^\ell\times 0$ and 
consider the map
$$
 \chi\colon G(k,n) \to \Nko, \, E \mapsto x , 
$$
which sends $E\in G(k,n)$ to the position vector~$x$ of the pair $(E,F)$. 
Recall that $\rho:=\rho_{k,\ell, n}$ is the pushforward of the uniform probability distribution on $G(k,n)$. 
By definition, $\rho(x)$ is the relative volume of the fiber of $\chi$ over $x \in \Nko$, i.e., 
  $$
  \rho(x) = \frac{|\chi^{-1}(x)|}{|G(k,n)|}.
  $$ 
  Lemma~\ref{le:psi} implies $\chi^{-1}(x) = \im \psi_x$. 
  Applying the coarea formula  (Theorem~\ref{th:coarea}) 
  to $\psi_x$ gives 
  $$
  \int_{U} \frac{J(\psi_x)}{|\psi_x^{-1}(\psi_x)|} \, dU
  = \int_{\chi^{-1}(x)} \, dG(k,n)  = |\chi^{-1}(x)| .
  $$
  The invariance properties of the absolute Jacobian and fiber volumes of $\psi_x$
  imply that the integrand of the the left-hand integral is 
  constant. Therefore, 
  $$
  |U| \cdot \frac{J(\psi_x)(I)}{|\psi_x^{-1}(\psi_x(I))|} = |\chi^{-1}(x)| .
  $$
  Using the notation introduced in Lemma~\ref{le:fibers}, we arrive at 
  $$
  \rho_{k,\ell, n}(x) = \frac{|\chi^{-1}(x)|}{|G(k,n)|} = 
  \frac{|U|}{|G(k,n)|} \frac{1}{\gamma_{n-k-\ell}} \cdot 
  \frac{J(\psi_x)(I,x)}{F_{k,\ell}(x)} .
  $$
  where the left-hand constant is denoted by $c_{k,\ell,n}$.
  Plugging in equation~\eqref{eq:volume-UG} into this yields 
  the stated formula for $c_{k,\ell,n}$.
\end{proof}

The following consequence of Theorem~\ref{th:joint-densityK} 
will be essential later on:
%\comm{for the computation of the volume of special Schubert varieties:}
let $m\le k$ and put $p := n -k-\ell +m$. Then we have 
\begin{equation}\label{eq:c-quot}
\begin{split}
 \frac{c_{k-m,\ell-m,n-m}}{c_{k-m,\ell-m,n}}  
   &= \frac{|U_{\ell-m, n-m}|}{|U_{\ell-m, n}|} \frac{|G(k-m,n)|}{|G(k-m,n-m)|} \frac{\g_{p +m}}{\g_{p}}   \\
   &= \frac{\g_{n-\ell}}{\g_{n-\ell+m}} \;\frac{|G(k-m,n)|}{|G(k-m,n-m)|} \;\frac{\g_{p +m}}{\g_{p}} . 
\end{split}
\end{equation}

The way the distribution $\rho_{k,\ell,n}$ depends on the ambient 
dimension~$n$ allows us to derive explicit formulas for the 
expectations of the functions 
$ x \mapsto \e^{s(x_1+\ldots + x_k)}$ 
with respect to  the distribution $\rho_{k,\ell,n}$, 
where $s\in\N$.
This will be important for the computation 
of the volume of special Schubert varieties. 

\begin{cor}\label{cor:rho-inheritance}
Let $k\le\ell$ be such that $k+\ell \le n$ and $n\le N$. Then 
we have for $x\in\Nko$ that  
$$
\rho_{k,\ell,n}(x) \e^{(N-n)(x_1+\ldots + x_k)} 
 = \frac{c_{k,\ell,n}}{c_{k,\ell,N}} \; \rho_{k,\ell,N}(x) .
$$
In particular,
$$
 \sum_{x\in\Nko} \rho_{k,\ell, n}(x) \e^{(N-n)(x_1+\ldots + x_k)} = \frac{c_{k,\ell,n}}{c_{k,\ell,N}} .
$$ 
\end{cor}

\begin{proof}
  Note that the statement of this corollary is trivially true when some of the entries of $x$ are equal to $\infty$,
  because these occur with zero density. 
  We again write $r_i=\varpi^{x_i}$. By inserting the Jacobian of $\psi_{k,\ell,n}$ from Lemma~\ref{le:Jac-psi}
  into Theorem~\ref{th:joint-densityK}, we obtain
  $$
  \rho_{k,\ell, n}(x) 
  = c_{k,\ell,n} \;\prod_{i=1}^k |r_i|^{n+1-k-\ell} \cdot \mathrm{Factor}_{k,\ell}(x) ,
  $$
  where the second factor
  $$
  \mathrm{Factor}_{k,\ell}(x) := \frac{1}{F_{k,\ell}(x)}  \prod_{i< j} |r_j|^2 
  $$ 
  does not depend on~$n$. Therefore, 
  $$
  \rho_{k,\ell, n}(x) \prod_{i=1}^k |r_i|^{N-n} = c_{k,\ell,n} \;\prod_{i=1}^k |r_i|^{N+1- k-\ell} \cdot 
  \mathrm{Factor}_{k,\ell}(x) ,
  $$
  and we indeed get 
  $$
  \rho_{k,\ell, n}(x) \prod_{i=1}^k |r_i|^{N-n} = \frac{c_{k,\ell,n}}{c_{k,\ell,N}}  \rho_{k,\ell, N}(x) .
  $$
  Finally, we use that $\sum_x \rho_{k,\ell, N}(x) =1$.
\end{proof}

\begin{remark}\label{re:psi-RC}
It is illuminating to compare our results with 
the corresponding situation over~$\R$ and $\C$.
There, one defines parameterizations
$\psi_\R\colon O(n) \times [0,\pi/2]^k \to G_\R(k,n)$
and 
$\psi_\C\colon U_n \times [0,\pi/2]^k \to G_\C(k,n)$
via the orthogonal group $O(n)$ and unitary group $U_n$ 
in a similar way; see
\cite[Appendix D.3]{amelunxen-diss} and 
\cite[\S 5.1]{buerg-intersect}. 
In stark contrast with the nonarchimedean~$\psi$, the generic fibers of $\psi_\R$ and $\psi_\C$ 
are very small: those of~$\psi_\R$ are finite with cardinality~$2^k$ 
(bases are uniquely determined up to sign), while  
the generic fibers of~$\psi_\C$ are $k$-fold products of unit circles. 
The normal Jacobians are given by 
\begin{align*}
 J(\psi_\R) &= \prod_{i=1}^k \sin(\theta_i)^{n-k-\ell} \cos(\theta_i)^{\ell-k}  \cdot 
  \prod_{1\le j<i\le k} |\sin^2(\theta_i) \cos^2(\theta_j) - \cos^2(\theta_i) \sin^2(\theta_j)| , \\
 J(\psi_\C) &= 2^{\frac{k}{2}} \prod_{i=1}^k \sin(\theta_i)^{2(n-k-\ell) +1} \cos(\theta_i)^{2(\ell-k)+1}  \cdot 
  \prod_{1\le j<i\le k} \big(\sin^2(\theta_i) \cos^2(\theta_j) - \cos^2(\theta_i) \sin^2(\theta_j)\big)^2 .
\end{align*}
where $\theta_1,\ldots,\theta_k$ denote the principal angles; 
see~\cite[Prop.~D.3.1]{amelunxen-diss} over~$\R$ and 
\cite[Thm.~5.3]{buerg-intersect} over~$\C$.
These normal Jacobians vanish if two principal angles coincide, 
which is not the case in the ultrametric setting, as already pointed out in Remark~\ref{re:non-repellent}.
The joint density of principal angles over $\R$ satisfies 
$$
 \rho^\R_{k,\ell,n}(\theta_1, \ldots, \theta_k) = C_{k, l, n}^\R\prod_{j=1}^k(\cos \theta_j)^{l-k}(\sin \theta_j)^{n-l-k}\prod_{i<j}
    \left((\cos \theta_i)^2-(\cos \theta_j)^2\right) .
$$
with an explicit constant $C_{k, l, n}^\R$, see \cite[Appendix D.3.1]{amelunxen-diss}.
This follows from the determination of the Jacobian and the fibers of the parameterization 
maps $\psi_\R$.
%compare Remark~\ref{re:psi-RC}.
Over $\C$, the joint density of principal angles can be derived analogously, 
compare \cite[\S 5.1]{buerg-intersect}. 

The reader may note that the exponent $n-k-\ell$ of $\sin(\theta_i)$ 
over $\R$ slightly differs from the exponent $n+1-k-\ell$ of $|r_i|$ in the situation over $K$. 
The explanation for this is as follows: consider the related parameterization
$U_{\ell,n}\times R^k \to G(k,n)$ sending $g\in U_{\ell,n}$ to the span of 
the first $k$ columns of the product~\eqref{eq:def-psi1}, 
where instead of 
$D_x = \diag(\varpi^{x_1},\ldots,\varpi^{x_k})$ 
we take $\diag(r_1,\ldots,r_k)$ with $r_i \in R$, which 
are also allowed to vary. 
It turns out that its absolute Jacobian is given by 
$\prod_{i=1}^k |r_i|^{n-k-\ell} \; \prod_{i< j} |r_j|^2$, which has 
same exponent $n-k-\ell$ as over~$\R$. 
\end{remark}

\subsection{Proof of two technical lemmas}\label{se:tech}

%We first determine the absolute Jacobian of $\psi_x$ at $I$.

\begin{proof}[Proof of Lemma~\ref{le:Jac-psi} (Absolute Jacobian of $\psi_x$)]
We fix $x\in\Nko$. 
From \eqref{eq:psi-lift} we see that 
$\psi_x(I)$ equals the column span of the matrix
$[I_k \,0 \,D_x\, 0]^T\in R^{n\times k}$.
Hence, for computing the derivative of $\psi_x$ at $I$, 
we can use the following chart from Example~\ref{ex:Grass}: 
$$
 \hat{U} \to R^{n\times k},\, 
  \mathrm{column span}\left(\begin{bmatrix} X\\ Y \end{bmatrix}\right)\mapsto Y X^{-1} ,
$$
where $X\in R^{k\times k}$, $Y\in R^{(n-k)\times k}$ and the subset 
$U\subseteq G(k,n)$ is characterized by $|\det(X)|=1$. 
In this chart, the second order approximation of $\psi_x$ at $I$ 
is the following bilinear map
(use~\eqref{eq:psi-lift} to see this) 
\begin{equation*} 
  \begin{bmatrix}
    \dot{A}^{11} & \dot{A}^{12} & \dot{B}^{11} & \dot{B}^{12}\\ 
    \dot{A}^{21} & \dot{A}^{22} & \dot{B}^{21} & \dot{B}^{22} \\
    0  & 0 & \dot{C}^{11} & \dot{C}^{12} \\ 
    0  &  0 & \dot{C}^{21} & \dot{C}^{22}
  \end{bmatrix} 
  \mapsto
  \begin{bmatrix}
    \dot{A}^{21} + \dot{B}^{21} D_x \\
    (I_k +\dot{C}^{11}) D_x  \\
    \dot{C}^{21} D_x 
  \end{bmatrix}
  \big(I_k - \dot{A}^{11} - \dot{B}^{11}D_x \big) .
\end{equation*} 
Therefore, the derivative $D_I\psi_x$ is given by 
\begin{equation*} 
  T_I U \to R^{(n-k)\times k}, \;
  \begin{bmatrix}
    \dot{A}^{11} & \dot{A}^{12} & \dot{B}^{11} & \dot{B}^{12}\\ 
    \dot{A}^{21} & \dot{A}^{22} & \dot{B}^{21} & \dot{B}^{22} \\
    0  & 0 & \dot{C}^{11} & \dot{C}^{12} \\ 
    0  &  0 & \dot{C}^{21} & \dot{C}^{22}
  \end{bmatrix} 
  \mapsto
  \begin{bmatrix}
    \dot{A}^{21} + \dot{B}^{21} D_x \\
    \dot{C}^{11}  D_x  - D_x \dot{A}^{11} - D_x \dot{B}^{11} D_x \\
    \dot{C}^{21} D_x  \\
  \end{bmatrix} .
\end{equation*} 
Note that the values of this map do not depend on the second and fourth column 
of the input, 
so we may ignore the $\dot{A}^{i2}, \dot{B}^{i2},\dot{C}^{i2}$ for $i=1,2$. 
The resulting map is a direct product of three linear maps
\[
  \begin{array}{llll}
    \phi_1 &\colon (K^{(\ell-k)\times k})^2 \to  K^{(\ell-k)\times k} , 
    &\; (\dot{A}^{21}, \dot{B}^{21}) &\mapsto \dot{A}^{21} + \dot{B}^{21} D_x ,\\
    \phi_2 &\colon (K^{k\times k})^3 \to   K^{k \times k} ,
    &\;(\dot{A}^{11}, \dot{B}^{11}, \dot{C}^{11}) &\mapsto \dot{C}^{11}  D_x  - D_x \dot{A}^{11} - D_x \dot{B}^{11} D_x , \\
    \phi_3 &\colon K^{(n-k-\ell)\times k} \to   K^{(n-k-\ell)\times k}, 
    &\; \dot{C}^{21} &\mapsto \dot{C}^{21} D_x  .
  \end{array}
\]
In fact, each of these maps $\phi_i$ is a product of smaller linear maps, e.g., 
$\phi_1$ is a product of $(\ell-k)k$ many maps 
$K\times K \to K$ of the form $(\dot{a},\dot{b}) \mapsto \dot{a} + d\dot{b}$, 
where $d\in R$. By Equation~\eqref{eq:adet-row}, 
such maps have the absolute determinant $\|(1,d)\|= 1$.
Therefore, $N(\phi_1)=1$. 
Similarly, one sees that 
$N(\phi_3)= |r_1 \cdots r_k|^{n-k-\ell}$, 
when writing $D=\diag(r_1,\ldots,r_k)$.  
Finally, for analyzing $\phi_2$, we observe that $\phi_2$ is the product of the 
linear maps ($1\le i,j\le k$): 
$$
  \lambda_{ij}\colon K^3 \to K, \;(\dot{a}_{ij},\dot{b}_{ij},\dot{c}_{ij}) \mapsto 
  \dot{c}_{ij} r_j - \dot{a}_{ij} r_i - \dot{b}_{ij} r_i r_j .
$$
According to Equation~\eqref{eq:adet-row}, we have 
$$
 N(\lambda_{ij}) = \|(r_j, -r_i, -r_i r_j)\| = \max\{|r_i|, |r_j|\} .
$$
Altogether, we arrive at 
$$
 J(\psi_x)(I) = N(\phi_1) N(\phi_2) N(\phi_3) = |r_1 \cdots r_k|^{n-k-\ell} \prod_{i,j} \max\{|r_i|, |r_j|\} 
 = |r_1 \cdots r_k|^{n+1-k-\ell} \prod_{i<j}|r_j|^2 
$$ 
as claimed.
\end{proof}

\begin{proof}[Proof of Lemma~\ref{le:fibers} (Fibers of $\psi_x$)]
We fix $x\in\N^k_o$.  
Let $g\in U$ be given by matrices 
$A\in \Gl_\ell(R), C\in \Gl_{n-\ell}(R), B \in R^{\ell\times (n-\ell)}$ 
as in \eqref{eq:def-U}. 
We further decompose $A,B,C$ into smaller matrices 
$A^{pq},B^{pq},C^{pq}$, $p=1,2$, 
as in~\eqref{eq:decomp-g}. 
Then we have $\psi_x(g)=\psi_x(I)$ iff 
$$
\mathrm{column span}\left(\begin{bmatrix} 
                        A^{11} + B^{11} D_x \\
                        A^{21} + B^{21} D_x \\
                        C^{11} D_x         \\
                        C^{21} D_x         \end{bmatrix}\right)
= \mathrm{column span}\left(\begin{bmatrix} I_k\\ 0 \\D_r \\ 0 \end{bmatrix}\right) . 
$$
This means that there exists $H\in\Gl_\ell(K)$ such that 
\begin{align*}
  A^{11} + B^{11} D_x &= H \\
  A^{21} + B^{21} D_x &= 0 \\
  C^{11} D_x          &= D_x H \\
  C^{21} D_x          &= 0 .
\end{align*}
This implies $ C^{21}=0$ since $D_x$ is invertible 
(here we use $x_i\ne\infty$). Moreover, the third equation implies 
$H= D_x^{-1} C^{11} D_x$, so that we can rewrite the first equation as 
\begin{equation}\label{eq:fiber-cond}
 D_x A^{11} D_x^{-1} + D_x B^{11} = C^{11} .
\end{equation}
Summarizing, we arrive at the following explicit characterization of the fiber:  
\begin{equation*}\begin{split}
 \psi^{-1}_x(\psi_x(I)) = \Big\{
 (A,B,C) = ([A^{pq}], [B^{pq}], [C^{pq}]) \mid  A^{21} + B^{21} D_x  = 0,\; C^{21} = 0, 
  \mbox{ \eqref{eq:fiber-cond} holds},\\
  \; A\in \Gl_\ell(R), \; C \in \Gl_{n-\ell}(R) \Big\}
\end{split}\end{equation*}
The conditions in the first line describe a linear subspace of $K^{n\times n}$; 
the codimension is easily checked to be 
$k(\ell -k) + k(n-k-\ell) + k^2 = k(n-k)$. 
The fiber is obtained by intersecting this subspace with $\Gl_n(R)$.
This implies that the codimension of the fiber indeed equals $k(n-k)$. 
Hence, we have $\dim\im\psi_x = k(n-k)$ as claimed. 

We now determine the volume of the fiber. 
First notice that the conditions on $(A,B,C)$ being in the fiber 
do not involve the second columns $A^{12},A^{22}$, $B^{12},B^{22}$, $C^{12},C^{22}$
of $A=[A^{pq}]$,  $B=[B^{pq}]$ and $C=[C^{pq}]$, respectively,   
except for the requirement that $A$ and $C$ should be invertible over $R$.

We will see in a moment that in this case, we can guarantee 
that $A$ and $C$ are invertible over~$R$ by  requiring that 
$A^{22}, C^{22}$ are invertible over~$R$ and 
{\em independently} that 
$A^{11}, C^{11}$ are invertible over $R$.
This shows that the volume of the set of 
$(A^{12},A^{22}$, $B^{12},B^{22}$, $C^{12},C^{22})$ 
being part of $(A,B,C)$ in the fiber equals 
$$
 |\Gl_{\ell -k}(R)| \cdot |\Gl_{n-k-\ell}
 (R)| = \g_{\ell -k} \g_{n-k-\ell} .
$$

We study now the conditions on the first columns 
$A^{11},A^{21}$, $B^{11},B^{21}$, $C^{11},C^{21}$ 
of $A,B,C$, , respectively.  
We have the independent conditions 
$A^{21} + B^{21} D_x = 0$ and $C^{21}=0$ on 
$A^{21}, B^{21}, C^{21}$, 
and the constraint \eqref{eq:fiber-cond} on 
$A^{11},B^{11},C^{11}$.

The set of $(A^{21}, B^{21}, C^{21})$ satisfying 
$A^{21} + B^{21} D_x = 0$ and $C^{21}=0$ 
is the graph of the linear map 
$R^{(\ell -k )\times k} \to R^{(\ell -k )\times k},\; B^{21}\mapsto - B^{21} D_x=A^{21}$
which is defined over $R$. Therefore the graph 
has the volume~$1$, see Corollary~\ref{cor:vol-graph} .

To analyze the set $\mathcal{G}$ of $(A^{11},B^{11},C^{11})$ arising in the fiber,  
recall the partition $[k]=S_1\cup\ldots\cup S_t$. 
We think of the $k\times k$ matrices $A^{11},B^{11}, C^{11}$ as being 
each decomposed into $t^2$ blocks accordingly,
where the blocks are indexed by elements in the sets $S_\tau$.
Let us rewrite \eqref{eq:fiber-cond} explicitly as 
\begin{equation}\label{eq:fiber-cond-I}
 \varpi^{x_i-x_j} a^{11}_{ij}  + \varpi^{x_i} b^{11}_{ij} = c^{11}_{ij} ,\quad 
  1\le i,j \le k .
\end{equation}
This implies $\varpi^{x_i-x_j} a^{11}_{ij} \in R$, hence 
$|a^{11}_{ij}| \le \e^{x_j-x_i}$. 
Therefore, when $i > j$, the exponent is nonnegative, i.e, $a_{ij}$ is zero modulo $\frm$. 
Thus, we see that $A^{11} \mod \frm$ is block lower triangular.
Similarly, $C^{11} \bmod \frm$ is block upper triangular.
In this situation,  
$A=[A^{pq}]$ is invertible over $R$ iff the diagonal blocks of $A$ are 
invertible over $R$.
The same characterization holds for $C$.
Note that Equation~\eqref{eq:fiber-cond-I} determines $C^{11}$ 
from $A^{11}$ and $B^{11}$. Moreover, 
\eqref{eq:fiber-cond-I} states for $i=j$ that 
$$
 a^{11}_{ii}  + \varpi^{x_i} b^{11}_{ii} = c^{11}_{ii}  .
$$
This implies $A^{11} \equiv C^{11} \bmod \frm$ if
$x_i >0$ for all $i$. 
The latter assumption means $\nu=0$.

So let us assume now that $\nu=0$ for simplicity.
Then we see that if all diagonal blocks of $A^{11}$ are 
invertible over $R$ and $C^{11}$ is computed via \eqref{eq:fiber-cond-I},
then $C$ will be invertible over $R$. 

We can view $\mathcal{G}$ as the graph of 
the following linear map $\phi$ defined over $R$.
It takes as input any $B^{11}\in R^{k\times k}$, 
as the diagonal blocks of $A$ any element of 
$\Gl_{\mu_1}(R)\times\cdots\times \Gl_{\mu_t}(R)$, 
as the lower triangular block of $A$ 
any element of $\prod_{\sigma>\tau} R^{\mu_\sigma\times\mu_\tau}$
(whose components are 
$\varpi^{x_i-x_j} a^{11}_{ij}\in R$ for $i>j$), 
and moreover, for the upper triangular part, 
an element of $\prod_{\sigma <\tau} R^{\mu_\sigma\times\mu_\tau}$
(whose components are $a^{11}_{ij}\in R$ for $i<j$).
Then $\phi$ computes the lower triangular blocks of $A$ by multiplying 
$$
 a^{11}_{ij} = \varpi^{x_j-x_i} \cdot (\varpi^{x_i-x_j}  a^{11}_{ij}) \mbox{ with $\varpi^{x_j-x_i} \in R$ for $i>j$ }, 
$$ 
and $\phi$ computes $C^{11}$ 
using \eqref{eq:fiber-cond-I}. 
Corollary~\ref{cor:vol-graph} now implies that 
$|\mathrm{graph}(\mathcal{G})| = \prod_{\tau=1}^t |\Gl_{\mu_\tau}|$. 
This proves that 
$$
|\psi_x^{-1}(\psi_x(I))|  = \g_{\ell -k} \g_{n-k-\ell} \cdot \prod_{\tau=1}^t |\Gl_{\mu_\tau} |
$$
in the situation where $\nu=0$.

In the case $\nu >0$ an additional reasoning, similar to the previous one is required.
We leave this to the reader.
\end{proof}

%HHH

%%% SECTION
\section{Nonarchimedean integral geometry}
In the next two sections we endow $K$--analytic groups and $K$--analytic homogeneous 
spaces with $R$--structures. 
We then proceed by proving a general integral geometry formula 
in the setting of nonarchimedean homogeneous spaces: 
the main result is Theorem~\ref{thm:GIGF}. 
Our development is parallel to Howard~\cite{howard:93} who carried this out over the reals. 
(For general background on integral geometry we refer to~\cite{klain-rota:97,santalo:04}.)
As a special case we obtain a new proof of the integral geometric formula in \cite{KL:19} 
for projective spaces over $\Q_p$.  
We also start to investigate the situation in Grassmannians along the lines of~\cite{BL:19}.

\subsection{$K$--analytic groups with $R$--structures}\label{se:anKgroups}

A {\em $K$--analytic group} is a group~$G$ with the structure of a $K$--analytic manifold, 
such that the multiplication $G\times G\to G,\,(g,h) \mapsto gh$ 
and the inversion $G\to G,\, g \mapsto g^{-1}$ are $K$--analytic maps; 
see~\cite[Part II]{serre:64} and \cite{schneiderp:11}. 
For instance $\Gl_n(K)$ is $K$--analytic group. 

We shall endow a $K$--analytic group~$G$ with an $R$--structure 
that is compatible with the group operations. 
For $g,h\in G$ we denote by 
$L_g\colon G \to G,\, x \to gx$ 
the left-multiplication with~$g$
and by $R_h\colon G \to G,\, x \to xh$
the right-multiplication with~$h$. 

\begin{defi}
We call an $R$--structure on a $K$--analytic group $G$ 
\emph{compatible} if, for all $g,h\in G$, 
the derivatives of the left-multiplication $L_g$ and the right-multiplication $R_h$ 
are isomorphisms defined over $R$. 
\end{defi}

The group $\Gl_n(R)$ is an open subset of $K^{n\times n}$ 
and thus carries an induced $R$--structure.
It is immediate to check that this $R$--structure is compatible. 

In the following, we assume that if $G$ is a $K$--analytic group with a compatible $R$--structure.
Thus the tangent spaces of $G$ carry an induced norm, 
which is invariant under left and right multiplications.
Moreover, there is an induced volume form $\Omega_G$ on $G$, which is 
invariant under left and right multiplications. 
If $G$ is compact, then it has a finite volume that we denote by 
$|G| := \int_G \Omega_G$.
(Compare~\cite[Chap.~II]{weil:41}.) 
Notice that if $H$ is an open subgroup of $G$, then $H$ inherits from $G$ a compatible $R$--structure.

\subsection{Homogeneous $K$--analytic spaces with $R$--structure}
\label{se:hom_K_analytic_spaces}

Suppose $G$ is a $K$--analytic group and 
$H$ is a $K$--analytic submanifold of $G$, which is also a closed subgroup of $G$.
It is known that if 
we form the quotient $G/H = \{gH \mid g\in G\}$ with the quotient topology, 
then $G/H$ gets the unique structure of a $K$--analytic manifold such that 
the canonical projection $p\colon G\to G/H$ is a $K$--analytic submersion; 
see~\cite[Part~II, Chap.~IV, \S 5]{serre:64}. 

If $G$ comes with a compatible $R$--structure, then we can endow 
the tangent spaces of $G/H$ with the quotient structure (see Lemma~\ref{le:Rstruct}(2)) 
induced from the $R$--structures on the tangent space of $G$ 
via the derivatives $D_gp\colon T_g G\to T_{p(g) }G/H$. 
We leave it to the reader to verify that the resulting $R$--structure 
on $G/H$ is compatible.
We call $G/H$ a {\em homogeneous $K$--analytic space with $R$--structure}. 

If $G$ is compact, then $G/H$  has a finite volume that we denote by $|G/H|$. 
The coarea formula implies that $|G/H| =|G|/|H|$ since $J(p)=1$ 
by construction.

\begin{example}\label{ex:parabolic}
For instance, consider 
the subgroup $H$ 
of $\Gl_n(R)$ consisting of the upper triangular matrices, or more generally, 
the (parabolic) subgroups of upper triangular matrices of a fixed block structure.
This means that flag varieties, notably Grassmannians and projective spaces, always
carry a canonical $R$--structure compatible with the action of $\Gl_n(R)$. 
\end{example}

\begin{example}\label{ex:hom-spaces}
We can construct projective spaces as homogeneous spaces as follows. 
Consider the standard action of $G=\Gl_{n}(R)$ on $R^{n}$ 
and let $H$ denote the subgroup of $G$ leaving the line $Ke_1$ invariant. 
Note that $H$ is a closed $K$--analytic subgroup of $G$. 
Then $G/H$ becomes a homogeneous $K$--analytic space with $R$--structure. 
It is straightforward to check that the canonical isomorphism of $G/H$ with 
the projective space $\proj^{n-1}(K)$ preserves the $R$--structure.
Similarly, by taking for $H$ the subgroup of $G$ leaving invariant 
the $K$--subspace spanned by $e_1,\ldots,e_r$, one obtains the 
Grassmann manifold $G/H\simeq \Gr(r,n)$ as a homogeneous 
$K$--analytic space with $R$--structure; see Example~\ref{ex:Grass}.   
\end{example}

\subsection{A nonarchimedean Poincar\'e formula}\label{se:poincare}

We have now developed all the necessary concepts and tools for proving integral geometry formulas 
in the setting of nonarchimedean homogeneous spaces. 

Let $G$ denote a compact $K$--analytic group with $R$--structure 
and $H\subseteq G$ be a closed $K$--analytic subgroup. 
Then $G/H$ is a homogeneous $K$--analytic space with $R$--structure; 
we denote by $p\colon G\to G/H$ the quotient map. 
We note that the natural action of $G$ on $G/H$ preserves the volume form $\Omega_G$. 
If $g\in G$ and $y\in G/H$, we denote by $g y$ the result of the action. 
Further, we denote by $y_0:=p(I)$ the projection of the identity element.
The multiplication with an element $h\in H$ fixes the point $y_0$; as a consequence, 
we have the derivative 
$h_*\colon T_{y_0}G/H\to T_{y_0}G/H$, 
so that we have an induced action of~$H$ on $T_{y_0}G/H$. 

  In Equation~\eqref{eq:def-sigma-many} we defined a quantity for capturing the relative position of 
  linear subspaces of a $K$--vector space with $R$--structure.
  We can extend this notion to describe the relative position of a collection of $K$-linear 
  subspaces $\{V_i\subseteq T_{g_i}^*G \ | \ 1 \leq i \leq m\}$ of the cotangent spaces 
  given by a collection of points 
  $g_1, \ldots, g_m\in G$, assuming $\sum_i\dim V_i \le \dim G$.
This is done by left-translating the~$g_i$ to the identity $I$: 
let $g_i^*\colon T_{g_i}^*G \to T_I^*G$ denote the dual map of the derivative 
$(g_i)_*\colon T_IG \to T_{g_i}G$ of the left-multiplication with $g_i$.
(By our assumption, $g_i^*$ is an isomorphism defined over $R$.) 
We now define 
\begin{equation}\label{eq:def-sigma}
 \sigma(V_1, \ldots, V_m) := \sigma\big( g_1^*V_1,\ldots,g_m^* V_m\big) .
\end{equation}

Let $Y$ be a $K$--analytic submanifold of $G/H$ and $y\in Y$.
We define the {\em conormal space} $N_y Y$ of $Y$ at $y\in Y$ as the orthogonal 
complement of the tangent space $T_yY$ in $T_yG/H$. 
Thus $N_y Y$ is the subspace of the cotangent space 
consisting of the linear forms on $T_y G/H$ vanishing on~$T_yY$. 

\begin{defi}\label{def:sigma-many}
For given $K$--analytic submanifolds $Y_1, \ldots, Y_m\subseteq G/H$, 
we define the {\em average scaling factor} as the function
$$
 \sigma_H:Y_1\times \cdots \times Y_m\to \R
$$
as follows. For $(y_1, \ldots, y_m)\in Y_1\times \cdots \times Y_m$ 
let $\xi_i\in G$ be such that $\xi_i y_i=y_0$ for all~$i$. We define 
$$
 \sigma_H (y_1, \ldots, y_m) :=
  \E_{(h_1,\ldots, h_m)\in H^m}\, \sigma(h_1^*\xi_1^* N_{y_1} Y_1, \ldots, h_m^* \xi_m^* N_{y_m} Y_m) ,
$$
where the expectation is taken over a uniform $(h_1, \ldots, h_m)\in H\times \ldots\times H$.
\end{defi}

Note that the dependence of $\sigma_H (y_1, \ldots, y_m)$ on the submanifolds $Y_i$ is only through 
their cotangent spaces $T_{y_i}^*Y_i$. Additionally, the average scaling factor is
independent of the choice of~$\xi_i$, since the transporter in $G$ of $y_i$ to $y_0$ is exactly~$H \xi_i$.

%The reader may compare this definition with~\cite[Def.~3.3]{howard:93}, which is just a special case.

\begin{example}\label{ex:sigma-proj}
We view the projective space $\proj^{n-1}=G/H$ as a 
homogeneous $K$--analytic space with $R$--structure 
as in Example~\ref{ex:hom-spaces}. 
Let us explicitly describe the average scaling factor $\s_H$ for the choice 
of $n-1$ projective hyperplanes $Y_i\simeq\proj^{n-2}$ in $\proj^{n-1}$. 
Then the function $\s_H$ is constant and takes the following value
\begin{equation}\label{eq:sigma-proj}
 \alpha_K(1,n-1):= \E \| v_1 \wedge\cdots\wedge v_{n-1}\| ,
\end{equation}
where the $v_i \in S(K^{n-1})$ are chosen independently and uniformly. 
In order to see this, note that the action of $H$ on $T_{e_1}\proj^{n-1}\simeq K^{n-1}$ 
factors through the standard action of $\Gl_{n-1}(R)$;  
see the proof of Proposition~\ref{pro:vols}. 
The uniform distribution on $H$ induces the uniform distribution on $\Gl_{n-1}(R)$.
Moreover, for uniformly chosen $A\in\Gl_{n-1}(R)$, the vector $Ae_1$ is uniformly distributed in $S(K^{n-1})$. 
We will determine  $\alpha_K(1,n-1)$ in 
Corollary~\ref{cor:alpha-proj}. 
\end{example}

The main result of this section is the following nonarchimedean generalization  of Poincar\'e's kinematic formula 
for homogeneous spaces, as stated in \cite{BL:19} and \cite[Thm.~3.8]{howard:93}. 
The proof will be provided in Section~\ref{se:proofGIGF}. 
Recall that we denote the volume of a compact 
$K$--analytic submanifold~$X$ of $G/H$ by $|X|$; see~\eqref{eq:defvol}. 

\begin{thm}\label{thm:GIGF}
Let $G$ denote a compact $K$--analytic group with $R$--structure  
and $H\subseteq G$ be a closed $K$--analytic subgroup. 
Let $Y_1, \ldots, Y_m$ be $K$--analytic submanifolds of $G/H$ 
such that $\sum_{i=1}^m\codim_{G/H} Y_i \leq \dim G/H$. 
Then the submanifolds $g_1Y_1, \ldots, g_mY_m$ intersect transversely, 
for almost all $(g_1, \ldots, g_m)\in G^{m}$, and
$$
 \E_{(g_1, \ldots, g_m)\in G^m} |g_1Y_1\cap \cdots\cap g_m Y_m|
   =\frac{1}{|G/H|^{m-1}} \int_{Y_1\times \cdots\times Y_m}\sigma_H\, \Omega_{Y_1\times\cdots \times Y_m} , 
$$
where the expectation is taken over a uniformly random $(g_1, \ldots, g_m)\in G\times \cdots\times G$. 
\end{thm}

Let us discuss some consequences before giving the proof in Subsection~\ref{se:proofGIGF}.
The kinematic formula above simplifies for certain submanifolds $Y$ of $G/H$.
The following extends~\cite[Def.~3.4]{BL:19}. 

\begin{defi}\label{def:TA}
Let $Y$ be a $K$--analytic submanifold of $G/H$. We say that 
{\em $G$ acts transitively on the tangent spaces} of $Y$ if for 
points $y_1,y_2$ of $Y$, there exists $g\in G$ such that 
$gy_1 = y_2$ and $g_*(T_{y_1}Y) = T_{y_2}Y$.  
\end{defi}

For example, in the case of a projective space $G/H=\proj^{n-1}(K)$ 
(see Example~\ref{ex:hom-spaces}),  the group $G=\Gl_{n}(R)$ acts transitively on 
the tangent spaces of any analytic submanifold of $\proj^{n-1}(K)$, 
since $\Gl_{n-1}(R)$ acts transitively on the Grassmannians $\Gr(r,n-1)$. 

\begin{cor}\label{cor:GIGF}
Under the assumptions of Theorem~\ref{thm:GIGF}, 
if moreover $G$ acts transitively on the tangent spaces to $Y_i$ for $i=1, \ldots, m$, 
then we have 
$$
  \E_{(g_1, \ldots, g_m)\in G^m}|g_1Y_1\cap \cdots\cap g_m Y_m|
     =\sigma_H(y_1, \ldots, y_m)\cdot |G/H| \cdot \prod_{i=1}^n\frac{|Y_i|}{|G/H|} ,
$$
where $(y_1, \ldots, y_m)$ is any point of $Y_1\times \cdots\times Y_m$.
\end{cor}
 
\begin{proof}
When $G$ acts transitively on the tangent spaces to each $Y_i$, 
then the function $\sigma_H\colon Y_1\times \cdots \times Y_m\to \R$ 
introduced in Definition~\ref{def:sigma-many} is constant. 
The assertion follows from Theorem~\ref{thm:GIGF}.
\end{proof}

We now obtain the integral geometry formula of~\cite{KL:19} as a consequence of Corollary~\ref{cor:GIGF}. 
(In~\cite{KL:19} the formula was only stated for the intersection of two factors and over $K=\Q_p$.)

\begin{cor}\label{cor:KL}Let $G=\mathrm{GL}_n(R)$, with the normalized Haar measure. Then:
\begin{enumerate}
\item Let $Y_1,\ldots,Y_m$ be  $K$--analytic submanifolds of $\proj^{n-1}$ 
of dimensions $d_1,\ldots,d_m$, respectively, such that $\sum_{i=1}^m (n-1- d_i) \leq n-1$. 
Then, for almost all $(g_1, \ldots, g_m)\in G^{m}$, 
the intersection $g_1Y_1 \cap\ldots\cap g_mY_m$ is transversal, 
of dimension $k := d_1 +\ldots + d_m - (m-1)(n -1)$, and 
\begin{equation*}
\E_{(g_1, \ldots, g_m)\in G^m} \frac{|g_1 Y_1 \cap\ldots\cap g_m Y_m|}{|\proj^{k}|} 
  = \prod_{i=1}^m \frac{|Y_i|}{|\proj^{d_i}|} .
\end{equation*}
\item Let $Y$ be a $d$--dimensional  $K$--analytic submanifolds of $\proj^{n-1}$
and $Y'\subseteq Y$ be a measurable subset. If we denote by $|Y'|$ the measure 
of $Y'$ (coming from the measure on $Y$), we have 
\begin{equation*}
\E_{H_1,\ldots,H_d} \#( Y' \cap H_1 \cap\ldots\cap H_{d})   = \frac{|Y'|}{|\proj^{d}|} ,
\end{equation*}
where the $H_i$ are independent uniformly random projective hyperplanes in $\proj^{n-1}$. 
\end{enumerate}
\end{cor}

\begin{proof}
Let us first prove point (1). The average factor $\sigma_H(y_1, \ldots, y_m)$ in Corollary~\ref{cor:GIGF} 
only depends on the cotangent spaces $T_{y_i}^* Y_i$ 
and the choice of $y_i$ in $Y_i$ is arbitrary. 
To determine~$\sigma_H$, we may thus 
choose for~$Y_i$ a projective subspace of dimension~$d_i$. Since,
for almost all $(g_1, \ldots, g_m)\in G^m$, we have 
$g_1\proj^{d_1} \cap\ldots\cap g_m\proj^{d_m} \simeq \proj^{k}$, 
we obtain from Corollary~\ref{cor:GIGF} 
\begin{equation}\label{eq:sigmaH-pro}
 |\proj^{k}| = \sigma_H \cdot |\proj^{n-1}| \cdot \prod_{i=1}^m \frac{|\proj^{d_i}|}{|\proj^{n-1}|} .
\end{equation}
When dividing  the left hand side 
of the equation of Corollary~\ref{cor:GIGF}  
by this, $\s_H$ cancels and the first assertion follows. 

We now move to point (2). We observe that this follows from the first when taking for $Y_2,\ldots,Y_m$ hyperplanes 
and $Y_1=Y$. The variation in terms of the measurable subset $Y'$ of $Y$ follows along 
the same line of the proof, based on Proposition~\ref{pro:BIF} and is omitted.
\end{proof}

\begin{cor}\label{cor:alpha-proj}
The average scaling factor defined in~\eqref{eq:sigma-proj} satisfies
$$
 \alpha_K(1,n-1) = \frac{1-\e}{1-\e^n} \Big(\frac{1-\e^n}{1-\e^{n-1}}\Big)^{n-1} .
$$
Furthermore, suppose that $A=[a_{ij}]\in R^{n\times n}$ is random, where the $a_{ij}\in R$ are chosen 
independently and uniformly at random. Then, for uniformly random~$x\in R^n$, 
$$
 \E |\det(A)| = \alpha_K(1,n) (\E \|x\|)^n = \frac{1-\e}{1-\e^{n+1}} .
$$
\end{cor}

\begin{proof}
We obtain from~\eqref{eq:sigmaH-pro} in the case $d_1=\ldots=d_{n-1}=1$ that 
$$
 1 = \alpha_K(1,n-1) \cdot |\proj^{n-1}| \cdot \Big(\frac{|\proj^{n-2}|}{|\proj^{n-1}|} \Big)^{n-1} .
$$
Substituting~\eqref{eq:vol-proj} yields the first assertion.

Next, let $a_1,\ldots,a_n\in R^n$ denote the columns of $A$ and write 
$a_i = \|a_i\|\cdot u_i$. According to Proposition~\ref{prop:pushfHaarII}, 
$\|a_i\|$ and $u_i$ are independent, and $u_i$ is uniformly distributed in $S(K^n)$.
Therefore
$$
 \E |\det [a_1,\ldots,a_n]| = \E\|a_1\| \cdots \E \|a_n\|\cdot \E |\det [u_1,\ldots,u_n]| 
  = (\E\|a_1\|)^n \alpha_K(1,n) .
$$
The second assertion follows with Proposition~\ref{prop:pushfHaarII}.
\end{proof}

We remark that the distribution of $|\det(A)|$ was determined by Evans~\cite{evans:02}.
The formula for $\E |\det(A)|$ was also obtained in \cite[Cor.~12]{el-manssour-lerario:20}.

\subsection{Proof of Theorem~\ref{thm:GIGF}}\label{se:proofGIGF}

It turns out that, once the analogue of Sard's lemma (Theorem~\ref{th:sard}) 
and the coarea formula (Theorem~\ref{th:coarea}) are settled,
the proofs in Howard~\cite[Thm.~3.8]{howard:93} and~\cite{BL:19} can be translated to 
the nonarchimedean setting. 
The major ingredient in Howard's proof is his basic integral formula~(2.7)
that we restate here in our setting. 

\begin{prop}\label{pro:BIF}
Let $G$ denote a compact $K$--analytic group with $R$--structure.
Further, let $M_1$ and $M_2$ be $K$--analytic submanifolds of $G$ such that 
$\dim(M_1) + \dim(M_2) \ge \dim(G)$. Then, for almost all $g\in G$, the submanifolds 
$M_1$ and $gM_2$ intersect transversally, and for any Borel measurable function 
$h\colon M_1\times M_2 \to [0,\infty]$, we have 
$$
 \int_G \int_{M_1\cap gM_2} h \circ\varphi_g(x) \cdot \Omega_{M_1\cap gM_1}(x)\, \Omega_G(g) = 
 \int_{M_1\times M_1} h\, \sigma(N_{x_1}M_1,N_{x_2}M_2) \, \Omega_{M_1\times M_2}(x_1,x_2) ,
$$ 
where the map $\varphi_g\colon M_1\cap gM_2 \to M_1\times M_2$ is given by 
$\varphi_g(x) := (x,g^{-1}x)$. 
\end{prop}

\begin{proof}(Sketch) 
We proceed exactly as for the proof of Howard~\cite[\S2.9--\S2.14]{howard:93}.
We note that, by a left-translation in the group $G$, for his computation of the 
Jacobian $Jf(\xi,\eta)$ in \S2.13, we may w.l.o.g. assume that $\xi=\eta=I$.
Now we replace Howard's argument by Lemma~\ref{le:auxy}. The remaining 
arguments are by straightforward adaptation and omitted.
\end{proof}

Based on Proposition~\ref{pro:BIF}, the proof of Theorem~\ref{thm:GIGF} 
now runs exactly as in~\cite[A5.2]{BL:19}.

\section{Applications to the geometry of Grassmannians and Schubert varieties} \label{se:grassmannian}

In this section we apply the previous results for studying the geometry of Grassmannians and Schubert varieties. 
In Section~\ref{sec:jointrandom}, we used the position vector
to give a parameterization of the Grassmannian in terms of the relative position of a space
with respect to a fixed one; 
in Section~\ref{sec:jointrandom}, we described the pushforward measure from the 
Grassmannian (endowed with the uniform measure) to the set of relative position vectors. 
We now use these results in Section~\ref{se:Vol-Schubert} to compute the volume of special Schubert varieties. 
In Section \ref{sec:PSC}, we shall outline the construction of a probabilistic version of nonarchimedean Schubert calculus, 
by studying the ``simplest Schubert problem" from a probabilistic point of view, as it was done in \cite{BL:19} for the real case.

We endow the Grassmannian $G(k,n)$ with the $R$--structure from Example \ref{ex:Grass}, 
which coincides with the $R$--structure as homogeneous space 
$G(k,n)=\mathrm{GL}_n(R)/(\mathrm{GL}_k(R)\times\mathrm{GL}_{n-k}(R))$;  
see~Example~\ref{ex:hom-spaces}.
Notice that, by construction, 
the group $\mathrm{GL}_n(R)$ acts on $G(k,n)$ by \emph{isometries}, 
i.e., by transformations that preserve the $R$--structure.

\subsection{Volume of special Schubert varieties}
\label{se:Vol-Schubert}

Let $k,\ell \le n$ and $m \le \min\{k,\ell\}$.
We fix $F\in G(\ell,n)$ and define the {\em special Schubert variety} 
and {\em special Schubert cell}, respectively, by 
\begin{align*}
\Omega_{k,\ell,m,n}(F) &:=\{E \in G(k,n) \mid \dim(E \cap F) \ge m \} , \\
\Omega_{k,\ell,m,n}^o(F) &:= \{E \in G(k,n) \mid \dim(E \cap F) =m \} .
\end{align*}

  It is well known that a Schubert variety is an algebraic variety defined over $\Q$, and
  thus it is an algebraic variety defined over any field containing $\Q$ such as $\R$, $\C$, or $K$.
  For further details about the defining equations, see~\cite[Section~4.1.3]{eisenbud-harris:16}
  or \cite{manivel:01}.
  It is no surprise that the locus of smooth $K$-rational points has the structure of a $K$--analytic
  manifold. The following proposition states precisely some details about this $K$--analytic manifold;
  we remark that the proof is essentially identical to the case over $\R$ or $\C$, and that the details
  are implicit in our discussion. 
We usually drop $F$ in the notation, as $F$ is consider fixed.

\begin{prop}
The Schubert variety $\Omega_{k,\ell,m,n}(F)$ is an algebraic subvariety of the Grassmannian $G(k,n)$, with 
smooth locus $\Omega_{k,\ell,m,n}^o(F)$, which is a 
$K$--analytic submanifold of $G(k,n)$ of codimension $mp$, where 
$p:=n+m-k-\ell$.  
\end{prop}

We note that $\Omega_{k,\ell, m,n}$ has codimension one iff 
$m=1$ and $p=1$, where the latter means $k+\ell=n$.  
Passing to the orthogonal complement provides the isometry
\begin{equation}\label{eq:goto-oc}
 G(k,n) \to G(n-k,n),\; E \mapsto E^\perp .
\end{equation}
It is straightforward to check that this isometry maps $\Omega_{k,\ell,m,n}(F)$ 
to $\Omega_{n-k,n-\ell,p,n}(F^\perp)$. 
Note that the parameters $m$ and $p$ are exchanged; 
the codimension of this Schubert variety is $pm$. 
Due to this observation on duality, we can assume that $k\le\ell$ 
without loss of generality. 

We now determine the volume of special Schubert varieties over $K$. 
It turns out to be very helpful 
to use a more symmetric notation for Grassmannians: 
let us define for $a,b\in\N$
$$
 \Gra(a,b) := G(a,a+b) = \{ E\subseteq K^{a+b} \mbox{ subspace }\mid \dim E = a\} ,
$$
so that $\Gra(b,a)\simeq\Gra(a,b)$ via the duality isomorphism~\eqref{eq:goto-oc}. 
In this notation, we can express the special Schubert variety as 
$$
 \Omega_{a,a_1;b,b_1} := \{ E \in \Gra(a+a_1,b+b_1) \mid \dim  (E\cap F) \ge a\} 
      = \Omega_{k,\ell,m,n},
$$
where $F\in\Gra(a+b_1, a_1 +b)$ is fixed, and 
\begin{equation}\label{eq:newnotation}
 a= m, \ a_1 = k-m, \ b= p = n+m-k-\ell, \ b_1 = n-k-p .
\end{equation}
Note that $\Omega_{a,a_1;b,b_1}$ has the codimension $ab$ in $\Gra(a+a_1,b+b_1)$. 
(In fact, this Schubert variety 
corresponds to the rectangular Young diagram $a\times b$, 
see~\cite{eisenbud-harris:16, manivel:01}.)
Duality can now be expressed in a more symmetric way: the map 
$$
 \Gra(a+a_1,b+b_1) \to \Gra(b+b_1,a+a_1), E\mapsto E^\perp
$$ 
sends $\Omega_{a,a_1;b,b_1}$ to the special Schubert variety 
$$
 \Omega_{b,b_1;a,a_1} = \{ E^\perp \in \Gra(b+b_1,a+a_1) \mid \dim  (E^\perp\cap F^\perp) \ge b\} ,
$$
where $F^\perp \in \Gra(a_1+b,a+b_1)$. 

We can now state the main result of this section.

\begin{thm}\label{th:main-vol-schubert} 
For any nonnegative integers $a,a_1,b,b_1$, we have 
\begin{equation*}
  \frac{|\Omega_{a,a_1;b,b_1}|}{|G(a+a_1,b+b_1)|} = \frac{
   |\Gra(a,b)| \cdot   |\Gra(a,b_1)| \cdot  |\Gra(a_1,b)|}{
   |\Gra(a+a_1,b)| \cdot  |\Gra(b+b_1,a)|}.
\end{equation*}
\end{thm}

For the codimension one case $m=p=1$ this specializes as follows. 

\begin{cor}\label{cor:dimvol-codim-1}
Consider the codimension one special Schubert variety 
$\Omega := \Omega_{1, k-1;1, n-1}$ of $G(k,n)$ 
consisting of the $k$--planes nontrivially meeting an $(n-k)$--plane. 
The volume of $\Omega$ is given by:
$$
\frac{|\Omega|}{|G(k,n)|} 
 = |\proj^1| \cdot \frac{|\proj^{k-1}|}{|\proj^{k}|} \cdot \frac{|\proj^{n-k-1}|}{|\proj^{n-k}|} . 
$$
\end{cor}

As we already observed, the formulas in Corollary~\ref{cor:dimvol-codim-1} 
are {\em identical} to the ones over $\R$ and over $\C$, which were 
obtained in~\cite{BL:19}.

\begin{remark}\label{re:vol-SpS-new}
By tracing our proof of Theorem~\ref{th:main-vol-schubert}, one can show 
that the same formula also holds in the setting over $\R$ and $\C$. 
Over $\R$ this is a new result 
(the reference \cite{BL:19} only dealt with the codimension one case).
\end{remark}

We turn now to the proof of Theorem~\ref{th:main-vol-schubert}. 
For this, we consider the intersection map 
\begin{equation}\label{eq:def-phi-map}
 \Psi\colon \Omega_{a,a_1;b,b_1}^o \to G(a,F),\, E\mapsto E\cap F.
\end{equation}
Note that all of the fibers of $\varphi$ are isomorphic to $\Gra(a_1,b+b_1)$. 
The next lemma describes the absolute Jacobian of $\Psi$. 
We postpone the proof to Section~\ref{se:lemma-Jac}.

\begin{prop}\label{pro:Jac-Psi-m}
For $E\in\Omega_{a,a_1;b,b_1}^o$ we have 
$$
  J(\Psi)(E) = \prod_{i=1}^{a_1} |r_i|^{-a}  ,
$$
where $r_1,\ldots,r_{a_1}\in R $ correspond to the position vector 
of $E$ and $F$. 
\end{prop}

\begin{proof}[Proof of Theorem~\ref{th:main-vol-schubert}] 
We abbreviate $\Omega:=\Omega_{a,a_1;b,b_1}$. 
The coarea formula (Theorem~\ref{th:coarea}) applied to the map $\Psi$ from~\eqref{eq:def-phi-map},
combined with isometric invariance, implies that for a fixed $L\in G(a,F)$ we have 
\begin{equation}\label{eq:Omega-produ}
 |\Omega| = |G(a,F)| \cdot \int_{\Psi^{-1}(L)} J(\Psi^{-1})\, d\Psi^{-1}(L) .
\end{equation}
Noting that $\Psi^{-1}(L) \simeq \cG :=\Gra(a_1,b+b_1)=G(k-m,n-m)$
and using Proposition~\ref{pro:Jac-Psi-m}, 
we can express the integral as 
$$
 \int_{E'\in \cG} \e^{m(x_1+\ldots +x_{k-m})} \, d\cG
$$
where $(x_1,\ldots,x_{a_1})$ denotes the position vector of $E'\in \cG$ 
with a fixed subspace $F'\in \Gra(b_1,a_1+b)$. 
We can express this integral as 
$$
 \int \e^{m(x_1+\cdots +x_{k-m})} \, d\cG =
  |\cG| \sum_{x\in\bN_o^{k-m}} \rho_{k-m,\ell-m,n-m}(x) \;\e^{m(x_1+\cdots +x_{k-m})} ,
$$ 
since $\rho_{k-m,\ell-m,n-m}$ is the pushforward distribution of the uniform 
distribution with respect to the map
$\chi\colon \cG \to \bN_o^{k-m}$, 
which sends $E'$ to its position vector with $F'$. 

{We now make essential use of Section~\ref{sec:jointrandom}:}
according to Equation~\eqref{eq:c-quot} and 
Corollary~\ref{cor:rho-inheritance}, 
this sum equals 
\begin{equation}\label{eq:theratio}
\begin{split}
 \frac{c_{k-m,\ell-m,n-m}}{c_{k-m,\ell-m,n}} &= 
 \frac{|\Gra(a_1,a+b+b_1)|}{|\Gra(a_1,b+b_1)|} \; \frac{\g_{n-\ell}}{\g_{n-\ell+m}} \;\frac{\g_{p +m}}{\g_{p}}  \\
 &= \frac{|\Gra(a_1,a+b+b_1)|}{|\Gra(a_1,b+b_1)|} \; \frac{\g_{a_1+b}}{\g_{a+a_1+b}} \;\frac{\g_{a+b}}{\g_{b}} . 
\end{split}
\end{equation}
Hence we obtain from \eqref{eq:Omega-produ}, 
$$
 \frac{|\Omega|}{|G(k,n)|} = 
\frac{|\Gra(a,b_1)| \;|\Gra(a_1,b+b_1)|}{|\Gra(a+a_1,b+b_1)|} \;
      \frac{c_{k-m,\ell-m,n-m}}{c_{k-m,\ell-m,n}} ,
$$
which after substituting in~\eqref{eq:theratio} leads to (note the cancelling) 
$$
 \frac{|\Omega|}{|G(k,n)|} = 
\frac{|\Gra(a,b_1)| \;|\Gra(a_1,a+b+b_1)|}{|\Gra(a+a_1,b+b_1)|} \;
 \frac{\g_{a_1+b}}{\g_{a+a_1+b}} \;\frac{\g_{a+b}}{\g_{b}} .
$$
Using Proposition~\ref{pro:vols}, 
we get 
$$
\frac{|\Omega|}{|G(k,n)|} = 
  \frac{\g_{a+b_1}}{\g_{a} \g_{b_1}} \; 
  \frac{\g_{a+a_1+b+b_1}}{\g_{a_1} \g_{a+b+b_1} }\; 
  \frac{\g_{a+a_1} \g_{b+b_1}}{\g_{a+a_1+b+b_1}} \; 
 \frac{\g_{a_1+b}}{\g_{a+a_1+b}} \; \frac{\g_{a+b}}{\g_{b}} .
$$ 
One term cancels. 
We expand by multiplying with $\frac{\g_a \g_b}{\g_a\g_b}$ 
and rearrange to arrive at 
$$
\frac{|\Omega|}{|G(k,n)|} =
 \frac{\g_{a+b}}{\g_{a} \g_{b}} \cdot 
 \frac{\g_{a+b_1}}{\g_{a} \g_{b_1}} \cdot 
 \frac{\g_{b +a_1}}{\g_b \g_{a_1}} \cdot 
 \frac{\g_{a+a_1} \g_b}{\g_{a+a_1+b}} \cdot 
 \frac{\g_{b+b_1} \g_a}{\g_{b+b_1+a}} . 
$$
This means that 
$$
\frac{|\Omega|}{|G(k,n)|} = 
  \frac{|\Gra(a,b)|\cdot |\Gra(a,b_1)| \cdot  |\Gra(b,a_1)| }{
   |\Gra(a+a_1,b)| \cdot  |\Gra(b+b_1,a)|} ,
$$
which completes the proof.
\end{proof}

\subsection{Proof of Proposition~\ref{pro:Jac-Psi-m}}  %(Absolute Jacobian of $\Psi$)
\label{se:lemma-Jac}

Let $m\le k \le \ell$ such that $k+\ell \le n$.
Suppose $E_0\in G(k,n)$, $F\in G(\ell,n)$ such that 
$\dim (E_0\cap F)=m$ and let $x\in\Nko$ be the 
position vector of the pair of subspaces $E_0,F$. 
Then $x_i=\infty$ for $i \le m$.
By Theorem~\ref{th:NF-pairs-subspaces}, 
there is an $R$-basis of~$K^n$, 
$$
 e_1,\ldots,e_m,e_{m+1},\ldots,e_k,f_{m+1},\ldots,f_k, g_1,\ldots, g_{n-2k+m} ,
$$
such that $E_0$ is spanned by $e_1,\ldots,e_k$, and $F$ is spanned by 
\begin{equation}\label{eq:generators-F}
  e_1,\ldots,  e_m,\; 
  e_{m+1} + r_{m+1} f_{m+1},\ldots, e_{k} + r_{k} f_{k}, \; g_1,\ldots,g_{\ell-k} ,
\end{equation}
where $r_i =\varpi^{x_i}$. We put $D:=\diag (r_{m+1},\ldots,r_{k})$. 
Without loss of generality, we may assume that this is the standard basis of $K^n$. 

Consider the Schubert cell 
$\Omega^o =\{ E\in G(k,n) \mid \dim (E\cap F)=m\}$. 
Lemma~\ref{le:Omega-chart} below 
describes its inverse image under the inverse chart
$\varphi\colon R^{(n-k)\times k} \to G(k,n)$, 
which maps $A$ to the column span of $\begin{bmatrix} I_k \\ A \end{bmatrix}$, 
compare Example~\ref{ex:Grass}. 

We will also use the following chart 
$\varphi' \colon U \to R^{(\ell -m) \times m}$ 
defined on an open subset $U\subseteq G(m,F)$ around $E_0\cap F$.
For given $L\in U$, let $v_1,\ldots,v_m$ be an $R$-basis of $L$. 
We expand $v_j$ in the $R$-basis $(e'_i)$ of $F$ defined in~\eqref{eq:generators-F},
\begin{equation}\label{eq:expand-v}
 v_j = \sum_{i=1}^\ell  c_{ij} e'_i ,\quad \mbox{$j=1,\ldots,m$} ,
\end{equation}
which defines the matrix $C=[c_{ij}] \in R^{\ell\times m}$.
Then we define 
\begin{equation}\label{eq:chart-varphi'}
 \varphi'(L) = \mathrm{column span}(C_2 C_1^{-1})
\end{equation}
where 
$C= \begin{bmatrix} C_1 \\ C_2 \end{bmatrix}$,
$C_2 \in R^{(\ell-m)\times m}$. 
(This is independent of the choice of the $R$-basis of $L$.
Note also that $|\det(C_1)|=1$ for $L$ sufficiently close to $E_0\cap F$.) 

The composition $\varphi' \circ \Psi \circ \varphi$ of the charts 
with the intersection map $\Psi$ from~\eqref{eq:def-phi-map}
is well defined on a neighborhood of~$0$.
Proposition~\ref{pro:Jac-Psi-m} is immediately deduced from
(assertion~(5) of) the following technical lemma, 
which analyzes the derivative of $\Psi$ at $0$.

\begin{lemma}\label{le:Omega-chart}
Write $A\in R^{(n-k)\times k}$ in the block decomposition
$$
 A= \begin{bmatrix}
A_{11} & A_{12}\\
A_{21} & A_{22} \\
A_{31} & A_{32}
\end{bmatrix} ,
$$
grouping the rows according to 
$n-k = (k-m) + (\ell -k) + (n+m-k-\ell)$ and the 
columns according to $k = m + (k-m)$. 
\begin{enumerate} 
\item We have $\varphi(A)\in\Omega^o$ iff the 
$(n-k)\times \ell$ matrix
$$
M_A:= 
\begin{bmatrix}
A_{11} & A_{12}-D & 0\\
A_{21} & A_{22}     & I_{\ell-k}  \\
A_{31} & A_{32}     & 0 
\end{bmatrix} 
$$
has rank $\ell -m$.  

\item The tangent space of $\varphi^{-1}(\Omega^o)$ at the origin consists of the matrices 
(using the same subdivisions as above) 
\begin{equation*}\label{eq:derivative-matrix}
\begin{bmatrix}
\dot{A}_{11} & \dot{A}_{12} \\
\dot{A}_{21} & \dot{A}_{22} \\ 
       0        & \dot{A}_{32} 
\end{bmatrix} ,
\end{equation*}
whose lower left $(n+m-k-\ell)\times m$ submatrix $\dot{A}_{31}$ vanishes. 
(This also shows that $\codim(\Omega^o) = (n+m-k-\ell)m$.) 

\item The derivative of $\varphi' \circ\Psi\circ\varphi$ at $0$ maps 
$$
R^{(n-k)\times k} \to R^{(\ell-m)\times m}, \; 
\begin{bmatrix}
\dot{A}_{11} & 0 \\
\dot{A}_{21} & 0\\ 
       0        & 0 
\end{bmatrix} 
\mapsto 
\begin{bmatrix}
\dot{A}_{11} D^{-1}\\
\dot{A}_{21} \\ 
\end{bmatrix} .
$$

\item The kernel of the derivative of $\Psi\circ\varphi$ at $0$ is given by 
$\dot{A}_{12}=0$, $\dot{A}_{22}=0$.
An $R$-complement of the kernel is given by 
$\dot{A}_{12}=0$, $\dot{A}_{22}=0$, $ \dot{A}_{32}=0$. 
 
\item The absolute Jacobian of $\Psi$ satisfies 
$$
  J(\Psi)(\varphi(A)) = \prod_{i=1}^{k-m} |r_i|^{-m} .
$$
\end{enumerate}
\end{lemma}

\begin{proof}
  \begin{enumerate}[itemindent=\dimexpr \labelwidth+\labelsep+\parindent \relax, leftmargin=0pt, itemsep=\parskip]
  \item[(1)]\hspace{-\labelsep}--(2)
    Let $S_A \in R^{n\times (k +\ell)}$ denote the matrix obtained from 
    $\begin{bmatrix} I_k \\ A \end{bmatrix}$ by appending to it as columns the generators of $F$ 
    from~\eqref{eq:generators-F}. Note that 
    $$
    \varphi(A) \in \Omega^o \Longleftrightarrow \dim(\varphi(A)+F) = k +\ell -m \Longleftrightarrow \rk(S_A) = k+\ell.
    $$ 
    Writing $S_A$ in block decomposition gives 
    \[
      S_A =
      \begin{bmatrix}
        I_m   & 0         & I_m           & 0 & 0 \\
        0      & I_{k-m} & 0 & I_{k-m} & 0 \\  
        A_{11} & A_{12}  & 0 & D        & 0   \\
        A_{21} & A_{22}  & 0 & 0        & I_{\ell-k}  \\
        A_{31} & A_{32}  & 0 & 0        & 0  
      \end{bmatrix} .
    \]
    The matrix $S_A$ has the same rank as (use column operations) 
    \[
      S'_A =
      \begin{bmatrix}
        0      & 0            & I_m           & 0 & 0 \\
        0      & 0            & 0 & I_{k-m} & 0 \\  
        A_{11} & A_{12}-D  & 0 & D        & 0   \\
        A_{21} & A_{22}     & 0 & 0        & I_{\ell-k}  \\
        A_{31} & A_{32}     & 0 & 0        & 0 
      \end{bmatrix} .
    \]
    Moreover, 
    $\rk(S'_A) = k + \rk(M_A)$, where $M_A$ 
    is defined in the statement of the lemma. 

    Hence we see that $\varphi(A) \in \Omega^o$ 
    iff the matrix $M_A$ has rank $\ell -m$,  
    which proves the first assertion. 
    The latter condition means that 
    \[
    \begin{bmatrix}
      A_{11} & A_{12} & 0\\
      A_{21} & A_{22} & 0\\
      A_{31} & A_{32} & 0  
    \end{bmatrix}
    + 
    \begin{bmatrix}
      0 &-D & 0\\
      0 &  0 & I_{\ell-k}\\
      0 & 0 &  0
    \end{bmatrix}
    \in  Z^0 ,
    \]
    where 
    \[
    Z^o := \{B\in K^{(n-k)\times \ell} \mid \rk(B) = \ell -m \} .
    \]
    Hence the tangent space of 
    $\varphi^{-1}(\Omega^o)$ at $0$ can be 
    identified with the tangent space of $Z^o$ at 
    \[
    \begin{bmatrix}
      0 &-D & 0\\
      0 &  0 & I_{\ell-k} \\
      0 & 0 &  0
    \end{bmatrix} .
    \]
    Applying Lemma~\ref{le:det-variety} below shows the 
    second assertion. 

    \setcounter{enumi}{2}
  \item
    Consider the curve $A(t) = t E_{1,1}$  (so $k>m$). 
    Then 
    $$
    E(t) := \varphi(A(t)) = \mathrm{span} \{e_1+ t f_{m+1}, e_2,\ldots,e_k \} .
    $$
    The vector
    $$
    v = (e_1 + t f_{m+1}) + \frac{t}{r_{m+1}} e_{m+1} = 
    e_1 + \frac{t}{r_{m+1}} \big(e_{m+1} + r_{m+1} f_{m+1} \big) 
    $$
    lies in $E(t)\cap F$. 
    It is straightforward to verify that 
    $$
    E(t) \cap F = \mathrm{span} \{v, e_2,\ldots,e_m \} .
    $$ 
    Expressing $v, e_2,\ldots,e_m$ in terms of our $R$-basis of $F$ 
    as in~\eqref{eq:expand-v} gives the matrix 
    $\begin{bmatrix} I_m \\ t /r_{m+1} E_{1,1}\\ 0\end{bmatrix}$ 
    of format $\ell\times m$. 
    This shows that $v, e_2,\ldots,e_m$ indeed form an $R$-basis of $E(t) \cap F$. 
    Moreover, for the chart $\varphi'$ introduced in \eqref{eq:chart-varphi'},  
    we obtain
    $$
    \varphi'(E(t) \cap F) = \frac{t}{r_{m+1}} E_{1,1}.
    $$ 
    Therefore, 
    $$
    D_0(\varphi' \circ\Psi\circ\varphi)(E_{1,1}) = \frac{1}{r_{m+1}} E_{1,1} . 
    $$
    In the same way we show that, for $1\le i \le k-m$ and $1\le j\le m$,  
    $$
    D_0(\varphi' \circ\Psi\circ\varphi)(E_{i,j}) = \frac{1}{r_{m+i}} E_{i,j} . 
    $$

    Consider now the curve $A(t) = t E_{k-m+1,1}$.
    Then 
    \begin{align*}
      E(t) :=& \ \varphi(A(t)) = \mathrm{span} \{e_1+ t g_{1}, e_2,\ldots,e_k \}
      \qquad \text{and} \\
      E(t) \cap F =& \ \mathrm{span} \{e_1+ t g_{1}, e_2,\ldots,e_m \} .
    \end{align*}
    For the chart $\varphi'$ introduced in \eqref{eq:chart-varphi'},  
    we obtain
    $$
    \varphi'(E(t) \cap F) = t E_{k-m+1,1} , 
    $$ 
    hence 
    $D_0(\varphi' \circ\Psi\circ\varphi)(E_{k-m+1,1}) =E_{k-m+1,1}$. 
    More generally, we see that for $1\le i \le \ell-m$ and $1\le j\le m$, 
    $$
    D_0(\varphi' \circ\Psi\circ\varphi)(E_{k-m+i,j}) = E_{i,j} . 
    $$
    This proves claim (3).
    
  \item
    If $t\mapsto A(t)$ is an analytic curve with 
    derivative $\frac{dA}{dt}(0)$ of the form~\eqref{eq:derivative-matrix} 
    with $\dot{A}_{11}=0$ and $\dot{A}_{12}=0$ , then (in first order)
    $e_1,\ldots,e_m$ lie in $\varphi(A)$. This implies that 
    the derivative of $\Psi\circ\varphi$ at $0$ in direction $\frac{dA}{dt}(0)$ 
    vanishes. Together with the statement in (3), this shows that the kernel of 
    $D_0(\Psi\circ\varphi)$ at~$0$ is as stated. 

  \item
    This follows from (3) and (4). \qedhere
  \end{enumerate}
\end{proof}

The proof of Lemma~\ref{le:Omega-chart} relied on a general fact 
on determinantal varieties that we discuss now. 
Consider the determinantal variety 
$Z := \{B\in K^{m\times n} \mid \rk(B) \le r \}$, 
where $r\le\min\{m,n\}$. It is well known that 
the singular locus 
$$
 Z^o := \{B\in K^{m\times n} \mid \rk(B) =r \}
$$ 
of $Z$ is a $K$--analytic submanifold of $K^{m\times n}$.
We need to compute its tangent spaces. 

\begin{lemma}\label{le:det-variety}
Suppose that $B_0=[b_{ij}]\in Z^0$ satisfies 
$b_{ii} \ne 0 $ for $i\le r$ and $b_{ij}=0$ otherwise. 
Then the tangent space of $Z^o$ at $B_0$ is given by 
$$
T_{B_0}Z^o  =\{ \dot{B} \in K^{(n-k)\times k} \mid  \dot{B}_{ij} = 0 \quad \forall i,j > r \}.
$$ 
In particular, $\codim Z^o = (m-r)(n-r)$. 
\end{lemma}

\begin{proof}
Consider an analytic curve $K\to Z^o, t\mapsto B(t) = B_0 + t\dot{B} +O(t^2)$. 
Then $f(B(t))$ vanishes identically for any minor $f$ of size $r+1$. 
Let $i,j>r$ and pick for $f$ the minor given by selecting the first $r$ rows and 
the $i$th row, and by selecting the first $r$ columns and the $j$th column. 
Expanding with respect to $t$ gives
$$
 f(B(t)) =  t \dot{B}_{ij} \prod_{\iota=1}^r b_{\iota \iota} + O(t^2) ,
$$
hence $\dot{B}_{ij}=0$. On the other hand, if $\dot{B}_{ij}=0$ fo all $i,j>r$, 
one easily verifies that $f(B(t)) = O(t^2)$ for every minor $f$ of size $r+1$.
\end{proof}

\subsection{Toward a probabilistic nonarchimedean Schubert Calculus}\label{sec:PSC}

We apply here the results of the previous sections to the study of Schubert Calculus, extending results of~\cite{BL:19} from 
the real to the nonarchimedean setting. 

We consider the following specific problem, which we call ``the simplest Schubert problem'':
\begin{quote}
\emph{What is the expected number of $k$--planes in $K^n$ nontrivially 
intersecting $k(n-k)$ random independent subspaces $L_1, \ldots, L_{k(n-k)}\subseteq K^n$ of dimension $n-k$?}
\end{quote}

\noindent
The problem can be formulated geometrically as follows. Let $N:=k(n-k)=\dim G(k,n)$ 
and denote by $\Omega(L)\simeq \Omega_{1, k-1;1, n-1}$ 
the Schubert variety of $k$--planes 
nontrivially intersecting 
a fixed subspace $L\simeq K^{n-k}$. Given $L_1, \ldots, L_N\in G(n-k, n)$, 
the set of $k$-planes intersecting all the $L_1, \ldots, L_N$ 
equals $\Omega(L_1)\cap \cdots \cap \Omega(L_N)$. 
Fix $L_0\in G(k,n)$ and denote $\Omega=\Omega(L_0)$; 
since the map $\mathrm{GL}_{n}(R)\to G(k,n),\, g\mapsto g^{-1}L_0$ 
induces the uniform distribution on $G(k,n)$, the expectation we are seeking equals
$$ 
  \mathbb{E}_{(g_1, \ldots, g_N)\in \mathrm{GL}_n(R)^N}\#\big(\Omega(g_1^{-1}L_0)\cap \cdots\cap \Omega(g_1^{-1}L_0)\big)
   =\mathbb{E}_{(g_1, \ldots, g_N)\in \mathrm{GL}_n(R)^N}\#\big(g_1\Omega\cap \cdots\cap g_N\Omega \big) .
$$
This problem can therefore be approached using Corollary \ref{cor:GIGF}, 
provided the hypotheses are verified, 
which is what we will discuss now.

Let us phrase 
a slightly more general result. Recall that the tangent spaces of $\Gr(k,n)$ have a product structure: we have 
$T_A \Gr(k,n) = A^* \otimes A$ for $A\in \Gr(k,n)$. As in \cite[Def.~3.6]{BL:19} and~\cite[\S4.3]{GKZ}, 
we call a $K$--analytic hypersurface $\cM\subseteq \Gr(k,n)$ 
{\em coisotropic} if, for all $A\in \cM$, its conormal spaces $N_A \cM$ are spanned by a rank one vectors. 
As in~\cite{BL:19}, one shows that $\Gl_n(K)$ acts transitively on the tangent spaces 
of a $K$--analytic coisotropic hypersurface $\cM\subseteq \Gr(k,n)$, compare 
Definition~\ref{def:TA}. 
(For a detailed analysis of the coisotropy notion over $\C$ we refer to \cite{kohn-mathews:21}.) 
If a hypersurface $\mathcal{M}$ is coisotropic, then $\mathrm{GL}_{n}(R)$ acts transitively on its tangent spaces. 

As for \cite[Lemma~4.1]{BL:19}, one shows 
that the regular locus of a special Schubert variety $\Omega_{1, k-1;1, n-1}$ of codimension one is coisotropic; 
in particular we have the following.

\begin{lemma}
The group $\mathrm{GL}_n(R)$ acts transitively on tangent spaces of a codimension one Schubert variety. \qed
\end{lemma}

We are interested in studying the intersection of $\dim G(k,n)$ many 
coisotropic $K$--analytic hypersurfaces of $G(k,n)$. For this, we 
introduce as in \cite[Def.~3.18]{BL:19} 
the following analogue of the real average scaling factor,  
generalizing \eqref{eq:sigma-proj}.

\begin{defi}\label{def:alpha}
For $k,m\ge 1$ we define the {\em average scaling factor} over $K$ as 
$$
 \alpha_K(k,m) := \E \| (u_1\ot v_1) \wedge \ldots \wedge (u_{km} \ot v_{km}) \|,
$$
where $u_j \in S(K^k)$ and $v_j \in S(K^{m})$, for $1\le j\le km$, 
are independently and uniformly chosen at random. 
\end{defi}

We have the following general result on intersecting coisotropic $K$--analytic 
hypersurfaces in $G(k,n)$, which in particular applies to the above Schubert problem.

\begin{cor} 
\label{thm:IGF}
Let $\Hy_1, \ldots, \Hy_{N}$ be coisotropic $K$--analytic hypersurfaces of $G(k,n)$, 
where $N :=k(n-k)$. 
Then we have:
$$
 \E_{(g_1,\ldots,g_N)\in \Gl_n(R)^N}
\#\left(g_1\Hy_1\cap\cdots\cap g_N \Hy_N \right)
 =  \alpha_K(k,n-k) \cdot |\Gr(k,n)| \cdot\prod_{i=1}^{N} \frac{|\Hy_i|}{|G(k,n)|}.
$$
In particular, the expected number $\eta_{k,n}$ of $k$-planes in $K^n$ 
nontrivially intersecting $k(n-k)$ 
random independent subspaces $L_1, \ldots, L_n\subseteq K^n$ of dimension $n-k$ equals:
\begin{equation}\label{eq:deltaex}
   \eta_{k,n}=\alpha_K(k,n-k) \cdot |G(k,n)|\cdot \left(|\proj^1| \cdot \frac{|\proj^{k-1}|}{|\proj^{k}|} 
            \cdot \frac{|\proj^{n-k-1}|}{|\proj^{n-k}|}\right)^N.
\end{equation}
\end{cor}

\begin{proof}
Since $\mathrm{GL}_n(R)$ acts transitively on the set of tangent spaces of a coisotropic hypersurface, 
we can apply Corollary~\ref{cor:GIGF} to obtain
\[
  \E_{(g_1,\ldots,g_N)\in \Gl_n(R)^N}
  \#\left(g_1\Hy_1\cap\cdots\cap g_N \Hy_N \right)
  =\sigma_H(y_1, \ldots, y_N)\cdot |G/H| \cdot \prod_{i=1}^N \frac{|\Hy_i|}{|G/H|}	     
\]
where $(y_1, \ldots, y_N)$ is any point of $\Hy_1 \times \cdots \times \Hy_N$.
We denote by $e_0$  the image of the identity point in $G/H$, and let $\xi_i \in G$ 
be elements such that $\xi_i y_i = e_0$.
Then
\[
  \s_H(y_1, \ldots, y_N) = \E_{(h_1, \ldots, h_N) \in H^N} \ \s(h_1^*\xi_1^* N_{y_1}\Hy_1, \ldots, h_N^*\xi_N^* N_{y_N} \Hy_N).	
\]
Because of the canonical identification of $T_A\Gr(k,n)$ and $\Hom(A, K^n/A)$, 
we see that this scaling factor, 
by its definition, equals $\alpha_K(k,n-k)$.
\end{proof}

It is not easy to analyze 
the average scaling factor over $\R$. 
However, in the nonarchimedean setting, we can determine the asymptotic behaviour of 
the average scaling factor 
as $q=|R/\frm|$ goes to infinity. 
For a related result we refer to~\cite{el-manssour-lerario:20}.

\begin{prop}\label{prop:alpha}
For fixed $k,m \geq 1$, we have $\alpha_K(k,m)\to 1$ as $q\to\infty$. 
\end{prop}

\begin{proof}
  We consider the polynomial function
  \[
    F\colon (R^{k} \times R^{m})^{km} \rightarrow  R,\  
    (u_1, v_1,u_2, v_2, \ldots, u_{km}, v_{km}) \mapsto (u_1 \otimes v_1) \wedge \ldots \wedge (u_{km} \otimes v_{km}) .
  \]
  By Definition~\ref{def:alpha}, we have 
  $\alpha_K(k,m) = \mathbb{E}_{X} |F(X)|$, 
  where $X=(X_1, \ldots, X_{km})$ with i.i.d.\ uniformly random variables $X_i$ on $S(K^k) \times S(K^m)$.
  Note that  
  $$
  \Prob(|F(X))| = 1) \leq |\alpha_K(k,m)| \le 1 ,
  $$
  It is therefore sufficient to show that $\Prob(|F(X)| = 1) \to 1$ as $q\to\infty$. 

  Suppose that $\xi = (\xi_1,\ldots,\xi_{km})$, where $\xi_i$ are i.i.d.\ uniformly random variables on $R^k \times R^m$. 
  We can think of $X_i$ as resulting from $\xi_i$ by conditioning on the event 
  $\xi_i \in S(K^k) \times S(K^m)$.
  When applying the canonical morphism $R\to R/\frm,\,\xi_i\mapsto \bar{\xi_i}$, the resulting 
  $\bar{\xi}_i$ are i.i.d.\ uniformly distributed in $R/\frm$.
  Note that 
  $|F(\xi)| <1$ iff $F(\bar{\xi}) = 0$. 
  Hence the Schwartz-Zippel lemma, applied over the residue field $R/\frm$, implies that 
  $$
  \Prob\big(|F(\xi)| < 1\big) \leq q^{-1}\deg(F) . 
  $$
  Moreover, setting $\e:=q^{-1}$ as usual, we have 
  for all~$i$, 
  \begin{align*}
    \Prob(\xi_i \not\in S(K^k) \times S(K^m)) 
    = \e^k +  \e^m - \e^{k+m} .
  \end{align*}
  Altogether, we obtain, 
  \begin{align*}
    \Prob(|F(X)| = 1) &\geq \Prob(|F(\xi)| = 1 \text{ and } \xi_1,\ldots,\xi_{km} \in S(K^k) \times S(K^m)) \\
                      &\geq \Prob(|F(\xi)| = 1)  - \sum_i \Prob(\xi_i \not\in S(K^k) \times S(K^m)) \\
                      &\geq 1 - \e \deg(F) - km\left(\e^k +  \e^m - \e^{k+m}\right).
  \end{align*}
  The right-hand side goes to $1$ as $\e\to 0$, which completes the proof.
\end{proof}

As a corollary we deduce the following asymptotic.

\begin{cor} \label{cor:k-plane-expectation-asyptotic}
  As $q\to \infty$, the expected number of $k$-planes in $K^n$ nontrivially intersecting $k(n-k)$ 
  random independent subspaces $L_1, \ldots, L_{k(n-k)} \subseteq K^n$ of dimension $n-k$ converges to~$1$.
\end{cor}

\begin{proof}
  This follows immediately from Equation~\eqref{eq:deltaex} and Proposition \ref{prop:alpha}, 
  since all the factors in \eqref{eq:deltaex} 
  converge to $1$ as $q$ goes to infinity.
\end{proof}

%%% SECTION
\section{Expected number of zeros of fewnomials}\label{se:fewnom}

As already mentioned at the end of the introduction, 
the study of the zeros of polynomial systems with few terms 
over an algebraically closed field is a thoroughly investigated topic~\cite{bernstein:75,BKK:76}. 
There is also a well developed theory concerned with counting the real zeros of such systems, 
for which we just refer to~\cite{khovanskii:91,sottile:11}. 
Much less is known in the nonarchimedian setting, we refer to~\cite{lenstra:97,poonen:98,rojas:04}. 

Our goal is to study the number of zeros of random fewnomial systems 
in a fixed nonarchimedean local field. This is strongly inspired by \cite{BETC:19,PB:23},
in which bounds for the expected number of real zeros of fewnomial systems were found. 
Evans~\cite{evans:06} was the first to investigate the expected number of zeros 
in the field of $p$-adics and his results were extended in \cite{KL:19}. 
However, these works 
only dealt with random systems of polynomial equations that are dense, i.e., 
polynomials of a certain degree in which all monomial terms appear.

Before moving to random fewnomials,
let us first 
summarize the known worst case bounds for univariate sparse polynomials.
We fix integers $a_1 < \ldots < a_t$ and consider the polynomial 
$$
 f(x) := c_{1} x^{a_1} + \ldots +  c_{t} x^{a_t} ,
$$
where $c_1,\ldots,c_t \in K$ and $c_1 c_t \ne 0$. 
If $K$ is algebraically closed, then $f$ has exactly $a_t -a_1$ zeros in~$K^\times$, 
counted with multiplicity. 
If $K$ is a real closed field, e.g., $K=\R$, then Descartes' rule states that 
$f$ has at most $t -1$ positive zeros in $\R$ (counted with multiplicity) 
and that this bound is sharp. 
It is remarkable that this upper bound only depends on the number of terms~$t$. 
Lenstra~\cite{lenstra:97} obtained a variant of Descartes' rule in the nonarchimedean setting 
in characteristic zero. 
(See Poonen~\cite{poonen:98} for results in positive characteristic.)
If $K$ is a field extension of $\Q_p$ with $R/\frm \simeq \F_{q}$
and ramification index~$e$, then 
the number of zeros of $f$ in~$K^\times$ 
(counted with multiplicity) is bounded from above by 
$O\Big(t^2 q e\log (et) \Big)$. In particular, 
Lenstra obtained the bound $O(t^2 p \log t)$ over $K=\Q_p$.  
The dependence on $t^2$ is necessary: 
the polynomial 
$f(x) = (x^t-a_1^t) \cdots (x^t-a_{t}^t)$ 
with distinct $a_1,\ldots,a_t \in\Q_p$
has at most $t+1$ terms and has $t^2$ zeros in $\Q_p$ 
provided $\Q_p$ has a $t$--th root of unity. 

Let us now move to the random setting. 
If $K=\R$ and the coefficients $c_i$ are i.i.d.\ standard real Gaussian coefficients, 
then the expected number of real zeros is $O(\sqrt{t})$~\cite{BETC:19,JindalPSZ20}.  
This bound was shown to be optimal in~\cite{JindalPSZ20}, but it seems difficult to 
characterize the supports for which the expected number of zeros is maximal.
The papers~\cite{BETC:19,PB:23} also contain an upper bound for random 
fewnomial systems, but the question of optimality remains open. 

Suppose now that $K$ is a finite extension of $\Q_p$ and that the coefficients $c_i$ are 
independent and uniformly chosen in $R$. 
Surprisingly, it turns out that the expected number of zeros of $f$ in~$K^\times$ is at most one, 
and equality is easy to characterize  (see Theorem~\ref{thm:n=1R}).

\subsection{Integral formula for the number of zeros of fewnomial systems}

As before, $K$ denotes a nonarchimedean local field of characteristic zero 
with valuation ring $R =\{x\in K \mid |x|\le 1 \}$. 
We fix a finite subset $A\subseteq\Z^n$ (the {\em support}) of size $t$
and study fewnomials of the form
\begin{equation}\label{eq:d-def}
 f_i(x) := \sum_{a\in A} c_{ia} x^a ,\quad 1\le i \le n, 
\end{equation}
where $c_{ia}\in K$ and $x^a := x_1^{a_1}\cdots x_n^{a_n}$. 
A zero $x\in (K^\times)^n$ of the system $f_1(x)=0,\ldots,f_n(x)=0$ 
is called {\em nondegenerate} if $\det (\partial_{x_i} f_j(x)) \ne 0$.
Without loss of generality, we assume 
that the affine span of $A$ equals $\R^n$; otherwise, every zero is degenerate.
In particular, we have $t\ge n+1$.

Let us assume now that the $c_{ia}$ are i.i.d.\ uniformly distributed in $R$ 
and consider the corresponding random system
$f_1(x)=0,\ldots,f_n(x)=0$.
We denote by $N_U(f)$ 
the number of nondegenerate zeros of this system lying in a subset $U\subseteq (K^\times)^n$.
We want to determine the expectation $\E(N_U(f))$ of this random variable.
To do so, we associate to $A$ the regular map
\begin{equation}\label{eq:psiA}
  \psi_A\colon (K^\times)^n \to \proj(K^A),\ x \mapsto [x^a]_{a\in A} .
\end{equation}
We also consider the affine version of this map 
\begin{equation}\label{eq:phiA}
  \varphi_A \colon (K^\times)^n \to (K^\times)^A,\ x \mapsto (x^a)_{a\in A} .
\end{equation}

\begin{lemma}\label{le:PropPsi} \
  \begin{enumerate}
  \item The image of $\varphi_A$ is an $n$--dimensional $K$--analytic submanifold of $K^A$ 
    if the linear span of $A$ equals $\R^n$.

  \item The image of $\psi_A$ is an $n$--dimensional $K$--analytic submanifold of $\proj(K^A)$ 
    if the affine span of $A$ equals $\R^n$.
  \end{enumerate}
\end{lemma}

\begin{proof}
  \begin{enumerate}[itemindent=\dimexpr \labelwidth+\labelsep+\parindent \relax, leftmargin=0pt, itemsep=\parskip]
    \item
      Let $A=\{a_1, \ldots, a_t\}$. Abusing notation, we also denote by $A\in\Z^{t\times n}$ 
      the matrix with the rows $a_1, \ldots, a_t$. 
      The matrix $A$ has rank $n\le t$ when we assume that the linear span of $A$ equals $\R^n$. 
      Using the Smith normal form, we may factor $A=SDT$, where $S\in\Gl_t(\Z)$, $T\in\Gl_n(\Z)$, 
      $D=\begin{pmatrix} D'\\0\end{pmatrix}$, and 
      $D'$ is the diagonal matrix with nonzero entries $d_1,\ldots,d_n\in\N$, 
      which are the invariant factors of $A$.
      It is straightforward to verify that 
      \begin{equation*}\label{eq:funct-phi}
        \varphi_A = \varphi_S \circ \varphi_D \circ \varphi_T .
      \end{equation*}
      In particular, $\varphi_S$ is bijective, $(\varphi_S)^{-1} = \varphi_{S^{-1}}$, 
      and similarly for $\varphi_T$. 
      Therefore, it suffices to prove the first assertion for 
      the matrix $D$. We have 
      \[
      \im \varphi_D
      = P_{d_1} \times\ldots\times P_{d_n} \times \{1\}^{t-n} , 
      \]
      where $P_d := \{x^d \ | \ x \in K^\times\}$. 
      Hensel's lemma implies that the set $P_d$ is an open subset of $K^\times$ 
      (here we use that $K$ has characteristic zero). 
      Therefore, $\im \varphi_D$ is indeed an $n$-dimensional $K$--analytic submanifold of $K^A$. 

    \item
      Note that $\im \psi_A$ is contained in all the affine charts $U_i=\{y_i \ne 0\}$, where 
      $y_1, \ldots, y_t$ denote the coordinates of $\proj(K^A)$. 
      E.g., viewing $\psi_A$ as a map to $U_t$, we have 
      $\psi_A = \varphi_{A'}$, where $A' := \{a_j - a_t \mid 1 \leq j \leq t-1\}$. 
      Note that $A'$ spans $\R^n$ linearly since $A$ spans $\R^n$ affinely by assumption.
      Applying the first part completes the proof. \qedhere
    \end{enumerate}
\end{proof}

The statement of Lemma~\ref{le:PropPsi} depends on our running assumption that $K$ has characteristic zero. 
For example, if $K=\F_2(\!(t)\!)$ and $A = \{2\} \subseteq \Z$, then the image of $\varphi_A$ is the set of squares. 
However, $\im \varphi_A$ is not a submanifold of $K$, compare Lemma~\ref{le:squares}.

\begin{remark}
Under the assumptions of Lemma~\ref{le:PropPsi}, the proof shows that the image of the map 
$\varphi_A\: (K^\times)^n \rightarrow K^A$ 
is an algebraic set iff all of the invariant factors of $A$ equal $1$. 
For instance, a nonzero $A \in \Z^{t\times 1}$ has exactly one invariant factor given by the gcd of its entries. 
\end{remark}

The following is a variant of~\cite[Theorem 2.2]{BETC:19}  in the nonarchimedean setting. 
The core idea of this approach (over $\R$) goes back to Edelman and Kostlan~\cite{edel_kost}. 

\begin{thm}\label{th:ENU}
  In the above setting, we have for a measurable subset $U\subseteq (K^\times)^n$ that
  \[
    \E(N_U(f))  = \frac{1}{|\proj^n|} \, \int_{U} J(\psi_A) \, d\mu_n,
  \]
  where $J(\psi_A)$ is defined with respect to the standard $R$--structure on $(K^\times)^n \subseteq K^n$ and 
  the quotient $R$--structure on $\proj(K^A)$ induced from $S(K^A) \rightarrow \proj(K^A)$. %}
\end{thm}

\begin{proof}
  Lemma~\ref{le:PropPsi} states that the image $Y$ of $\psi_A$ 
  is a $K$--analytic submanifold of $Z:=\proj(K^A)$ of dimension $n$.
  We may therefore apply Theorem~\ref{th:vol-image} to the map  
  $\psi_A\colon (K^\times)^n\to Z$ and the measurable subset $U$. 
  This implies 
  \begin{equation}\label{eq:iggy}
    \int_U J(\psi_A)\, d\mu_n = \int_{y\in Y} \# \big(U\cap\psi_A^{-1}(y)\big) \Omega_Y(y) = \sum_{d\in\N} d |Y_d| ,
  \end{equation}
  where we have put 
  \[
    Y_d := \{y\in Y \mid \# \big(U\cap\psi_A^{-1}(y)\big) =d \},
  \]
  and $|\cdot|$ refers to the measure coming from the $R$--analytic submanifold $Y$.

  We apply the integral geometry formula of Corollary~\ref{cor:KL}.(2) to the measurable subset $Y_d$ of~$Y$,
  giving 
  \begin{equation}\label{eq:above}
    \E_{H_1,\ldots,H_n} \#( Y_d \cap H_1 \cap\ldots H_n)   = \frac{|Y_d|}{|\proj^{n}|} ,
  \end{equation}
  where the $H_i$ are independent and uniformly random projective hyperplanes in~$\proj(K^A)$.

  A random choice of $(c_{ia})_{a\in A}$ in $R^A$ defines via~\eqref{eq:d-def} a random system of 
  fewnomial equations $f_1(x)=0,\ldots,f_n(x)=0$.  
  It also defines a system of random hyperplanes $H_1,\ldots,H_n$ given by the 
  linear equations $\sum_{a\in A} c_{ia} y_a= 0$ in the variables $(y_a)_{a\in A}$. 
  Clearly, the number of solutions in $U$ of the system $f_1(x)=0,\ldots,f_n(x)=0$ 
  is given by 
  \[
    N_U(f) = \sum_{d\in\N} d\, \# (Y_d \cap H_1 \cap\ldots H_n) 
  \]
  when the hyperplanes $H_1,\ldots,H_n$ intersect $Y_d$ transversally, 
  which happens almost surely by Theorem~\ref{th:sard} (Sard's lemma).
  Taking expectations and combining with \eqref{eq:iggy} and \eqref{eq:above} yields
  \[
    \E N_U(f) = \sum_{d\in\N} d\, \E \# (Y_d \cap H_1 \cap\ldots H_n) 
    = \frac{1}{|\proj^{n}|} \sum_{d\in\N} d \, |Y_d| 
    =  \frac{1}{|\proj^{n}|}\int_U J(\psi_A)\, d\mu_n ,
  \]
  which proves the assertion.
\end{proof}

\subsection{Zeros of random univariate fewnomials}\label{se:univariate}

We apply Theorem~\ref{th:ENU} in the special case $n=1$. 
We let $t \ge 2$, fix integers $a_1 < \ldots < a_t$,    
and study the zeros in $K^\times$ of the random fewnomial
$$
 f(x) := c_{1} x^{a_1} + \ldots +  c_{t} x^{a_t} ,
$$
where the $c_{j}$ are i.i.d.\ uniformly distributed in $R$. 
The map from~\eqref{eq:psiA} specializes to 
\begin{equation}\label{eq:mon-map}
  \psi \colon K^\times \to \proj^{t-1}, \ x \mapsto [x^{a_1} :\ldots:x^{a_t} ] 
  = [1 : x^{\tilde{a}_2} :\ldots:x^{\tilde{a}_t} ] ,
\end{equation}
where $\tilde{a}_i := a_i -a_1$. 
As we restrict attention to zeros $x\ne 0$,  
we could assume without loss of generality that 
all $a_j$ are nonnegative and $a_1 = 0$ via shifting the exponents. 
However, for the sake of clarity, we prefer to formulate the results without this assumption.

\begin{lemma}\label{le:n=1Jpsi}
For $x\in R \backslash \{0\}$ we have 
$$
 J(\psi)(x) 
  = \max_{j\ge 2} |\tilde{a}_j| \cdot |x|^{\tilde{a}_j -1}  \ \le\  |x|^{\tilde{a}_2 -1}
$$
with equality if $|\tilde{a}_2| = 1$. 
Here and in the following, $|\tilde{a}_j| := |\tilde{a}_j|_K$ 
denotes the absolute value of $\tilde{a}_j$ 
as an element of $K$, see \S\ref{se:na-fields}.
\end{lemma}

\begin{proof}
  We put $\varphi(x) := (1, x^{\tilde{a}_2},\ldots,x^{\tilde{a}_t})$
  and note that 
  $K\varphi(x) \cap R^n = R\varphi(x)$. 
  Therefore, the tangent space $T_{\psi(x)}\proj^{t-1} = K^t / K\varphi(x)$
  has the $R$--structure $R^t / R\varphi(x)$. 
  Note that 
  $R^t = R\varphi(x) \oplus (0\times R^{t-1})$.
  By the definition of the absolute Jacobian, we obtain 
  \begin{equation*}\label{eq:Jformel}
    J(\psi)(x) = \|\varphi'(x)\| = \max_{j\ge 2} | \varphi_j'(x) | = 
    \max_{j\ge 2} |\tilde{a}_j| \cdot |x|^{\tilde{a}_j-1} \le |x|^{\tilde{a}_2 -1} .
  \end{equation*}
  Here we used $|x| \le 1$ and $|\tilde{a}_j| \le 1$ since $\tilde{a}_j$ is an integer. 
  If $|\tilde{a}_2| = 1$, then the term $|\tilde{a}_2| \cdot |x|^{\tilde{a}_2}$
  is maximal for every $x$, and get equality above.
\end{proof}

We next show that 
$\E(N_{K^\times} (f)) \le 1$, which we find quite remarkable!
More precisely, we will derive upper bounds on the
expected number of zeros in the subsets 
$\RnO:= R\setminus\{0\}$ and $\mnO := \frm\setminus\{0\}$.
Typically, these upper bounds are sharp. 
Note that we have the disjoint decomposition 
$K^\times = \mnO \cup R^\times\cup (R\setminus\frm)$
into open subsets. 

\begin{thm}\label{thm:n=1R}
  Let $a_1 < a_2 < \ldots < a_t$ be integers and 
  $f(x) := c_{1} x^{a_1}  + c_2 x^{a_2} + \ldots +  c_{t} x^{a_t}$ 
  with i.i.d.\ uniformly distributed coefficients~$c_{j}$ in $R$.
  We have 
  \[
    \E(N_{R^\times}(f)) = \frac{1-\e}{1+\e}\, \cdot \max_{j\ge 2} |a_j -a_1|, \quad
    \E(N_{\mnO}(f)) \ \le\ \frac{1-\e}{1+\e} \Big( (1-\e^{a_2 - a_{1}})^{-1} -1 \Big) , 
  \]
  with equality if $|a_2-a_1| = 1$. Moreover, 
  we have  
  $$
  \E(N_{K^\times}(f)) \ \le\ \frac{1-\e}{1+\e} 
  \Big( (1-\e^{a_2-a_1})^{-1} + (1-\e^{a_t - a_{t-1}})^{-1} -1 \Big) , 
  $$
  with equality if $|a_2-a_1| = 1$ and $|a_t - a_{t-1}| = 1$. 
In particular, $\E(N_{K^\times}(f)) \le 1$. 
\end{thm}

\begin{proof}
Theorem~\ref{th:ENU} implies for any open subset $U\subseteq \RnO$ that 
\[
    \E(N_{U}(f))  =  \frac{1}{|\proj^1|} \, \int_{U} J(\psi) \, d\mu .
\] 
Taking $U=R^\times$ and using Lemma~\ref{le:n=1Jpsi}, this gives the  
stated formula for $\E(N_{R^\times}(f))$.  
Moreover, we obtain with Lemma~\ref{le:n=1Jpsi} that 
\[
    \E(N_{U}(f))  
    \ \leq \ 
    \frac{1}{|\proj^1|} \, \int_{U} |x|^{\tilde{a}_2 -1} \, d\mu(x) ,
\]
with equality holding if $|\tilde{a}_2|=1$. 
  Taking $U=\mnO$ and using Proposition~\ref{prop:pushfHaarII} gives
  \[
    \E(N_{\mnO}(f)) \leq \frac{1}{|\proj^1|} \int_{\frm} |x|^{\tilde{a}_2 -1} \, d\mu(x)
    = \frac{1}{1+\e} \left(\frac{1-\e}{1-\e^{\tilde{a}_2}}  - \mu(R^\times)\right), 
  \]
with equality holding if $|\tilde{a}_2|=1$. 
  
To bound $\E(N_{K^\times}(f))$, we consider 
the reverse polynomial $\tilde{f} := x^{a_t}f(x^{-1})$, which satisfies 
$N_{K\setminus R}(f) = N_{\mnO}(\tilde{f})$ since $x \mapsto x^{-1}$ 
is a bijection between $K \backslash R$ and $\mnO$. 
We obtain
\begin{equation*}\label{eq:EEE}
  \E(N_{K^\times}(f)) = \E(N_{R^\times}(f)) + \E(N_{\mnO}(f)) + \E(N_{\mnO}(\tilde{f})).
\end{equation*}
The bound on $\E(N_{K^\times}(f))$ follows now from the estimates for the summands.
\end{proof}

\begin{remark}
  We observe that in the above proof, we could have bounded directly 
  \[
    \E(N_{K^\times}(f)) \ \le\ \frac{1}{|\proj^1|} 
    \big(\mu(R^\times) + 2 \mu(\frm) \big) 
    = \frac{1}{|\proj^1|} \big( 1+ \mu(\frm) \big)  = 1 ,
  \]
  where we used $\mu(\frm) = \e$ for the right-hand equality. 
  Theorem~\ref{thm:n=1R} even allows us to conclude $\E(N_{K^\times}(f)) = 1$ 
  when 
  $a_t - a_{t-1} = a_2 - a_1 = 1$. 
\end{remark}

\begin{example}
Consider the random polynomial 
$f(x)= c_0 + c_2 x^2 + c_3 x^3 \in \Z_2[x]$ 
with independent coefficients $c_i\in\Z_2$. We claim that 
\[
  \E(N_{\mnO}(f)) = \frac{1-\e}{1+\e}\cdot \frac{\e^3}{1-\e^2} = \frac{1}{18},\quad 
  \E(N_{R^\times}(f)) = \frac{1-\e}{1+\e} = \frac13,\quad 
  \E(N_{K\setminus R}(f)) = \frac{\e}{1+\e} = \frac13 .
\]
Note that $J(\psi)$ is constant on the annulus $\frm^{k} \backslash \frm^{k+1}$, so 
\[
  \E(N_{\mnO}(f)) = \frac{1}{|\proj^1|} \, \int_{\frm} J(\psi) \, d\mu 
  = \frac{1}{1+\e} \sum_{k=1}^\infty (1-\e) \e^k J\psi(\e^k) .
\]
Since $a_2-a_1=2$, we get
\[
  J\psi(\e^k) = \min\{ |2|_2 \e^k, |3|_2 \e^{2k} \} = \min\{ \e^{k+1}, \e^{2k}\} = 
  \left\{
    \begin{array}{cl}
      1 & \mbox{ if $k=0$,} \\
      \e^{k+1} &\mbox{ if $k \ge 1$.}
    \end{array}
  \right.
\]
Hence 
\[
  \sum_{k=1}^\infty \e^k J\psi(\e^k) = \sum_{k=1}^\infty \e^{2k+1} =  \sum_{i=0}^\infty \e^{2(i+1)+1} = \frac{\e^3}{1-\e} ,
\] 
which shows the first equality. The other two follow similarly. 
\end{example}

As an immediate consequence of Theorem~\ref{thm:n=1R}
we obtain a result reminiscent of the situation for real random fewnomials, 
where it is known that the positive real zeros concentrate around~$1$; see \cite[Thm.~3]{JindalPSZ20}.

\begin{cor}\label{cor:n=1concentrate_roots}
  In the setting of Theorem~\ref{thm:n=1R}, 
  the probability that the random fewnomial $f$ has a zero in $\mnO$ goes to zero, 
  as $a_2-a_1\to\infty$. \qed
\end{cor}

Consequently, if $a_2-a_1$ is large, we expect the zeros of $f$ within $R\setminus\{0\}$ to lie in $R^\times$. 
If we assume that both gaps $a_2-a_1,a_t-a_{t-1}$ tend to infinity, then we expect the zeros of $f$ in $K^\times$ 
to lie in~$R^\times$.

\begin{example}
  Consider the random polynomial $f(x) := c_1 + c_2 x^d $ with independent uniformly distributed coefficients $c_1, c_2$ in $R$. 
  In order for $f$ to have a zero outside $R^\times$, it is necessary that $\val(c_1/c_2)$ is nonzero and 
  divisible by~$d$. It is easily seen that the probability for this event goes to zero as $d\to\infty$ 
  (use Proposition~\ref{prop:pushfHaarII}).
\end{example}

%%%
\subsection{A few results on zeros of random fewnomial systems}\label{se:fewnom-systems}

Let $A\subseteq\Z^n$ be a finite support.
By the {\em Newton polytope} $P$ of $A$ we understand its convex hull,
and by a {\em vertex} of $P$ we mean an extremal point of it. 
We call $P$ {\em rectangular} if it is of the form 
\[
  P =\{\ell_1,\ell_1+1,\ldots,r_1\} \times\ldots\times \{\ell_n,\ell_n+1,\ldots,r_n\} ,
\] 
for some integers $\ell_i < r_i$. 
We will call the vertex $\ell := (\ell_1, \ldots, \ell_n)$ 
the \emph{$\R_+^n$-minimal vertex of $P$}; 
note that for any $u\in \R_+^n$, $\ell$ minimizes $P\to\R,v\mapsto \langle u, v\rangle$.
We call $A$ \emph{gap-free rectangular at the vertex $\ell$} iff 
$A$ contains all the $n$  neighbours $\ell + e_1,\ldots,\ell+e_n$ at distance~$1$ of $\ell$, 
where $e_i$ stands for the $i$th standard basis vector. 
Similarly, we define $A$ to be gap-free rectangular at the vertex $v$ 
if all of its $n$ neighbours at distance~$1$ within $P \cap \Z^n$ are contained in~$A$. 
If this is the case at all vertices of $P$, then we call $A$ {\em gap-free rectangular}.

For instance, the support 
$$
 A=\{(0,1), (0,0), (1,0),(r-1,0),(r,0), (r,1),(r,r-1),(r,r),(r-1,r),(1,r),(0,r),(0,r-1)\}.
$$
of cardinality $12$ is gap-free rectangular (illustrated in Figure~\ref{figure:gapfree rectangular support}).

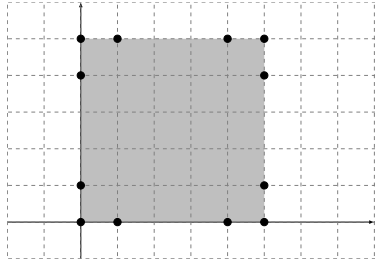
\begin{figure}[h]
  \centering
  \newcommand{\rforpic}{5}
  \resizebox{5cm}{!}{
    \begin{tikzpicture}{scale=0.05}
      \coordinate (Origin)   at (0,0);
      \coordinate (XAxisMin) at (-2,0);
      \coordinate (XAxisMax) at (\rforpic+3,0);
      \coordinate (YAxisMin) at (0,-1);
      \coordinate (YAxisMax) at (0,\rforpic+1);
      \draw [thin, black,-latex] (XAxisMin) -- (XAxisMax);% Draw x axis
      \draw [thin, black,-latex] (YAxisMin) -- (YAxisMax);% Draw y axis
      \clip (-2,-1) rectangle (8cm,6cm); % Clips the picture...
      \draw[style=help lines,dashed] (-14,-14) grid[step=1cm] (14,14); % Draws a grid in the new coordinates.

      \def\pointlist {(0,1), (0,0), (1,0), (\rforpic-1,0), (\rforpic,0), (\rforpic,1), (\rforpic,\rforpic-1),
        (\rforpic,\rforpic), (\rforpic-1,\rforpic), (1,\rforpic), (0,\rforpic), (0,\rforpic-1)}
      
      % Draw the support vertices and fill in the convex hull.
      \fill[gray, opacity=0.5] (0,1) \foreach \mypt in \pointlist {-- \mypt} -- cycle;
      \foreach \mypt in \pointlist {
        \node[draw, circle, inner sep=2pt, fill] at \mypt {};
      }
    \end{tikzpicture}
  }
  \caption{An example of a gap-free rectangular support set.}
  \label{figure:gapfree rectangular support}
\end{figure}

\begin{defi}
Consider a random system 
$f_1(x)=0,\ldots,f_n(x)=0$, where 
$$
 f_i(x) = \sum_{a\in A} c^{(i)}_{a} x_1^{a_1}\cdots x_n^{a_n},
$$
with i.i.d.\ uniformly distributed coefficients~$c^{(i)}_{a}$ in $R$. 
We denote by 
$N_U(A)$ the number of nondegenerate zeros of 
this system lying in the subset $U\subseteq(K^\times)^n$.
\end{defi}

The next result generalizes Theorem~\ref{thm:n=1R}.
Note the upper bound 
$(1-\e) (1+\e)^n(1-\e^{n+1})^{-1}$ in Proposition~\ref{pro:rectN}
equals~$1$ when $n=1$, but exceeds $1$ otherwise.

\begin{prop}\label{pro:rectN}
  Suppose the Newton polytope of $A\subseteq\Z^n$ is rectangular and 
  let $\ell$ denote its $\R_+^n$-minimal vertex. Then:
  \begin{enumerate}
  \item
    We have 
    \[
      \E(N_{(\RnO)^n}(A)) \le 1.
    \] 
    Equality holds iff 
    $A$ is gap-free rectangular at the vertex $\ell$. 

  \item
    We have 
    $$
    \E(N_{(K^\times)^n}(A)) \ \le\ \frac{(1-\e) (1+\e)^n}{1-\e^{n+1}} .
    $$
    Equality holds iff $A$ is gap-free rectangular. 
  \end{enumerate}
\end{prop}

The proof relies on the following auxiliary result, which extends Lemma~\ref{le:n=1Jpsi}. 
Recall the map~$\psi_A$ from~\eqref{eq:psiA}.

\begin{lemma}\label{le:Jle1}
  Let $A\subseteq\Z^n$ be a support with a rectangular Newton polytope. Then: 
  \begin{enumerate}
  \item $J(\psi_A)(x) \le 1$ for all $x\in (\RnO)^n$.

  \item If $A$ is gap-free rectangular at the $\R_+^n$-minimal vertex $\ell$, 
    then $J(\psi_A)(x) = 1$ for all $x\in (\RnO)^n$.

  \item  If $J(\psi_A)(x) = 1$ at some point $x\in (\mnO)^n$, 
    then $A$ is gap-free rectangular at $\ell$.

  \end{enumerate}
\end{lemma}

\begin{proof}
  \begin{enumerate}[itemindent=\dimexpr \labelwidth+\labelsep+\parindent \relax, leftmargin=0pt, itemsep=\parskip]
  \item
    By shifting, we may assume that $\ell_1=\ldots=\ell_n=0$ so that $0\in A$ and 
    $A\subseteq\N^n$. Let us write 
    $A=\{0,a_2,\ldots,a_t\}$, where $a_i\in\N^n$. 
    Consider the map 
    \begin{equation}\label{eq:defphi}
      \varphi\colon (K^\times)^n \to (K^\times)^{t-1},\, \varphi(x) =(x^{a_2},\ldots,x^{a_t}).
    \end{equation}
    We have $\psi_A(x)=[1:\varphi(x)]$.
    We can identify the tangent space of $\proj^t(K)$ at $\psi_A(x)$ 
    with $0\times K^{t-1}$, since the latter is an 
    $R$--complement of $\psi_A(x)= \operatorname{span}_K\{ (1,\varphi(x))\}$.
    By the definition of the absolute Jacobian~\eqref{eq:AJ-nl}, we have 
    $$
    J(\psi_A)(x) = N(D\psi_A)(x) = N(D\varphi)(x) .
    $$
    If we denote by $D\varphi(x)_I$ the square submatrix of  $D\varphi(x)$
    obtained by selecting the rows indexed by a subset $I\subseteq\{a_2,\ldots,a_t\}$ of cardinality $n$, 
    then we can write by~\eqref{eq:abs-det}: 
    $$
    N(D\varphi)(x) = \max_I |\det D\varphi(x)_I| . 
    $$ 
    The derivative of $\varphi$ is given by 
    $$
    D\varphi(x) = \mathrm{diag}(x^{a_2},\ldots,x^{a_t})\, M\, \mathrm{diag}(x_1^{-1},\ldots,x_n^{-1}) ,
    $$
    where $M\in\Z^{(t-1)\times n}$ denotes the matrix with the rows $a_2,\ldots,a_t$. 
    Therefore, writing $b:=\sum_{i\in I} a_i$, we obtain 
    $$
    \det D\varphi(x)_I = x^b \cdot \det M_I \cdot (x_1\cdots x_n)^{-1} 
    = x_1^{b_1-1} \cdots x_n^{b_n-1} \det M_I  .
    $$
    If $M_I$ is invertible, then $b_1,\ldots,b_n>0$, 
    otherwise the row span of $M$ would be contained in an $(n-1)$--dimensional coordinate subspace. 
    Therefore, $x_i^{b_i-1}\in R$ and hence $|x_i^{b_i-1}| \le 1$. 
    Moreover, $|\det M_I| \le 1$ since $\det M_I$ is an integer. 
    It follows that $|\det D\varphi(x)_I| \le 1$ and therefore 
    $ J(\psi_A)(x) =N(D\varphi)(x) \le 1$. 
    This proves the first assertion. 

  \item
    For the second assertion, assume that $A$ is gap-free rectangular at $0$.
    Then we choose $I$ to be the set of indices of
    the neighbours of $0$. Then $M_I$ is a permutation matrix and 
    $b=(1,\ldots,1)$. Hence indeed $ J(\psi_A)(x) = 1$ in this case. 

  \item
    For the third assertion, assume conversely that $ J(\psi_A)(x) = 1$ for some $x\in (\mnO)^n$. 
    Then there exists  $I$ such that $|\det D\varphi(x)_I| = 1$. 
    The corresponding $b$ must satisfy $b=(1,\ldots,1)$. Since $M_I$ is 
    invertible and has the row sum $b$, it must be a permutation matrix. 
    This shows that $A$ is gap-free rectangular at $0$ and 
    completes the proof. \qedhere
  \end{enumerate}
\end{proof}

\begin{proof}[Proof of Proposition~\ref{pro:rectN}]
  \begin{enumerate}[itemindent=\dimexpr \labelwidth+\labelsep+\parindent \relax, leftmargin=0pt, itemsep=\parskip]

    \item
      From Lemma~\ref{le:Jle1} we conclude that 
      \begin{equation}\label{eq:JB}
        \int_{(\mnO)^r\times (\RnO)^{n-r}} J(\psi_A)\, d\mu \ \le\ 
        \mu(\frm^r\times R^{n-r}) = \mu(\frm^r) = \e^r ,
      \end{equation}
      with equality holding if $A$ is gap-free rectangular at~$\ell$. 
      In the case $r=0$ we get
      \begin{equation*}
        \int_{(\RnO)^n} J(\psi_A)\, d\mu \ \le\ 1 .
      \end{equation*}
      and equality holds if $A$ is gap-free rectangular at~$\ell$.
      Conversely, if equality holds above, then 
      $J(\psi_A)(x)=1$ for almost all $x\in (\RnO)^n$.
      Hence $J(\psi_A)(x)=1$ for some $x\in (\mnO)^n$.
      Lemma~\ref{le:Jle1} implies that $A$ is gap-free rectangular at~$\ell$.
      This proves the first assertion.

    \item
      For the second assertion, Theorem~\ref{th:ENU} and \eqref{eq:JB} imply that 
      \begin{equation}\label{eq:ENr}
        \E(N_{(\mnO)^r \times (\RnO)^{n-r}}(A)) \ \le\ \frac{\e^r}{|\proj^n|} .
      \end{equation}
      We now proceed by a symmetry argument 
      based on the observation that 
      the involution $x\mapsto x^{-1}$ yields a bijection between 
      $K \backslash R$ and $\mnO$.
      For a subset $I \subseteq \{1, \ldots, n\}$, we define 
      \[
        U_I := \left\{(x_1, \ldots, x_n) \in (K^\times)^n \ \Big| \ 
          \begin{array}{ll}
            x_i \in K \backslash R &\text{for } i \in I ,\\
            x_i \in R \backslash \{0\} &\text{for } i \notin I
          \end{array}
        \right\}.
      \]
      Clearly the $U_I$ are mutually disjoint and cover $(K^\times)^n$.
      We claim that for all $I$ 
      \begin{equation}\label{eq:partition volume}
        \E(N_{U_I}(A)) \le \frac{\e^{\#I}}{|\proj^n|} .
      \end{equation}
      In order to show this, we may assume $I=\{1,\ldots,r\}$ by symmetry, so 
      $U_I = (K\setminus R)^r \times (\RnO)^{n-r}$.
      Let $\tau\: \Z^n \rightarrow \Z^n$ be the involution defined by 
      $\tau(a)_i = -a_i$ if $i\le r$ and $\tau(a)_i = a_i$ otherwise. We note that 
      $\tau(A)$ is gap-free rectangular exactly when $A$ is. 
      Using that $x\mapsto x^{-1}$ yields a bijection between 
      $K \backslash R$ and $\mnO$, we see that 
      \[
        \E(N_{U_I} (A)) = \E(N_{(\mnO)^r \times (\RnO)^{n-r}}(\tau(A))) .
      \]
      Thus~\eqref{eq:partition volume}  follows from \eqref{eq:ENr}.

      Since the $U_I$ form a partition of $(K^\times)^n$, we obtain 
      \begin{align*}
        \E(N_{(K^\times)^n}(A)) \ = \sum_{I \subseteq \{1, \ldots, n\}} \E(N_{U_I}(A)) 
        \ \le\ 
        \frac{1}{|\proj^n|}  \sum_{r=0}^n {n\choose r} \e^r 
        = \frac{1}{|\proj^n|} (1+\e)^n .
      \end{align*}
      By \eqref{eq:vol-proj} we have $|\proj^n|= \frac{1-\e^{n+1}}{1-\e}$, and so
      the upper bound of the second statement follows.

      Moreover, we get equality if $A$ is gap-free rectangular.
      Conversely, assume we have equality. Then there must be 
      equality in \eqref{eq:ENr}, which implies that $A$ is gap-free rectangular
      at $\ell=0$. By symmetry,
      it follows that $A$ is gap-free rectangular 
      at every vertex of $P$. \qedhere
    \end{enumerate}
\end{proof}

Besides the case where the Newton polytope has a rectangular shape, 
we can also give sharp estimates for certain simplices
that arise naturally when studying homogeneous polynomials.
The next results extends insights from~\cite{KL:19}  
obtained for dense homogeneous random systems. 
The proof is analogous to the one of Proposition~\ref{pro:rectN} 
and so is omitted.

\begin{prop}
  Consider the simplex 
  $P=\{a\in\R^n \mid a_i \ge 0, a_1 +\cdots + a_n \le d\}$. 
  Let $A$ be a subset of $P$ containing all the vertices $v$ of $P$ 
  and such that for each vertex $v$, its $n$ closest 
  neighbors in $P\cap \Z^n$ are also contained in $A$.
  Then we have $\E(N_{(K^\times)^n}(A)) = 1$. 
\end{prop}

Concretely, the condition of the $n$ closest neighbors in $P\cap \Z^n$ 
being contained in $A$ 
is equivalent to the statement that
$A$ contains the neighbors 
$(1,0,\ldots,0),\ldots,(0,\ldots,0,1)$ of the vertex $0$, the 
neighbors 
$(d-1,1,0,\ldots,0),\ldots,
(d-1,0,\ldots,0,1)$ of the vertex $(d,0,\ldots,0)$, 
and so on.

\subsection{Example: quadratic polynomial}
We study here the special case of a random quadratic polynomial
$f=c_1 + c_2 x^2$. Theorem~\ref{thm:n=1R} tells us that in characteristic zero 
\begin{equation}\label{eq:quadeq}
 \E(N_{R^\times}(f)) \ \le\  \frac{1-\e}{1+\e}, \quad 
 \E(N_{\mnO}(f)) \ \le\ \frac{\e^2}{1+\e^2} ,
\end{equation}
with equality when $|2|_K=1$. 
The goal of this subsection is to 
directly verify these formulas. 
This is not only instructive, but also reveals that 
the formulas are also valid in positive 
characteristic, even in characteristic two!

We need the following auxiliary result. 

\begin{lemma}\label{le:squares}
  Let $R/\frm\simeq\F_q$ with $q=p^m$ and denote by 
  $N$ the index of $(R^\times)^2$ in $R^\times$. 
  \begin{enumerate}
  \item
    If $p>2$, then $N=2$.

  \item
    If $p=2$ and $\chara R = 0$, then $N=2q$.

  \item
    If $p=2$ and $\chara R >0$, then $N=\infty$. 
  \end{enumerate}
\end{lemma}

\begin{proof}
  \cite[Satz~7]{lorenzII} states that 
  \[
    N= [R^\times : (R^\times)^s] =  \|s\|_K \cdot \#\{ x\in K \mid x^s =1\},
  \]
  which for $s=2$ implies that $N = \|2\|_K\cdot 2$.
  If $p>2$, then $\|2\|_K =1$, hence $N=2$.
  If $p=2$ and $\chara R = 0$, then $\|2\|_K =q^{-1}$, hence $N=2q$.
  Finally, if  $p=2$ and $\chara R >0$, then $|2|_K = 0$, hence $N=\infty$. 
  This proves the lemma.  
\end{proof}

According to Proposition~\ref{prop:pushfHaarII}, we can assume that 
$c_1 = \varpi^k u$, $c_2 = \pi^\ell  v$ with independent $k,f,u,v$, where 
$e,f\in\Z$ have the discrete distribution $\rho_1$ 
(defined in Proposition~\ref{prop:pushfHaarII})  
and $u,v\in R^\times$ are uniform 
with respect to the Haar measure $\nu$ of $R^\times$. 
Observe that $w:=-uv^{-1}$ is uniformly distributed 
in $R^\times$ with respect to $\nu$. 

We want to count the solutions of the equation 
$$
 x^2 = -c_1 c_2^{-1} = \varpi^{k-\ell} w .
$$
A solution exists in $R^\times$ iff $k=\ell$, and $w$ is a square in $R^\times$. 
Moreover, a solution exists in $\frm$ iff $k> \ell$, $k-\ell$ is even, and $w=-uv^{-1}$ 
is a square in $R^\times$. 
Furthermore, the solutions come in pairs if 
the characteristic~$p$ of the residue field of $R$ satisfies $p>2$. 
We calculate with Proposition~\ref{prop:pushfHaarII} 
\begin{align*}
 \Prob(k=\ell) = \sum_{i\ge 0} \rho_1(i)^2 
    = (1-\e)^{2} \sum_{i\ge 0} \e^{2i} 
    = \frac{(1-\e)^2}{1-\e^2}     
    = \frac{1-\e}{1+\e} ,
\end{align*}
\begin{align*}
 \Prob(k > \ell, \mbox{$k-\ell$ even}) = \sum_{f\ge 0} \sum_{k >0} \rho_1(f)\rho_1(f+2k) 
    = (1-\e)^{2} \sum_{\ell\ge 0} \e^{2\ell}\sum_{i> 0} e^{2i} 
   = \frac{\e^2}{(1+\e)^2} .
\end{align*}
We note that 
$$
 \Prob(\mbox{$w$ is square in $R^\times$}) = N^{-1} ,
$$
where $N$ denotes the index of the subgroup $(R^\times)^2$ of squares in $R^\times$.
Lemma~\ref{le:squares} above determines~$N$ in all possible cases.

Assume first that $p>2$. 
Then $N=2$ by Lemma~\ref{le:squares} and we obtain
\begin{align*}
 \E(N_{R^\times}(f)) &= \Prob(k= \ell) \cdot 
  \Prob(\mbox{$w$ is square in $R^\times$}) \cdot 2 
  = \frac{1-\e}{1+\e} , \\
 \E(N_{\mnO}(f)) &= \Prob(k>\ell, \mbox{$k-\ell$ even}) \cdot 
  \Prob(\mbox{$w$ is square in $R^\times$}) \cdot 2 
   = \frac{\e^2}{(1+\e)^2} 
\end{align*}
as stated in \eqref{eq:quadeq} above. 
(Note $|2|_K=1$ since $2\not\in\frm$.)  

Suppose now $p=2$. 
If $\chara R =0$, then $N=2q$ by Lemma~\ref{le:squares} 
and we obtain as above 
\begin{equation}\label{eq:caseq=2}
 \E(N_{R^\times}(f)) =  |2|_K \cdot \frac{(1-\e)}{1+\e},\quad 
 \E(N_{\mnO}(f)) = |2|_K \cdot \frac{\e^2}{(1+\e)^2} .
\end{equation}
This is consistent with Theorem~\ref{thm:n=1R} 
since we have $|2|_K< 1$ in this case. 
More precisely, when entering the proof of this theorem, 
we first note that by Lemma~\ref{le:n=1Jpsi},
we have 
$J(\psi)(x) = |2|_K |x|$. 
Proposition~\ref{prop:pushfHaarII} gives 
$$
\int_{\frm} |x| \, d\mu(x) = 
  (1- \e) (1-\e^{2})^{-1}  - \mu(R^\times) 
  = \frac{\e^2}{1+ \e} .
$$
Thus, using 
$|\proj^1| = \frac{1-\e^{2}}{1-\e}$, 
we obtain as for Theorem~\ref{thm:n=1R} that 
\[
  \E(N_{\mnO}(f))  = 
  \frac{1}{|\proj^1|} \, \int_{\frm} J(\psi) \, d\mu = 
  \frac{|2|_K}{|\proj^1|} \, \int_{\frm} |x| \, d\mu(x) 
  = |2|_K \cdot \frac{\e^2}{(1+\e)^2} , 
\]
which is the same result as in~\eqref{eq:caseq=2}. 
Similarly, one checks the formula for $\E(N_{R^\times}(f))$.

We finally turn to the case $p=2$ and $\chara R >0$
(in which case $\chara R = 2$).
Lemma~\ref{le:squares} gives $N=\infty$ and hence 
$\E(N_{\RnO}(f))  = 0$ in this case. 
We note that this is consistent with the proof of Theorem~\ref{thm:n=1R}:
we have 
$|2|_K = |0|_K =0$.
Hence, $J(\psi)(x)$ vanishes identically. 

Summarizing, we have directly verified the Equations~\eqref{eq:quadeq} 
in any characteristic.

\bibliographystyle{plain}
%\bibliography{lit.bib}
\bibliography{lit}
\section{Appendix}

\subsection{Proof of Theorem~\ref{th:sard} (Sard's lemma)}\label{sec:proofsard}

%The following special case of Fubini's theorem will be crucial for the proof  of Theorem~\ref{th:sard}: 
%Let $E\subseteq K^m\times K^n$ be measurable and suppose that its sections
%$E\cap (\{x\}\times K^n)$ have measure zero, for almost all $x\in K^m$. Then $E$ has 
%measure zero; e.g., see \cite[\S 3.6, Thm.~A]{halmos:50}.  
First we observe that, using the fact that our $K$--analytic manifolds have a countable basis, 
we can reduce to the case of open subsets $X=U\subseteq K^n$ and $Y=V\subseteq K^m$.

The proof proceeds by induction on $n$. 
If $n=0$, then $U$ is countable. If $m=0$, then $C(\varphi)$ is empty and the statement is trivially true.
Otherwise, $m>0$, in which case $C(\varphi)=U$, $\varphi(U)$ is countable, and hence of 
measure zero in $V$.  

We assume now $n>0$ for the induction step. 
If $m=0$, then $C(\varphi)=\varnothing$ and we are done. So we can assume $m>0$. 

For $k\geq 1$, we denote by  $C_k$ the subset of $U$ consisting of the points where all the derivatives of $\varphi$ 
up to order~$k$ vanish. We also set $C_0:=C(\varphi)$. 
Clearly, each $C_k$ is closed and we have the sequence of inclusions 
$C_0\supseteq C_1\supseteq C_2\supseteq \ldots$. 
The intersection
$\cap_{k\geq 1}C_k$ consists of the points~$x$ where {\em all} the derivatives of $\varphi$ vanish.
Since $\varphi$ is $K$--analytic, and we assume that $\chara K=0$, it follows that, if $x\in \cap_{k\geq 1} C_k$, then
$\varphi$ is {\em locally constant} on a neighbourhood of~$x$. 
Therefore, if we define $U'$ to be the open subset of $U$ consisting of the points~$p$ for which there is an
open neighbourhood on which $\varphi$ is not constant, we can write 
\begin{equation}\label{eq:dec}
 U' \cap C_0 = \bigcup_{k\geq 0}U'\cap (C_k\setminus C_{k+1}) .
\end{equation} 
We note that $C_0\setminus U'$ can be written as a union of open subsets on which $\varphi$ is constant. 
Since $K^n$ has a countable basis, we can assume that this is a countable union. Therefore, the image 
$\varphi(C_0\setminus U')$ is countable and hence of measure zero in $K^m$. Since
$$
\varphi(C_0) = \varphi(C_0\setminus U') \cup \bigcup_{k\geq 0}\varphi(U'\cap (C_k\setminus C_{k+1})).
$$ 
it suffices to prove that all the sets $\varphi(C_k\setminus C_{k+1})$ have measure zero.

We first treat the case $k=0$, which is the crucial step. 
Again, since $K^n$ has a countable basis, it is sufficient to show that around each point $p\in C_0\setminus C_1$ 
there is an open set $U_p$ such that $\varphi(U_p\cap C_0)$ has measure zero. 
Since $p\notin C_1$, 
there exists some partial derivative of $\varphi$ that does not vanish at $p$. 
By relabelling the coordinates and the components $\varphi_i$ of $\varphi$, we 
can assume that 
$\frac{\partial \varphi_1}{\partial x_1}(p)\neq 0$. 
Consider now the $K$--analytic map $h:U\to K^n$ defined by
\begin{equation}
  h(x)=(\varphi_1(x), x_2, \ldots, x_n) .
\end{equation}
The derivative $D_p h$ is invertible. By the Inverse Function Theorem~\cite[Prop.~6.4]{schneiderp:11}, 
there exist open neighbourhoods $U_p$ and $\tilde{U}_p$ of $p$ and $h(p)$, 
respectively, 
such that $h\colon U_p\to \widetilde{U}_p$ is a $K$--bianalytic isomorphism. 
We define the $K$--analytic map $\widetilde{\varphi}:=\varphi\circ h^{-1}\colon \widetilde{U}_p\to K^m$, which  
has the same critical values as $\varphi$ restricted to $U_p$. 
Moreover $\widetilde\varphi$ preserves the first coordinate. 
Therefore, $\widetilde{\varphi}$~induces $K$--analytic maps
$\widetilde{\varphi}_t\colon U_{t,p} \to K^{m-1}$  
such that $\widetilde\varphi(t,x_2,\ldots,x_m)= (t,\widetilde{\varphi}_t(x_2,\ldots,x_m))$,
where $U_{t,p} \subseteq K^{n-1}$ denotes the section of $U_p$ over~$t$.
The Jacobian matrix of $\widetilde{\varphi}$ has the following form 
\begin{equation}
 D \widetilde\varphi=\left(\begin{array}{c|ccc}
   1 & 0 & \cdots & 0 \\\hline a_2 &  &  &  \\\vdots &  & D\widetilde{\varphi}_t &  \\a_n &  &  & 
  \end{array}\right) .
\end{equation}
Therefore $\rk(D\widetilde \varphi)=1+\rk(D\widetilde \varphi_t)$ and hence 
$$
 C(\widetilde\varphi) = \bigcup_t \{t\}\times C(\widetilde \varphi_t) .
$$
We obtain 
$$
 \widetilde\varphi (C(\widetilde\varphi)) = 
  \bigcup_t \{t\}\times \widetilde\varphi_t(C(\widetilde \varphi_t)).  
$$
By induction hypothesis, $\widetilde\varphi_t(C(\widetilde \varphi_t))$ has measure zero in $K^{m-1}$.
Fubini's theorem implies that $\widetilde\varphi (C(\widetilde\varphi))$ has measure zero as well;
e.g., see \cite[\S 3.6, Thm.~A]{halmos:50}.  
This shows that $\varphi(C_0\setminus C_{1})$ has measure zero.

Let us prove now that $\varphi(C_k\setminus C_{k+1})$ has measure zero 
in the case $k\geq 1$.  Let $p\in C_k\setminus C_{k+1}$. 
As before, it is sufficient to show there exists an open neighbourhood $U_p$ 
such that $\varphi(U_p\cap C_k)$ has measure zero.
There is a $(k+1)$--th derivative of of some component $\varphi$, say $\varphi_1$, 
that does not vanish at $p$.  
Thus we can find a $k$--th derivative of $\varphi_1$, say $\psi$, which vanishes on $C_k$ (by definition of $C_k)$ 
and such that $\frac{\partial \psi}{\partial x_1}(p)\neq 0$. Consider the map $h:U\to K^n$ defined by
\begin{equation}
 h(x_1, \ldots, x_n)=(\psi(x), x_2, \ldots, x_n).
\end{equation}
As above, the inverse function theorem guarantees there are open neighbourhoods $U_p$ and $\widetilde{U}_p$ 
of $p$ of $h(p)$, respectively, 
such that $h:U_p\to \widetilde{U}_p$ is a $K$--analytic isomorphism. 
By construction, this isomorphism carries $C_k\cap U_p$ into $\{0\}\times K^{n-1}$ (since $\psi$ vanishes on $C_k$). 
Therefore the map $\widetilde{\varphi}:=\varphi\circ h^{-1}$ has all its critical points of type $C_k$ 
on the hyperplane $\{0\}\times K^{n-1}$. Let us denote by  
$U_{0,p}\subseteq K^{n-1}$ the section of $U_p$ over $0$ obtained 
by intersecting with this hyperplane. 
We apply the induction hypothesis to the restriction 
$\widetilde{\varphi}_0\colon U_{0,p} \to K^m$ 
of $\widetilde{\varphi}$.
Accordingly, 
the set of critical values of $\widetilde{\varphi}_0$ has measure zero in $K^m$.
But each critical point of $\widetilde{\varphi}$ of type $C_k$ 
is a critical point of $\widetilde{\varphi}_0$. 
Therefore, 
$\widetilde{\varphi}_0(U_p \cap C_k)$
has measure zero as well. 
This concludes the proof.

\subsection{A proof of the coarea formula}\label{sec:proofcoarea}
We proceed by a technical result that will be needed in the proof of the coarea formula.

\begin{lemma}\label{pro:section_complement}
Suppose $A_0\in K^{m\times n}$ has rank $m\le n$  
and let $L$ be an $R$-complement of $\ker A_0$.  
There is an open neighborhood $\mathcal{O}$ of $A_0$ in $K^{m\times n}$ 
and there are rational functions $f_i\colon\mathcal{O} \to R^n$ such that
$(f_1(A),\ldots,f_{n-m}(A))$ is an $R$--basis of $(\ker A)_R$,  
and $L$ is an $R$-complement of $\ker A$, for all $A\in\mathcal{O}$. 
\end{lemma}

\begin{proof} 
By scaling we may assume without loss of generality that $A_0$ has entries in $R$. 
Using the Smith normal form, we may assume that 
$A_0=[D \ 0]$ 
and $L=K^m\times \{0\}^{n-m}$,
where $D$ is an $m\times m$ diagonal matrix with nonzero entries in~$R$
and $0$ stands for the zero matrix of format $m\times (n-m)$. 
Notice that, since $A_0\in R^{m\times n}$, a small neighborhood of $A_0$ will still be made of matrices with integral entries.
We write a matrix $A\in R^{m\times n}$ in the form
$A=[B \ C]$, 
with $B\in R^{m\times m}$ and $C\in R^{m\times (n-m)}$. 
If $B$ is invertible, we have  
\begin{equation}\label{eq:kern}
  \ker A = \{ (-B^{-1}Cy, y) \mid y \in K^{n-m} \}.
\end{equation}
Put $\delta:= |\det(D)|>0$ and 
consider the  open neighborhood of $A_0=[D \ 0]$ consisting of 
the matrices $A=[B \ C] \in R^{m\times n}$ satisfying 
$\|C\| < \delta$ and $|\det(B) - \det(D)| < \delta$. 
For those $B$, we have 
$|\det(B)| = \delta $ by the ultrametric property, hence 
$$
 \|B^{-1}\| = |\det(B)^{-1}| \cdot \|\mathrm{adj}(B)\| \le |\det(B)^{-1}| = \delta^{-1} ,
$$
using that the adjugate $\mathrm{adj}(B)$ of $B$ has entries in $R$. 
Let $c_j$ denote the $j$th column of $C$ and $e_j\in K^{n-m}$ the $j$th standard basis vector.  
From $\|c_j\| \le \|C\| < \delta$, we get 
$\|B^{-1}c_j\| \le \|B^{-1}\| \|c_j\| <1$, hence $B^{-1}c_j \in R^m$. 
From this we conclude that  
$$
 (f_1(A),\ldots,f_{n-m}(A)) :=((-B^{-1} c_1,e_1), \ldots,  (-B^{-1}c_{n-m},e_{n-m}))
$$
is an $R$--basis of $(\ker A)_R$. 
Clearly, the basis vectors depend rationally on $B,C$. 
Moreover, $L=K^m\times \{0\}^{n-m}$ is an $R$-complement of $\ker A$.  
\end{proof}

\begin{proof}[Proof of Theorem~\ref{th:coarea}]
The proof runs as in Howard~\cite[Appendix]{howard:93}. 
We first note that $J(\varphi)(x)=0$ for $x\in C(\varphi)$. 
Hence the domain $X$ of integration on the left-hand side can be replaced 
by the open subset $X\setminus C(\varphi)$. Note that the assertion 
trivially holds if $X = C(\varphi)$: in fact, in this case, the image of $X$ has measure zero and 
for all the points $y\notin \varphi(X)$ the integrand on the right hand side is zero.

We can therefore w.l.o.g.\ focus on the case where $C(\varphi)=\varnothing$.
Using a partition of unity, 
we can reduce to the local statement, where 
$X\subseteq R^n$ and $Y\subseteq R^m$ are open subsets, whose 
$R$--structures are induced from the standard $R$--structures. 
(In fact, we do not even need to employ partitions of unity 
since every covering of~$X$ by open subsets can be refined to a 
{\em disjoint} covering (see \cite[Lemma~1.4]{schneiderp:11}.)

Let $x_0\in X$, and consider the Jacobian $A_0 := D_{x_0}\varphi \in K^{m\times n}$.
According to Lemma~\ref{pro:section_complement}, 
there is an $R$--complement $L$ of $\ker A_0$ and 
there are rational functions $f_i(x)$ of~$x$ defined 
in an neighborhood of $x_0$, such that 
$L$ is an $R$--complement of $\ker D_{x}\varphi$, and such that 
$(f_{m+1}(x),\ldots,f_{n}(x))$ form an $R$--basis of $\ker  D_{x}\varphi$, 
for all~$x$ in this neighborhood. By shrinking, we may 
assume that this neighborhood equals $X$. 
We fix an $R$--basis 
$f_1,\ldots,f_m$ of $L_R$. 
Further, let 
$e_1,\ldots,e_m$ denote the standard basis of $K^m$. 
Then, for $x\in X$, we have the $R$--bases
$$ 
(f_1,\ldots,f_m) \mbox{ of  $L_R$}, \quad
(f_{m+1}(x),\ldots,f_{n}(x)) \mbox{ of  $(\ker D_x\varphi)_R$},\quad 
(e_{1},\ldots,e_{m}) \mbox{ of $(T_{\varphi(x)}Y)_R= R^m$} . 
$$ 
Joining the first two bases yields an $R$--basis of~$(T_xX)_R=R^n$. 
We denote the corresponding dual bases of 
$L_R^*$, $(\ker D_x\varphi)_R^*$, and $(T_y^*Y)_R$ by 
$(\s_1,\ldots,\s_{m})$, $(\s_{n-m+1}(x),\ldots,\s_n(x))$, and 
$(\tau_{1},\ldots,\tau_{m})$, respectively.
Then we define the analytic forms $\omega_1,\omega_2$ on $X$ 
and $\eta$ on $Y$ by 
$$
 \omega_1 := \s_1\wedge \ldots\wedge \s_{m},\quad 
 \omega_2(x) := \s_{m+1}(x)\wedge \ldots\wedge \s_{n}(x) , \quad 
  \eta := \tau_{1}\wedge \ldots\wedge \tau_{m}. 
$$
By the definition of the volume form \eqref{eq:defOmega} we have,
$$
 \Omega_X = | \omega_1\wedge \omega_2| , \quad 
 \Omega_Y = |\eta| , \quad 
 \Omega_{\varphi^{-1}(y)} = \left\vert \ \omega_{2} {|}_{\varphi^{-1}(y)} \ \right\vert , 
$$
where for the third statement we used that 
$T_x \varphi^{-1}(y) = \ker D_x\varphi$ if $y=\varphi(x)$.  

For $1\le j \le m$ we write 
$D_x\varphi (f_j) =\sum_{i=1}^m a_{ij} e_i$
with $a_{ij}\in K$, with the matrix $A=[a_{ij}]$.
(Note that $D_x\varphi (f_j(x)) =0$ for $j>m$.) 
We have $\varphi^* \tau_i = \tau_i \circ D_x\varphi = \sum_{j=1}^m a_{ij} \sigma_j$.
Hence
$$
 \varphi^* (\tau_1 \wedge\ldots\wedge \tau_m) =  
  \varphi^* (\tau_1) \wedge\ldots\wedge \varphi^*(\tau_m) = 
  \sum_{j_1,\ldots,j_m} a_{1 j_1}\ldots a_{m j_m} \sigma_{j_1}\wedge\ldots,\wedge\sigma_{j_m} = 
  \det(A)\, \omega_1 .
$$ 
Moreover, we have $J(\varphi)(x) =|\det(A)|$ for $x\in X$ by the definition of the absolute Jacobian.
Therefore, 
$$
 J(\varphi)\, \Omega_X = | \varphi^* \, \eta  \wedge \omega_2 | .
$$
So it suffices to prove that 
\begin{equation}\label{eq:Esi}
  \int_X h\, |\omega_2 \wedge \varphi^* \eta | = 
  \int_{y\in Y} \Big(\int_{\varphi^{-1}(y)} h\, |\omega_2| \Big) \eta (y) ,
\end{equation}
but this is guaranteed by Lemma~\ref{le:FInt}. 
\end{proof}

%%%

\end{document}